\documentclass{amsart}
\usepackage{mathrsfs}
\usepackage{amssymb}
\usepackage{amsmath}
\usepackage{amsthm}
\usepackage{commath}
\usepackage{mathtools}
\usepackage{amsfonts}
\usepackage{tikz}
\usepackage[all]{xypic}
\usepackage{xcolor}
\usepackage{blindtext}
\usepackage{tikz-cd}
\usetikzlibrary{matrix}
\usepackage{stmaryrd}
\usepackage[utf8]{inputenc}
\usepackage[russian,english]{babel}
\usepackage[normalem]{ulem}

\newcommand{\PP}{\mathbb{P}}

\newcommand{\QQ}{\mathbb{Q}}

\newcommand {\Gm}{\mathbb{G}_m}

\newcommand{\RR}{\mathbb{R}}
\newcommand{\CC}{\mathbb{C}}

\newcommand{\Spec}{\text{Spec}\hspace{0.05 cm}}
\newcommand{\ZZ}{\mathbb{Z}}

\newcommand{\NSr}{NS^{1}_{\orb,\RR}}
\newcommand{\NSc}{NS^1_{\orb,\CC}}

\newcommand{\cJ}{\mathcal J}

\newcommand{\XX}{\mathcal{X}}
\newcommand{\cX}{\XX}
\newcommand{\doots}{,\dots ,}
\newcommand{\OO}{\mathcal O}
\newcommand{\cO}{\OO}

\newcommand{\Fv}{F_v}

\newcommand{\Ov}{\mathcal O _v}

\newcommand{\vMF}{v\in M_F}
\newcommand{\vMFz}{v\in M_F ^0}

\newcommand{\vMFi}{v\in M_F ^\infty}

\newcommand{\YY}{\mathcal Y}
\newcommand{\cY}{\YY}

\newcommand{\sss}{\mathbf{s}}
\newcommand{\xxx}{\mathbf{x}}

\newcommand{\muu}{\boldsymbol{\mu}}

\newcommand{\TTT}{\mathscr T}
\newcommand{\ttt}{\mathbf{t}}

\newcommand{\AAA}{\mathbb A}
\newcommand{\AAF}{\mathbb A_F}

\newcommand{\yyy}{\mathbf{y}}
\newcommand{\zzz}{\mathbf{z}}

\newcommand{\mmm}{\mathbf{m}}

\newcommand{\wH}{\widehat H}

\newcommand{\TT}{\mathcal T}
\newcommand{\cT}{\TT}
\newcommand{\nnn}{\mathbf n}

\newcommand{\pijx}{\pi_0^*(\mathcal J_0\mathcal X)}
\newcommand{\Spijx}{\Sigma(1)\cup\pijx}

\newcommand{\bt}{\mathbf t}
\newcommand{\id}{\mathrm{id}}
\newcommand{\fp}{\mathfrak p}
\newcommand{\fa}{\mathfrak a}
\newcommand{\pr}{\mathrm{pr}}
\newcommand{\bSigma}{\boldsymbol{\Sigma}}

\DeclareMathOperator{\Tam}{Tam}

\DeclareMathOperator{\age}{age}

\DeclareMathOperator{\Res}{Res}
\DeclareMathOperator{\Pic}{Pic}

\DeclareMathOperator{\Coker}{Coker}
\DeclareMathOperator{\Ker}{Ker}
\DeclareMathOperator{\Hom}{Hom}
\DeclareMathOperator{\ulHom}{\underline{Hom}}

\DeclareMathOperator{\Aut}{Aut}
\DeclareMathOperator{\rk}{rk}

\DeclareMathOperator{\tor}{tor}

\DeclareMathOperator{\rig}{rig}

\DeclareMathOperator{\Boxx}{Box}
\DeclareMathOperator{\vvv}{\mathbf{v}}

\DeclareMathOperator{\cont}{cont}

\DeclareMathOperator{\orb}{orb}
\DeclareMathOperator{\ord}{ord}
\DeclareMathOperator{\proba}{prob}

\numberwithin{equation}{subsection}
\newtheorem{mydef}[equation]{Definition}
\newtheorem{lem}[equation]{Lemma}
\newtheorem{thm}[equation]{Theorem}
\newtheorem{prop}[equation]{Proposition}

\newtheorem{rem}[equation]{Remark}

\newtheorem{cor}[equation]{Corollary}

\newtheorem{conj}[equation]{Conjecture}

  \DeclareFontFamily{U}{wncy}{}
    \DeclareFontShape{U}{wncy}{m}{n}{<->wncyr10}{}
    \DeclareSymbolFont{mcy}{U}{wncy}{m}{n}
    \DeclareMathSymbol{\Sh}{\mathord}{mcy}{"58}

\title{The Manin conjecture for toric stacks}

\author[Ratko Darda]{Ratko Darda}
\address{}
\email{ratko.darda@gmail.com}
\author[Takehiko Yasuda]{Takehiko Yasuda}
\address{Department of Mathematics\\ Graduate School of Sciences\\Osaka University}
\email{yasuda.takehiko.sci@osaka-u.ac.jp}

\begin{document}
\begin{abstract}
Split toric stacks over a number field~$F$ are natural generalization of split toric varieties over~$F$. Notable examples are weighted projective stacks. In our previous work, we defined heights on Deligne--Mumford stacks using so-called {\it raised line bundles} and made predictions on asymptotic formulas of the number of rational points of bounded height. In this paper,  we prove that the number of rational points of any split toric stack of bounded height with respect to the {\it anti-canonical} raised line bundle satisfies one of  our predictions, the Manin conjecture for Deligne--Mumford stacks. 
\end{abstract}

\maketitle

\section{Introduction}
\subsection{Manin conjecture}
One of fundamental questions of Diophantine geometry is to count rational points of algebraic varieties. Such a question is the Manin conjecture:
\begin{conj}
Let~$X$ be a Fano variety over a number field~$F$. Suppose that the set~$X(F)$ of its rational points is Zariski dense in~$X$. Let~$H$ be a height on~$X$ defined by an adelic metric on the anticanonical line bundle. For a constant~$C(H)>0$ explicited by Peyre~\cite{Peyre}, there exists a thin subset $Z\subset X(F)$ such that $$\#\{x\in X(F)-Z|\hspace{0,1cm} H(x)\leq B\}\sim_{B\to\infty}CB\log(B)^{\rho(X)-1},$$
where~$\rho(X)$ is the Picard number of~$X$.
\end{conj}
The known cases of the conjecture include: projective spaces~\cite{Schanuel}, smooth complete intersections if the number of variables is sufficiently large~\cite{Birch}, certain Châtelet surfaces~\cite{Chatelet}, equivariant compactifications of vector groups~\cite{Vgps}, etc.  It has been verified for {\it toric varieties} by Batyrev and Tschinkel~\cite{Toric} by methods of harmonic analysis.  A survey on the conjecture can be found in \cite{zbMATH05657087}. 

\subsection{Manin conjecture for stacks} Counting problems of objects possessing non-trivial automorphism groups are not in the scope of the problem of counting rational points of varieties, but are instead concerned by rational points of {\it algebraic stacks}. A well-known example is the Malle conjecture \cite{Malle}. Recently, there have been advances to extend the Manin conjecture to the algebraic stacks:
\begin{itemize}
\item The first named author in his PhD thesis~\cite{darda:tel-03682761} develops and proves a version of the Manin conjecture for {\it weighted projective stacks} \cite[Theorem 2.5.7.1]{darda:tel-03682761} with respect to so-called {\it quasi-toric heights}. The stacks are the quotients of punctured affine space~$\AAA^n-\{0\}$, for some $n\geq 1$ by the multiplicative group scheme~$\Gm$ for a weighted action $$t\cdot (x_1\doots x_n)=(t^{a_1}x_1\doots t^{a_n}x_n)$$ where $a_i\geq 1$ are positive integers.
\item Ellenberg, Satriano and Zureick-Brown~\cite{ellenberg_satriano_zureick-brown_2023} develop a version of what they call Batyrev--Manin--Malle conjecture for the Artin stacks with finite diagonal~\cite[Conjecture 4.14]{ellenberg_satriano_zureick-brown_2023}. The conjecture predicts the exponential growth (i.e.\ the exponent of~$B$) of the number of rational points.
\item In \cite{dardayasudabm} the authors develop a version of what they call Batyrev--Manin conjecture for Deligne--Mumford stacks~\cite[Conjecture 5.6]{dardayasudabm} which predicts the asymptotic growth rate (i.e.\ also the exponent of logarithm) for the number of rational points.
\end{itemize}
We now give more details about~\cite{dardayasudabm}. Let~$\XX$ be a separated, geometrically irreducible, smooth Deligne--Mumford stack over a number field~$F$, which is not isomorphic to~$\Spec(F)$, having a projective variety~$\overline\XX$ for its coarse moduli space, which is {\it Fano}. Let us explain the last condition: the pullback homomorphism $\Pic(\overline\XX)_{\QQ}\xrightarrow{\sim}\Pic(\XX)_{\QQ}$ of the $\QQ$-Picard groups of the coarse moduli space~$\overline\XX$ and of~$\XX$ is an isomorphism~\cite[Proposition 3.2]{dardayasudabm} and the condition asks that the anticanonical class~$\omega_{\XX}^{-1}$ corresponds to an ample $\QQ$-line bundle on~$\overline\XX$ for the isomorphism. For an $F$-field~$K$, we denote by $\XX(K)$ the set of isomorphism classes of morphisms $\Spec(K)\to\XX$. (We note a difference with the notation in~\cite{dardayasudabm} where for the same notion we have used the notation \(\cX\langle K \rangle\)) Let~$\mathcal J_0\XX$ be the stack of {\it twisted $0$-jets of} $\XX$, it parametrizes the representable morphisms $B\mu_{\ell}\to\XX$ for $\ell\geq 1$, where $\mu_{\ell}$ denotes the finite group scheme of $\ell$-th roots of unity. The stack~$\mathcal J_0\XX$ is of finite type of~$\XX$ \cite[Section 2.1]{dardayasudabm} containing a copy of~$\XX$ for a connected component. The element~$\XX$ of the finite set~$\pi_0(\mathcal J_0\XX)$ of connected components of~$\XX$ we call the {\it untwisted sector of~$\XX$}, while the other elements of~$\pi_0(\mathcal J_0\XX)$ are called the {\it twisted sectors of~$\XX$}. We denote by~$\pijx$ the set of the twisted sectors of~$\XX$. For almost all~$v$ in the set of the finite places~$M_F^0$ of~$F$, we construct \cite[Section 2.2]{dardayasudabm} a {\it residue map} $\psi_v:\XX(F_v)\to\pi_0(\mathcal J_0\XX)$ which satisfies that for almost all $\vMFz$ one has that the composite map $X(F)\to X(F_v)\to\pi_0(\mathcal J_0\XX)$ is the map to the singleton~$\{\XX\}$. 

We call a function $c:\pi_0(\mathcal J_0\XX)\to\RR$ with \(c(\cX)=0\)  a {\it raising function. }
A {\it raised line bundle} on \(\cX\) is a pair \((L,c)\) of a line bundle \(L\) on \(\cX\) and a raising function \(c\).  A height function \(H_{(L,c)}\) of a raised line bundle \((L,c)\) is defined by 
\[
H_{(L,c)}(x) = H_L(x) \cdot \prod_v q_v^{c\circ \psi_v (x)},  
\]
where \(H_L\) is an ordinary/stable height function of \(L\),  \(v\) runs over those finite places with \(\psi_v\) being defined and $q_v$ is the cardinality of the residue field at~$v$. 
The function \(H_{(L,c)}\) is uniquely determined from the given raised line bundle up to multiplication with functions bounded below and above by positive constants.
\begin{rem}
\normalfont
In \cite{dardayasudabm}, we assume that a raising function take non-negative values in early sections, then consider a more general situation allowing negative values in later sections. In the present paper, we allow negative values.
\end{rem}


For an element~$\YY$, we denote by $\age(\YY)\in\QQ_{\geq 0}$ the value defined in \cite[Definition~2.23]{dardayasudabm}. 
{In the same paper, we generalized the Manin conjecture to Deligne--Mumford stacks as follows.} 
\begin{conj}\label{firstconj}
Let~$\XX$ be a Fano stack. Let~$c$ be a raising function which is adequate, that is, for every twisted sector \(\cY\), \(\age(\cY)+c(\cY)\ge 1\) \cite[Definition 5.2]{dardayasudabm}. Suppose that~$\XX(F)$ is Zariski dense in~$\XX$. Then, there exists a thin subset $Z\subset\XX(F)$ \cite[Definition 5.4]{dardayasudabm} and a positive constant~$C$ such that $$\#\{x\in\XX(F)-Z|\hspace{0,1cm}H_{\omega^{-1}_{\XX},c}(x)\leq B\}\sim_{B\to\infty} CB(\log(B))^{\rho(\XX)+j_{c}(\XX)-1},$$
where~$\rho(\XX)$ is the Picard number of~$\XX$ \cite[Definition 5.5]{dardayasudabm} and $$j_c(\XX):=\#\{\YY\in\pijx|\hspace{0,1cm} \age(\YY)+c(\YY)=1\}.$$  
\end{conj} 
\subsection{Toric stacks}
{\it Toric stacks} over~$\CC$ are natural generalization of toric varieties over~$\CC$. In the literature, however, we have not found a general definition of a toric stack over a number field~$F$. A standard definition \cite{borisovchensmith} over~$\CC$  uses {\it stacky} fans and translates naturally to the case of a number field. We assume that the absolute Galois group acts trivially on the stacky fan and thus prefer calling such stacks {\it split} toric stacks.  An example of such a stack is provided by a weighted projective stack. A split toric stack possesses a Zariski-dense {\it stacky} torus~$\TT$, i.e.\ the product $\TT=T\times BG$ of a split (algebraic) $F$-torus~$T\cong \Gm^d,$ where $d\geq 1,$ and the stack~$BG$, where $G=\prod_{i=1}^{\ell}\mu_{n_i}$ for some~$\ell\geq 1$ and~$n_i\geq 1$ for $i=1\doots \ell$. The rational points of~$\TT$ form a commutative group $$\TT(F)\cong {(F^{\times})^d\times\prod_{i=1}^{\ell}(F^{\times}/(F^{\times})_{n_i})},\text{ where }(F^{\times})_{n_i}=\{x^{n_i}|\hspace{0,1cm}x\in F^{\times}\}.$$
We endow~$\omega_{\XX}^{-1}$ with a natural metrics (which is specified in Paragraph~\ref{cwgh}) and set $c=1-\age,$ so that $K_{\XX,\orb}^{-1}=(\omega_{\XX}^{-1}, c),$ where~$K_{\XX,\orb}$ is the class defined in the sense of~\cite[Definition 9.1]{dardayasudabm}. 
We denote by \(H_{K^{-1}_{\XX,\orb}}\) a specific choice of height function of \(K^{-1}_{\XX,\orb}\). The main result of this article is:
\begin{thm}[Manin conjecture for split toric stacks, Corollary \ref{cor:main}] 
\label{firstthm}
For an explicit $C>0$, one has that $$\#\{x\in\TT(F)|\hspace{0,1cm} H_{K^{-1}_{\XX,\orb}}(x)\leq B\}\sim_{B\to\infty}CB\log(B)^{\rho(\XX)+\#\pijx -1}.$$
\end{thm}

This theorem hence verifies Conjecture~\ref{firstconj} for toric Fano stacks given with our specific choices  of a raising function \(c\) and a height function, although the theorem holds without the condition that \(\omega_\cX^{-1}\) is ample. Note that this raising function \(c\) is automatically adequate.  The thin subset \(Z\) in the conjecture is chosen to be \(\cX(F)-\cT(F)\) in this theorem.   
This is a first counting result for {\it every} split toric stack. In the case when~$\XX$ is of the dimension zero, the claim is covered by \cite[Theorem 1.3.3]{commcase}. The above result of the first named author for the weighted projective stacks is with the respect to {\it quasi-toric} heights which do {\it not} correspond to the anti-canonical raised line bundle (formally, this is proven in \cite[Example 5.8]{dardayasudabm}, but can be observed immediately as the for quasi-toric heights the exponent of the logarithm is $0$, while here, in general, is non-zero).  Thus, even for the weighted projective stacks the result is new. 
{In preparation for the proof of the main theorem, we also show the Northcott property for raised line bundles that are not necessarily anti-canonical or positively raised, thus not covered by our previous result \cite[Proposition 4.8]{dardayasudabm} (see Section \ref{northcottsection}).
In fact, a large part of  preparation is not restricted to the anti-canonical raised line bundle. We hope that we can generalize also the main theorem to more general raised line bundles as well as to non-split toric stacks (i.e.\ the ``Batyrev--Manin conjecture'' for non-split toric stacks) in a future paper. 
}

In Section \ref{Split Stacky tori} we discuss split stacky tori, in Section~\ref{Heights} we write down our theory of heights for the split toric stacks using stacky fans and in Section~\ref{aohzf}, we use harmonic analysis  to prove the theorem

\subsection*{Acknowledgements}
The authors would like to thank to Antoine Chambert-Loir and Pierre Le Boudec for useful discussions. The second-named author 
was supported by JSPS KAKENHI Grant Numbers JP21H04994 and JP23H01070.

\section{Split Stacky tori} \label{Split Stacky tori}
We discuss split stacky tori. They are products of ordinary split tori and stacks $\prod_{i=1}^{k}B\mu_{\ell_i},$ where $k\geq 1$ and $\ell_i\geq 1$ are integers.
\subsection{Local situation} In the following paragraphs we define logarithmic maps from and study their functorial properties.
\subsubsection{}
For an $F$-split torus~$T$, we denote by $X_*(T)=\Hom(\Gm, T)$ its cocharacter group and by $X^*(T)=\Hom (T, \Gm)$ its character group. The groups~$X_*(T)$ and~$X^*(T)$ are always finitely generated free abelian groups. %
For a homomorphism $\alpha:X_*(T)\to X_*(T')$ of the cocharacter groups of $F$-split tori~$T$ and~$T'$,  we denote by $\widetilde{\alpha}: T \to T'$ the corresponding homomorphism of the tori. By \cite[Section 2.2.3]{Bourqui}, for an $F$-split torus~$T$ and a finite place~$v$ of~$F$, there exists a continuous homomorphism $$\log_{T,v}:T(F_v)\to X_*(T),$$
while for an infinite place~$v$, there exists a continuous homomorphism
$$\log_{T,v}:T(F_v)\to X_*(T)_{\RR}.$$
When $T=\Gm$ and~$v$ is finite, one has $\log_{T,v}=\ord_v$. The homomorphisms~$\log_{T,v}$ are surjective and their kernel is the maximal compact subgroup~$T(\Ov)$ of~$T(F_v)$.
For any place~$v$ and any $\alpha:X_*(T)\to X_*(T')$, the diagram: \begin{equation}
\begin{tikzcd}
T(F_v) \arrow{d}{\log_{T,v}} \arrow{r}{\widetilde{\alpha}(F_v)}& T'(F_v)\arrow{d}{\log_{T',v}}\\
X_*(T)\arrow{r}{\alpha}& X_*(T'),
\end{tikzcd}
\label{toriandlatt}
\end{equation}
is commutative.  
\subsubsection{} Let~$\ell\geq 1$ be an integer and let~$v$ be a place of~$F$. The stack~$B\mu_{\ell}$ satisfies $B\mu_{\ell}(F_v)=\Gm(F_v)/\Gm(F_v)_{\ell},$ where $\Gm(F_v)_{\ell}$ is the closed subgroup of~$\Gm(F_v)$ given by the elements which are $\ell$-th powers. These groups are finite \cite[2.3.5.7]{darda:tel-03682761}. If~$v$ is finite, we have a surjective map $$\log_{B\mu_{\ell}, v}:B\mu_{\ell}(F_v)\to \ZZ/\ell\ZZ$$induced from $\ell\ZZ$-invariant homomorphism $$\Gm(F_v)\xrightarrow{\ord_v}\ZZ\to\ZZ/\ell\ZZ.$$
If~$v$ is infinite, we denote by~$\log_{B\mu_{\ell}, v}$ the canonical map to the singleton $\log_{B\mu_{\ell},v}:B\mu_{\ell}(F_v)\to\{0\}$.  For a diagonalizable $F$-finite group scheme $G=\mu_{\ell_1}\times\cdots\times \mu_{\ell_k}$ with~$k$ and~$\ell_i$ strictly positive integers by taking the product homomorphism, we obtain a surjective homomorphism $$\log_{BG, v}:BG(F_v)=B\mu_{\ell_1}(F_v)\times\cdots\times B\mu_{\ell_k}(F_v)\to (\ZZ/\ell_1\ZZ) \times\cdots\times (\ZZ/\ell_{k}\ZZ)=:G^D$$ if~$v$ is finite and $$\log_{BG, v}:BG(F_v)\to\{0\}$$if~$v$ is infinite. For every~$\vMF$, we set $$BG(\Ov):=\ker(\log_{BG, v}).$$ We set $p^{\ell}_v:\Gm(F_v)=F_v^{\times}\to B\mu_{\ell}(F_v)$ to be the quotient homomorphism. Let $m\in\{0\doots \ell-1\}$. The diagram 
\begin{equation}
\label{opsovu}
\begin{tikzcd}
\Gm(F_v) \arrow{r}{\log_{\Gm, v}} \arrow{d}{ p^{\ell}_v}& \ZZ\arrow{d}{x\mapsto x+{\ell\ZZ}}\\
B\mu_{\ell}(F_v)\arrow{d}{x\mapsto x^m}\arrow{r}{\log_{B\mu_{\ell}, v}}&\ZZ/\ell\ZZ\arrow{d}{x\mapsto mx}\\
B\mu_{\ell}(F_v)\arrow{r}{\log_{N\mu_{\ell}, v}}&\ZZ/\ell\ZZ.
\end{tikzcd}
\end{equation}
is commutative.
 Set $p^{\ell}_F:\Gm(F)\to B\mu_{\ell}(F)=F^{\times}/(F^{\times})_{\ell},$ where $(F^{\times})_{\ell}$ is the subgroup of $F^{\times}$ given by the $\ell$-th powers, to be the quotient homomorphism. The diagram 
\begin{equation}
\label{opsov}
\begin{tikzcd}
\Gm(F) \arrow{r}{} \arrow{d}{p^{\ell}_F}& \Gm(F_v)\arrow{d}{p^{\ell}_v}\\
B\mu_{\ell}(F)\arrow{r}{}&B\mu_{\ell}(F_v).
\end{tikzcd}
\end{equation}is commutative. Hence is commutative the diagram:
\begin{equation}
\label{opsov}
\begin{tikzcd}
\Gm(F) \arrow{r}{} \arrow{d}{(p^{\ell_i}_F)^k_{i=1}}& \Gm(F_v)\arrow{d}{(p^{\ell_i}_v)_{i=1}^k}\\
\prod_{i=1}^kB\mu_{\ell_i}(F)=BG(F)\arrow{r}{}&\prod_{i=1}^kB\mu_{\ell_i}(F_v)=BG(F_v).
\end{tikzcd}
\end{equation}
\subsubsection{} Let~$T$ be an $F$-split torus and let $G=\prod_{i=1}^{k}B\mu_{\ell_k}$ be an $F$-diagonalizable finite group scheme. We say that $\TT:=T\times BG$ is an $F$-{\it split stacky torus}, or for simplicity, {\it split stacky torus}.  We define $$X_*(\TT):=X_*(T)\times G^D.$$ 
For a finite place~$v$, we have a surjective homomorphism \begin{align*}\log_{\TT, v}:\TT(F_v)=T(F_v)\times BG(F_v)&\to X_*(T)\times G^D=X_*(\TT)\\ (\xxx, x)&\mapsto (\log_{T,v}(\xxx),\log_{BG, v}(x)).\end{align*}
For an infinite place~$v$, we set~$\log_{\TT, v}$ to be the composite map $$\log_{\TT, v}:\TT(F_v)\to T(F_v)\xrightarrow{\log_{T, v}}X_*(T)_{\RR}.$$
\begin{lem}
For $\yyy\in X_*(\TT),$ we define a homomorphism $$f_{\yyy}:\ZZ\to X_*(\TT),\hspace{1cm}1\mapsto \yyy.$$ For $\yyy\in X_*(T)$, we set: $$f_{\yyy}:=f_{(\yyy,0)}.$$ 
\end{lem} 
\begin{lem}
Let $(\yyy,y)=(\yyy,y_1\doots y_k)\in X_*(\TT)=X_*(T)\times \prod_{i=1}^{k}(\ZZ/\ell_k\ZZ)$. Let~$v$ be a place of~$F$. For $i=1\doots k$ let $\widetilde y_i\in \{0\doots \ell_i-1\}$  be lifts of~$y_i$. We define a homomorphism \begin{align*}
\widetilde{f_{(\yyy,y)}}(F_v):\Gm(F_v)&\to T(F_v)\times BG(F_v)=\TT(F_v),\\ x&\mapsto (\widetilde{f_{\yyy}}(x), p^{\ell_1}_v(x)^{\widetilde y_1}\doots p^{\ell_k}_v(x)^{\widetilde y_k}).
\end{align*} We also define a homomorphism 
$$\widetilde{f_{(\yyy,y)}}(F):\Gm(F)\to T(F)\times BG(F)=\TT(F),\hspace{0,2 cm}x\mapsto (\widetilde{f_{\yyy}}(x), p^{\ell_1}_F(x)^{\widetilde y_1}\doots p^{\ell_k}_F(x)^{\widetilde y_k}).$$
\end{lem}
For every place~$v$ of~$F$, we define $$\TT(\Ov):=T(\Ov)\times BG(\Ov).$$
\begin{lem}\label{mumija}
Let $\yyy\in X_*(\TT)$. Let~$v$ be a place of~$F$. 
\begin{enumerate}
\item If~$v$ is finite, the diagram 
$$
\begin{tikzcd}
\Gm(F_v) \arrow{d}{\widetilde{f_{\yyy}}(F_v)} \arrow{r}{\log_{\Gm, v}}&\ZZ\arrow{d}{{f_{\yyy}}}\\
\TT(F_v)\arrow{r}{\log_{\TT, v}}&X_*(\TT)
\end{tikzcd}
$$
is commutative. (If~$v$ is infinite and the groups in the second column are tensorized by~$\RR$, the obtained diagram is immediately commutative by the commutativity of Diagram~\ref{toriandlatt}.)
\item One has that $\widetilde{ f_{\yyy}}:\Gm(F_v)\to\TT(F_v)$ is continuous.
\item One has that $\widetilde {f_{\yyy}}(\Gm(\mathcal O_v))\subset \TT(\Ov).$
\item The diagram
$$
\begin{tikzcd}
\Gm(F) \arrow{d}{\widetilde{f_{\yyy}}(F)} \arrow{r}{}&\Gm(F_v)\arrow{d}{\widetilde{f_{\yyy}}(F_v)}\\
\TT(F)\arrow{r}{}&\TT(F_v).
\end{tikzcd}
$$
is commutative.
\item The map $$X_*(\TT)\to \Hom_{\cont}(\Gm(F_v), \TT(F_v)),\hspace{1cm}\yyy'\mapsto \widetilde {f_{\yyy'}}(F_v)$$
is a homomorphism.
\end{enumerate}
\end{lem}
\begin{proof}
\begin{enumerate}
\item This follows from the commutativity of Diagram~\ref{toriandlatt} and~\ref{opsovu}.
\item The homomorphism~$\widetilde {f_{\yyy}}(F_v)$ is continuous. The quotient maps $p_v^{\ell_i}(x)$ are obviously continuous, hence are such the homomorphisms $x\mapsto p_v^{\ell_i}(x)^{\widetilde y_i}$. We deduce that, being a product of continuous homomorphisms, the homomorphism $\widetilde{f_{\yyy}}(F_v)$ is continuous. 
\item We already know that $\widetilde{f_{\yyy}}(\Gm(\Ov))$ is contained in $T(\Ov)$. If~$v$ is infinite, we have that $BG(\Ov)$ is defined to be the $BG(\Fv)$, hence $\widetilde{f_{(\yyy,x)}}(\Gm(\Ov))\subset\TT(\Ov)=T(\Ov)\times BG(\Ov)$. If~$v$ is finite, by \cite[Proposition 2.3.5.4]{darda:tel-03682761}, we have that $p_v^{\ell_i}(\Gm(\Ov))^{\widetilde y_i}\subset B\mu_{\ell_i}(\Ov)^{\widetilde y_i}\subset B\mu_{\ell_i}(\Ov)$ for every~$i$. We obtain the same conclusion. The second claim now follows.
\item  By the commutativity of Diagram (\ref{opsov}), we obtain immediately that for any $i=1\doots k$, the diagram
$$
\begin{tikzcd}
\Gm(F)\arrow{r}{}\arrow{d}{(x\mapsto x^{\widetilde y_i})\circ p^{\ell_i}_F} &\Gm(F_v)\arrow{d}{(x\mapsto x^{\widetilde y_i})\circ p^{\ell_i}_{F_v}}\\
B\mu_{\ell}(F)\arrow{r}{} &B\mu_{\ell}(F_v)
\end{tikzcd}
$$
is commutative. By using this fact and the commutativity of
$$ 
\begin{tikzcd}
\Gm(F)\arrow{r}{}\arrow{d}{\widetilde{f_{\yyy}}(F)}& \Gm(F_v)\arrow{d}{\widetilde{f_{\yyy}}(F_v)}\\
T(F)\arrow{r}{}& T(F_v),
\end{tikzcd}
$$
we obtain the commutativity of the diagram in the statement.
\item It is immediate that $$X_*(\TT)\to \Hom(\ZZ, X_*(\TT)), \hspace{1cm} \yyy'\mapsto f_{\yyy'}$$ is an isomorphism. One has that $\Hom(\Gm,T)\to \Hom_{\cont}(\Gm(F_v),T(F_v))$ given by $\alpha\mapsto \alpha (F_v)$ is a homomorphism. Finally, for every $i=1\doots k$, one has that $$\ZZ/\ell_i\ZZ\to \Hom_{\cont}(\Gm(F_v),B\mu_{\ell_i}(F_v)),\hspace{1cm}t\mapsto (z\mapsto p_v^{\ell_i}(z)^{\widetilde t}),$$ is a homomorphism. The map in the question is the composite of the product of the latter maps with the first map, hence is a homomorphism.
\end{enumerate}
\end{proof}
\subsection{Global situation} We define and study the adelic space of~$\TT$.
\subsubsection{}Let~$\TT$ be an $F$-split stacky torus. 
 We define $$\TT(\AAF):=\sideset{}{'}\prod_v\TT(F_v)$$ where the restricted product is taken with the respect to compact open subgroups $\TT(\Ov)\subset\TT(F_v)$ for $\vMFz$.  We will use occasionally the identification given by the isomorphism of abelian topological groups $$\TT(\AAF)\xrightarrow{\sim}T(\AAF)\times BG(\AAF),\hspace{1cm}(\xxx_v,y_v)_v \mapsto ((\xxx_v)_v,(y_v)_v).$$ 
%
For $x\in X_*(\TT)$ we will by abuse of notation also write~$x$ for its image for the canonical morphism $X_*(\TT)\to X_*(\TT)_{\RR}.$ We have a continuous homomorphism $$\log_{\TT}:\TT(\AAF)\to X_*(\TT)_{\RR},\hspace{1cm} (\xxx_v)_v\mapsto \sum_{\vMF}\log(q_v)\log_{\TT,v}(\xxx_v),$$
where for a finite place~$v$ we define~$q_v$ to be the cardinality of the residue field at~$v$ and for an infinite place~$v$, we define~$q_v:=\exp([F_v:\RR])$. We note that $\log_{\TT}$ factorizes as $\TT(\AAF)\to T(\AAF)\xrightarrow{\log_{T}}X_*(T)_{\RR}=X_*(\TT)_{\RR},$ and, in particular, the map~$\log_{\TT}$ is surjective. Let us define $$K(\TT):=\prod_{\vMF}\TT(\Ov)\subset\TT(\AAF).$$
Lemma~\ref{mumija}(2) gives for $\yyy\in X_*(\TT)$ and every place~$v$ a continuous homomorphism $\widetilde{f_{\yyy}}(F_v):\Gm(F_v)\to \TT(F_v)$ which satisfies by Lemma~\ref{mumija}(3) that $$\widetilde{f_{\yyy}}(F_v)(\Ov)\subset \TT(\Ov).$$
Moreover, the map $$X_*(\TT)\to \Hom_{\cont}(\Gm(F_v),\TT(F_v))$$ is a homomorphism.  From \cite[Proposition 2.4.1.1]{darda:tel-03682761}, for $\yyy\in X_*(\TT)$, we deduce a continuous homomorphism $$\widetilde{f_{\yyy}}:\Gm(\AAF)\to \TT(\AAF),\hspace{1cm} (x_v)_v\mapsto (f_{\yyy}(x_v))_v$$which satisfies $$f_{\yyy}(\AAF)(K(\Gm))\subset K(\TT).$$ Moreover, it follows from Lemma~\ref{mumija}(5) that $$\Hom_{\cont}(\Gm(\AAF), \TT(\AAF)),\hspace{1cm}\yyy\mapsto \widetilde {f_{\yyy}}(\AAF)$$ is a homomorphism. Fix $(\yyy, x)\in X_*(\TT)$. The diagram
$$ 
\begin{tikzcd}\label{necfortaj}
\Gm(\AAF)\arrow{r}{\log_{\Gm}}\arrow{d}{\widetilde{f_{\yyy,x}}(\AAF)}&\RR\arrow{d}{(f_{\yyy})_{\RR}}\\
\TT(\AAF)\arrow{r}{\log_{\TT}}&X_*(\TT)_{\RR}
\end{tikzcd}
$$ is commutative because $$\log_{\TT}\circ (\widetilde {f_{(\yyy,x)}})=\log_T\circ \widetilde {f_{\yyy}}(\AAF)=(f_{\yyy})_{\RR}\circ\log_{\Gm}=((f_{(\yyy,x)})_{\RR})\circ\log_{\Gm}.$$
By using \cite[Lemma 2.5.1]{commcase} which implies that for almost all~$v$ the image of $BG(F)\to BG(F_v)$ lies in~$BG(\Ov)$ and the well-known fact that for almost all~$v$ the image of $T(F)\to T(F_v)$ lies in~$T(\Ov)$, we deduce that for almost all places~$v$, the image of the canonical map $$\TT(F)=T(F)\times BG(F)\to T(F_v)\times BG(F_v)=\TT(F_v)$$ lies in~$\TT(\Ov)=T(\Ov)\times BG(\Ov).$ We deduce a canonical homomorphism $\Delta^{\TT}:\TT(F)\to\TT(\AAF).$ We will often use only the notation~$\Delta$ to denote~$\Delta^{\TT}$.
 We have seen in Lemma \ref{mumija}(4), that the diagram 
$$
\begin{tikzcd}
\Gm(F) \arrow{d}{\widetilde{f_{(\yyy,x)}}(F)} \arrow{r}{}&\Gm(F_v)\arrow{d}{\widetilde{f_{(\yyy,x)}}(F_v)}\\
\TT(F)\arrow{r}{}&\TT(F_v).
\end{tikzcd}
$$
is commutative. We deduce that the diagram
\begin{equation}
\label{oppm}
\begin{tikzcd}
\Gm(F) \arrow{d}{\widetilde{f_{(\yyy,x)}}(F)} \arrow{r}{}&\Gm(\AAF)\arrow{d}{\widetilde{f_{(\yyy,x)}}(F_v)}\\
\TT(F)\arrow{r}{}&\TT(\AAF).
\end{tikzcd}
\end{equation}
 is commutative.  The kernel of $T(F)\times BG(F)=\TT(F)\to\TT(\AAF)=T(\AAF)\times BG(\AAF)$ canonically identifies with the kernel of $BG(F)\to BG(\AAF)$, which is well known (see e.g.\ \cite[Chapter I, Theorem 4.10]{ArithmeticDuality}) to be the finite group~$\Sh^1(G)$.
\subsubsection{}  We set: $$\TT(\AAF)_1:=\ker(\log_{\TT}),$$ so that $$1\to\TT(\AAF)_1\to\TT(\AAF)\to X_*(\TT)_{\RR}\to 1 $$ is exact. 
We note that for any $\yyy\in X_*(\TT),$ by the commutativity of Diagram~\ref{necfortaj}, we have that  $\widetilde{f_{\yyy}}(\AAF)(\Gm(\AAF)_1)\subset\TT(\AAF)_1. $
 It is immediate that $$\TT(\AAF)_1=T(\AAF)_1\times BG(\AAF).$$ Moreover, by the fact that the image of the diagonal $\Delta:T(F)\to T(\AAF)$ lies in~$T(\AAF)_1$, we have that the image of $\Delta:\TT(F)\to \TT(\AAF)$ lies in~$\TT(\AAF)_1$. The following property is immediately deduced from the corresponding property for~$T$ and the corresponding property for~$BG$ \cite[Chapter I, Theorem 4.10]{ArithmeticDuality}.
\begin{prop}
The image $\Delta(\TT(F))$ is a closed, discrete and cocompact subgroup of~$\TT(\AAF)_1$.
\end{prop}
\subsection{Measures} We endow our locally compact abelian groups by Haar measures. Let~$\TT=T\times BG$ be an $F$-stacky torus.
\subsubsection{} In this paragraph, we define local measures. 
\begin{mydef}
We say that the sequence $$1\to (A,da)\to (B,db)\to (C,dc)\to 1,$$ of pairs of locally compact abelian groups endowed with Haar measures is exact, 
 the homomorphism $A\to B$ is a homeomorphism of~$A$ to a closed subgroup of~$B$, the homomorphism $B\to C$ identifies~$C$ with a quotient of~$B$ by the image of~$A$ and, and if $db/da=dc$. 
\end{mydef}
\begin{mydef}
For a discrete set~$Y$, we denote by~$\mu^{\#}_Y$ the counting measure on~$Y$. If~$Y$ is clear from the context, we may shorten the notation and write only~$\mu^{\#}$. For a compact space~$Z$, denote by~$\mu^{\proba}_Z$ the probability measure on~$Z$. If~$Z$ is clear from the context, we may shorten the notation and write~$\mu^{\proba}$. 
\end{mydef}
Endow the discrete group~$X_*(\TT)$ with the measure $\mu^{\#}_{X_*(\TT)}$.   Endow the real vector space $X_*(T)_{\RR}=X_*(\TT)_{\RR}$ with the Lebesgue measure~$d\xxx$ normalized by the lattice $X_*(T)$.
\begin{mydef}\label{muvdef}
\begin{enumerate}
\item Let~$v$ be a finite place of~$F$.  We define~$\mu_v$ to be the unique Haar measure on~$\TT(F_v)$  for which the sequence $$1\to (\TT(\Ov), \mu^{\proba})\to(\TT(F_v), \mu_v)\xrightarrow{\log_{\TT, v}}(X_*(\TT), \mu^{\#})\to 1 $$
is exact.
\item Let~$v$ be an infinite place of~$F$. We define~$\mu_v$ to be the unique Haar measure on~$\TT(F_v)$ for which the sequence 
$$1\to (\TT(\Ov), \mu^{\proba})\to(\TT(F_v), \mu_v)\xrightarrow{\log_{\TT, v}}(X_*(\TT)_{\RR}, d\xxx)\to 1$$
is exact.
\end{enumerate}
\end{mydef}
\subsubsection{} We now study the global measures.
\begin{mydef}
 We define a Haar measure~$\mu_{\TT}$ on~$\TT(\AAF)$ to be the product measure $\mu_{\TT}:=\otimes_{\vMF}\mu_v.$ When~$\TT$ is clear from context, we may shorten the notation and write~$\mu$.
\end{mydef}
Clearly, the identification $\TT(\AAF)=T(\AAF)\times BG(\AAF)$ of locally compact abelian groups induces an identification of pairs of locally compact abelian groups endowed with Haar measures:
$$(\TT(\AAF),\mu)=(T(\AAF),\mu_T)\times (BG(\AAF),\mu_{BG}).$$
\begin{mydef}
We define~$\mu_{\TT,1}$ to be the unique Haar measure on~$\TT(\AAF)_1$ for which the sequence $$1\to (\TT(\AAF)_1,\mu_{\TT,1}) \to (\TT(\AAF),\mu)\xrightarrow{\log_{\TT}} (X_*(\TT)_{\RR},d\xxx) \to 1$$
is exact.
\end{mydef}
Clearly, the identification $\TT(\AAF)_1=T(\AAF)_1\times BG(\AAF)$ of locally compact abelian groups induces an identification of pairs of locally compact abelian groups endowed with Haar measures:
$$(\TT(\AAF)_1,\mu_{\TT,1})=(T(\AAF)_1,\mu_{T,1})\times (BG(\AAF),\mu_{BG}).$$
\begin{mydef}
Endow the closed and discrete subgroup $\Delta(\TT(F))\subset\TT(\AAF)_1$ with the discrete Haar measure $$\frac{\#\Sh^1(G)}{\# G(F)}\cdot\mu^{\#}_{\Delta(\TT(F))}.$$ Endow $\TT(\AAF)/\Delta(\TT(F))$ and $ \TT(\AAF)_1/\Delta(\TT(F))$ with the Haar measures $$(\mu/{\Delta}):=\mu/(\#\Sh^1(G)\cdot G(F)^{-1}\cdot\mu^{\#}_{\Delta(\TT(F))})$$and $$(\mu_{1}/\Delta):=\mu_1/(\#\Sh^1(G)\cdot G(F)^{-1}\cdot\mu^{\#}_{\Delta(\TT(F))}),$$respectively. We define $$\Tam(\TT):= (\mu_{1}/\Delta)\big(\TT(\AAF)_1/\Delta(\TT(F))\big).$$ 
\end{mydef}
\begin{prop}\label{ukmar}
One has that $$\Tam(\TT)=\# G^D.$$
\end{prop}
\begin{proof}
Clearly, one has that 
\begin{align*}\Tam(\TT)&=
\big((\mu_1/\mu^{\#}_{\Delta(T(F))} )(T(\AAF)_1/T(F))\big)\times\\
&\quad\quad\times \big(\mu_{BG}/ (\#\Sh^1(G)\cdot G(F)^{-1}\cdot\mu^{\#}_{\Delta(\TT(F))})\big)\big(BG(\AAF)/\Delta(BG(F))\big).
\end{align*}
Our torus~$T$ is split, so \cite[Theorem 3.5.1]{Ono} gives $$(\mu_1/\mu^{\#}_{\Delta(T(F))} )(T(\AAF)_1/T(F)) =1.$$
By \cite[Lemma 3.1.5]{commcase}, one has that $$\mu_{BG}/ (\#\Sh^1(G)\cdot G(F)^{-1}\cdot\mu^{\#}_{\Delta(\TT(F))})\big)(BG(\AAF)/\Delta(BG(F)))=\frac{\# G^D}{\#\Sh^2(G)},$$
where~$\Sh^2(G)$ is the second Tate--Shafarevich group of~$G$. By using the fact that $\#\Sh^2(G)=\#\Sh^1(G^D)$ from \cite[Chapter I, Theorem 4.20(a)]{Milne} and the fact that $\Sh^1(G^D)=1$ because~$G^D$ is constant \cite[Lemma 1.1] {Sansuc}, we obtain that $\Tam(\TT)=\# G^D.$
\end{proof}
\subsection{Characters} We now study the characters of our groups. For a locally compact abelian group endowed with a Haar measure $(\mathcal G,  dg)$ we denote by $(\mathcal G^*, dg^*)$ its Pontryagin dual (its group of characters $\chi:\mathcal G\to S^1$, where~$S^1$ stands for the unit circle in~$\CC$) endowed with the dual Haar measure \cite[Chapter II, \S 1, $\text{n}^{\circ}$ 3 Definition 4]{TSpectrale}. If $\mathcal H\leq\mathcal G$ is a closed subgroup, we denote by $\mathcal H^{\perp}[\mathcal G^*]$ the closed subgroup of~$\mathcal G^*$ given by the characters vanishing on~$\mathcal H$. For Haar measures $dh$ and~$dg$ on $\mathcal H$ and $\mathcal G$, we denote by $dh^*$ and~$dg^*$, the dual Haar measures on~$\mathcal H^*$ By combining \cite[Chapter II, \S 1, $\text{n}^{\circ}$ 7, Theorem 4]{TSpectrale} and \cite[Chapter II, \S 1, $\text{n}^{\circ}$ 8, Proposition 9]{TSpectrale}, we have that the following sequence is exact $$1\to (\mathcal H^{\perp}[\mathcal G^*], (dg/dh)^*)\to (\mathcal G^*, dg^*)\to (\mathcal H^*, dh^*)\to 1.$$
\subsubsection{} We will define ``components'' of local characters.
\begin{mydef}\label{betafy}
Let~$v$ be a place of~$F$ and let $\chi\in\TT(F_v)^*$ be a character. For $\yyy\in X_*(\TT),$ we define a character $$\chi^{(\yyy)}:=\chi\circ \widetilde{ f_{(\yyy)}}(F_v)\in\Gm(F_v)^*.$$
For $\yyy\in X_*(T)$, we set $\chi^{(\yyy)}:=\chi^{(\yyy,0)}$. 
\end{mydef}
The following corollary follows easily  from Lemma~\ref{mumija}.
\begin{cor} \label{muiop}
Let~$v$ be a place of~$F$. Let $\yyy\in X_*(\TT)$.
\begin{enumerate}
\item  One has that the association $\chi\mapsto\chi^{(\yyy)}$ is a continuous homomorphism $\TT(F_v)^*\to \Gm(\Fv)^*.$
\item If $\chi\in\TT(F_v)^*$ vanishes on $\TT(\Ov)$, then~$\chi^{(\yyy)}$ vanishes on~$\Gm(\Ov)$.
\item  One has that $\yyy'\mapsto \chi^{(\yyy')}$ is a homomorphism $X_*(\TT)\to \Gm(\Fv)^*. $ 
\end{enumerate}
\end{cor}
\begin{proof}
\begin{enumerate}
\item The claim follows from Lemma~\ref{mumija}(2)  and \cite[Chapter II, \S 1, $\text{n}^{\circ} 1$, Paragraph~7]{TSpectrale}.
\item The claim follows from Lemma \ref{mumija}(4).
\item This is immediate from Lemma \ref{mumija}(5).
\end{enumerate}
\end{proof}

\subsubsection{} For a stacky torus $\TT=T\times BG$, we define
$$X^*(\TT):=X^*(T)\times (G^D)^*,$$where~$X^*(T)$ is the group of homomorphisms of algebraic groups $T\to\Gm$. Recall that we have a perfect pairing $$X_*(T)\times X^*(T)\to\ZZ, \hspace{1cm}  (\xxx,\mmm)\mapsto \mmm\circ\xxx\in\Hom(\Gm,\Gm)=\ZZ.$$ We will write $\langle\xxx,\mmm\rangle$  for the image of $(\xxx,\mmm)$. 
\begin{lem} \label{djurdjule}
Let~$d\mmm$ be the Lebesgue measure on the real vector space~$X^*(\TT)_{\RR}$ normalized by the lattice~$X^*(T)$.
The isomorphism \begin{align*}
\nnn_{\RR}:X^*(\TT)_{\RR}=X^*(T)_{\RR}&\xrightarrow{\sim} (X_*(\TT)_{\RR})^*=(X_*(T)_{\RR})^*\\
\mmm&\mapsto \big(\xxx\mapsto \exp(i\langle\xxx,\mmm\rangle)\big)
\end{align*}
is in fact an isomorphism of pairs of real vector spaces and Lebesgue measures $$\bigg(X^*(\TT)_{\RR}, \frac{1}{(2\pi)^{\dim(\TT)}}d\mmm\bigg)\xrightarrow{\sim}((X_*(\TT)_{\RR})^*, d\xxx^*).$$
\end{lem}
\begin{proof}
By \cite[Chapter II, \S 1, $\text{n}^{\circ}9,$ Corollary 3]{TSpectrale}, the isomorphism $$X^*(\TT)_{\RR}\to (X_*(\TT)_{\RR})^*,\hspace{1cm} \mmm\mapsto (\xxx\mapsto \exp(2i\pi\cdot\langle\xxx,\mmm\rangle))$$ induces an identification of pairs $$(X^*(\TT)_{\RR},d\mmm) \xrightarrow{\sim} ((X_*(T)_{\RR})^*, d\xxx^*).$$ 
This means that the isomorphism $$X^*(\TT)_{\RR}\to ((X_*(\TT)_{\RR})^*,\hspace{1cm} \mmm\mapsto (\xxx\mapsto \exp(i\cdot\langle\xxx,\mmm\rangle))$$ induces an identification of pairs $$(X^*(\TT)_{\RR},(2\pi)^{-\dim(\TT)}d\mmm)=(X^*(\TT)_{\RR},(2\pi)^{-\rk(X_*(T))}d\mmm) \xrightarrow{\sim} ((X_*(\TT)_{\RR})^*, d\xxx^*).$$
\end{proof}
\subsubsection{}We calculate the components of certain special characters.
\begin{lem}\label{locexp} Let $(\yyy,x)\in X_*(\TT)=X_*(T)\times G^D$ and let~$v$ be a place of~$F$. Let~$\mmm\in X^*(\TT)_{\RR}$. Consider the character $\chi_{\mmm}\in (\TT(\AAF)_1)^{\perp}[\TT(\AAF)^*]$ 
given by $$\chi_{\mmm}:\xxx\mapsto \exp(i\cdot\log(q_v)\cdot \langle\log_{T,v}(\xxx),\mmm\rangle).$$ One has that $$(\chi_{\mmm})^{(\yyy,x)}=\chi_{\langle\yyy,\mmm\rangle}=(\chi_{\mmm})^{(\yyy)}.$$
\end{lem}
\begin{proof}
Let $z\in\Gm(F_v)$. We have that
\begin{align*}(\chi_{\mmm})^{(\yyy,x)}(z)&=\chi_{\mmm}\big(\widetilde{f_{(\yyy,x)}}(z)\big)\\&=\exp\big(i\cdot\log(q_v)\cdot\big\langle\log_{\TT,v}\big( \widetilde{f_{(\yyy,x)}}(z)\big),\mmm\big\rangle\big)\\
&=\exp\big(i\cdot\log(q_v)\cdot\big\langle\log_{T,v}\big( \widetilde{f_{\yyy}}(z)\big),\mmm\big\rangle\big)\\
&=\exp\big(i\cdot\log(q_v)\cdot\big\langle f_{\yyy}(\log_{\Gm}(z)),\mmm\big\rangle)\\
&=\exp(i\cdot\log(q_v)\cdot\langle \log_{\Gm}(z)\cdot f_{\yyy}(1),\mmm\rangle)\\
&=\exp(i\cdot\log(q_v)\cdot\langle\log_{\Gm}(z),\langle \yyy,\mmm\rangle\rangle)\\&=\chi_{\langle\yyy,\mmm\rangle}(z).
\end{align*} Moreover, it is clear that $(\chi_{\mmm})^{(\yyy,x)}=(\chi_{\mmm})^{(\yyy,0)}=(\chi_{\mmm})^{(\yyy)}$. The statement is proven.
\end{proof}
\subsubsection{} 
Suppose that~$v$ is finite. The exact sequence  $$1\to\TT(\Ov)\to\TT(F_v)\xrightarrow{\log_{\TT, v}} X_*(\TT)\to 1$$provides an exact sequence $$1\to (X_*(\TT))^*\xrightarrow{\log_{\TT, v}^*} (\TT(F_v))^*\to (\TT(\Ov))^{*}\to 1,$$ which provides an isomorphism
\begin{align*}(X_*(\TT))^*&\xrightarrow{\log_{\TT,v}^*}\TT(\Ov)^{\perp}[\TT(F_v)^*].\\
\xxx&\mapsto \xxx\circ\log_{\TT, v}.
\end{align*}
\begin{lem}\label{coolrelat}
Let $\chi\in (\TT(\Ov))^{\perp}[\TT(\Fv)^*]$ be a character and let $\yyy\in X_*(\TT)$.  One has that  $$((\log_{\TT,v}^*)^{-1}(\chi)) (\yyy)=\chi^{(\yyy)}(\pi_v).$$
\end{lem}
\begin{proof}
We have a commutative diagram 
$$ 
\begin{tikzcd}
\Gm(F_v)\arrow{d}{\widetilde{f_{\yyy}}} \arrow{r}{\log_{\Gm,v}}&\ZZ\arrow{d}{f_{\yyy}}\\
\TT(F_v)\arrow{r}{\log_{\TT, v}}& X_*(\TT).
\end{tikzcd}
$$
In particular $(\log_{\TT,v}(\widetilde{f_{\yyy}}(\pi_v)))=\yyy.$  Hence, $$\chi^{(\yyy)}(\pi_v)=\chi\big(\widetilde{f_{\yyy}}(\pi_v)\big)=(\chi\circ (\log_{\TT,v})^{-1})\circ (\log_{\TT, v}(\widetilde {f_{\yyy}}(\pi_v)))=(((\log_{\TT, v})^*)^{-1}(\chi))(\yyy).$$as claimed. 
\end{proof}
\subsubsection{}Let us now look at the characters globally. We start by defining components.
  \begin{mydef} Let $\chi\in\TT(\AAF)^*$ be a character.
Let $\yyy\in X_*(\TT)$. We define a homomorphism
$$\chi^{(\yyy)}:\AAF^{\times}\to S^1,\hspace{1cm}(x_v)_v\mapsto\prod_{\vMF}(\chi_v)^{(\yyy)}(x_v).$$ 
\end{mydef}
The following properties are easily deduced.
\begin{cor}\label{obrik} Let $\yyy\in X_*(\TT)$.
\begin{enumerate}
\item  For every $\chi\in\TT(\AAF)^*$, one has that $$\chi^{(\yyy)}=\chi\circ(\widetilde{f_{\yyy}}(\AAF)).$$ Hence,~$\chi^{(\yyy)}$ is continuous, i.e.\ an element of $\Gm(\AAF)^*$, and the homomorphism $$\TT(\AAF)^*\to\Gm(\AAF),\hspace{1cm}\chi\mapsto \chi^{(\yyy)}$$
is continuous.
\item One has that the map $$X_*(\TT)\to\Hom_{\cont}(\Gm(\AAF)^*,\TT(\AAF)^*), \hspace{1cm}\chi\mapsto\chi^{(\yyy)} $$ is a homomorphism. 
\item Suppose that $\chi\in\TT(\AAF)^*$ vanishes on the group~$K(\TT)$. Then~$\chi^{(\yyy)}$ vanishes on the group~$K(\Gm)$.
\item Suppose that $\chi\in\TT(\AAF)^*$ vanishes on the image of the diagonal homomorphism $\Delta:\TT(F)\to\TT(\AAF).$  Then $\chi^{(\yyy)}$ vanishes on the image of the diagonal homomorphism $\Delta:\Gm(F)\to\Gm(\AAF).$ 
\end{enumerate}
\end{cor}
\begin{proof}
\begin{enumerate}
\item For every place~$v$, one has by Definition \ref{betafy} that $$(\chi^{(\yyy)})_v=\chi_v\circ (f_{\yyy}(F_v))=(\chi_v)^{(\yyy)}.$$ The statement follows.
\item This follows by the fact that $\widetilde{f_{\yyy'}}(\AAF)$ is a homomorphism for every $\yyy'\in X_*(\TT)$.
\item We have seen above that $f_{\yyy}(\AAF)(K(\Gm))\subset K(\TT).$ Now the claim follows from Part (1).
\item This is immediate from the commutativity of Diagram (\ref{oppm}).    
\end{enumerate}
\end{proof}
The exact sequence $$1\to\TT(\AAF)_1\to\TT(\AAF)\xrightarrow{\log_{\TT}} X_*(\TT)_{\RR}\to 1, $$ together with the isomorphism $\nnn_{\RR}^{-1}:X_*(\TT)_{\RR}\xrightarrow{\sim}X^*(\TT)_{\RR}$  induces the exact sequence
$$1\to X^*(\TT)_{\RR}\to \TT(\AAF)^*\to(\TT(\AAF)_1)^{\perp}[\TT(\AAF)^* ]\to 1.$$
\begin{lem}\label{obrikmmm}
 Let~$\mmm\in X^*(\TT)_{\RR}$. Let $(\yyy, x)\in X_*(\TT)$. Consider~$\chi_{\mmm}\in\TT(\AAF)^*$ given by $$\chi_{\mmm}:=\xxx\mapsto\exp(i\langle \log_{\TT}(\xxx),\mmm\rangle).$$ Then $$\chi_{\mmm}^{(\yyy,x)}=\chi_{\langle\yyy,\mmm\rangle}=\chi_{\mmm}^{(\yyy)}.$$
In particular, if $\theta\in\TT(\AAF)^*$ vanishes on $\TT(\AAF)_1$, then $\theta^{(\yyy,x)}$ vanishes on $\Gm(\AAF)_1$. 
\end{lem}
\begin{proof}
 For every place~$v$, we have established in Lemma \ref{locexp}  that $$\eta:\xxx\mapsto \exp(i\cdot\log(q_v)\cdot \langle\log_{\TT,v}(\xxx),\mmm\rangle)\in \TT(F_v)^*,$$ satisfies that $$\eta^{(\yyy,x)}=y\mapsto \exp(i\cdot\log(q_v)\cdot\langle\log_{\Gm,v}(y),\langle\yyy,\mmm\rangle\rangle) =\eta^{(\yyy)}.$$
Now, for $(z_v)_v\in\Gm(\AAF)$, we have that
\begin{align*}
\chi_{\mmm}^{(\yyy,x)}((z_v)_v)&=\exp\big(i\cdot\langle\log_{\TT} (f_{(\yyy,x)}((z_v)_v)),\mmm\rangle\big)\\
&=\exp\big(i\cdot\big\langle \sum_{\vMF}\log(q_v)\log_{\TT,v}(f_{(\yyy,x)}(z_v)),\mmm\big\rangle\big)\\
&=\prod_{\vMF}\exp\big(i\cdot\big\langle \log(q_v)\log_{\TT,v}(f_{(\yyy,x)}(z_v)),\mmm\big\rangle\big)\\
&=\prod_{\vMF}\exp\big(i\cdot\log(q_v)\cdot\langle \log_{\Gm,v}(z_v), \langle \yyy,\mmm\rangle\rangle\big)\\
&=\prod_{\vMF}\exp\big(i\cdot\log(q_v)\cdot\langle \log_{\Gm}((z_v)_v),\langle \yyy,\mmm\rangle\rangle\big)\\
&=\chi_{\mmm}^{(\yyy)}((z_v)_v).
\end{align*} 
The claim is proven.
\end{proof}
\subsection{Splittings}
We will give a spliting of the exact sequence $$1\to\TT(\AAF)_1\to\TT(\AAF)\to X_*(\TT)_{\RR}\to 1.$$
This will enable us to choose a subgroup~$\mathfrak A_{\TT}$ of $\TT(\AAF)^*$ having certain properties for its components.  
\subsubsection{} For every split stacky $F$-torus $\TT=\Gm^d\times BG$ and every~$\vMFi$, we will define a continuous homomorphism $n_{\TT,v}:X_*(\TT)_{\RR}\to \TT(F_v).$ 
First, we suppose that $\TT=\Gm^d$ for certain $d\geq 0$. 
We set \begin{align*}n_{\Gm^d,v}:\RR^d=X_*(\Gm^d)_{\RR}&\to\Gm^d(F_v)=\Gm(F_v)^d,\\ (x_i)_{i=1}^d&\mapsto (\exp(x_i)\cdot [F:\QQ]^{-1})_{i=1}^d.
\end{align*}
For $\xxx\in X_*(T)$, one has that \begin{align*}
\log_{\TT,v}(n_{\TT,v}(\xxx))&=\log_{\TT,v}((\exp(x_i)\cdot [F:\QQ]^{-1})_{i=1}^d)\\
&=(x_i\cdot [F:\QQ]^{-1})_{i=1}^d\\
&=\xxx\cdot[F:\QQ]^{-1}.
\end{align*} 
Suppose now that $\TT=\Gm^d\times BG$ for certain $d\geq 0$ and certain diagonalizable~$G$. For~$\vMFi,$ we have a section to the projection $\TT(F_v)\to \Gm^d(F_v)$ given by $\yyy\mapsto (\yyy,1)$. Now, for every $\vMFi$, we define a homomorphism 
$$n_{\TT,v}:X_*(\TT)_{\RR}\to \TT(F_v),\hspace{1cm} n_{\TT,v}:=n_{\Gm^d,v}\circ (\yyy\mapsto (\yyy,1))$$ which is clearly continuous.
For every $\xxx\in X_*(\TT)_{\RR}$, one has that 
\begin{align*}
\log_{\TT,v}(n_{\TT,v}(\xxx))&=\log_{\TT,v}((\exp(x_i)\cdot[F:\QQ]^{-1})_{i=1}^d,1)\\
&=\log_{\Gm^d,v}((\exp(x_i)\cdot [F:\QQ]^{-1})_{i=1}^d)\\
&=(\exp(x_i)\cdot[F:\QQ]^{-1})_{i=1}^d\\
&=\xxx[F:\QQ]^{-1}.
\end{align*}
\begin{lem}\label{fynt}
Let $\TT=\Gm^d\times BG$ be a split stacky torus. Let~$\yyy\in X_*(\Gm^d)=\ZZ^d$ and let~$\vMFi$. The following diagram is commutative:
$$
\begin{tikzcd}
\RR\arrow{r}{n_{\Gm,v}}\arrow{d}{(f_{\yyy})_{\RR}}& \Gm(F_v)\arrow{d}{\widetilde{f_{\yyy}}(F_v)}\\
X_*(\TT)_{\RR}\arrow{r}{n_{\TT,v}}&\TT(F_v).
\end{tikzcd}
$$
\end{lem}
\begin{proof}
The image of~$x\in\RR$ for the map $\widetilde{f_{\yyy}}(F_v)\circ n_{\Gm,v}$ is $$\widetilde{f_{\yyy}}(n_{\Gm,v}(x))=\widetilde{f_{\yyy}}(\exp(x)\cdot [F:\QQ]^{-1})=(\exp(y_ix)\cdot[F:\QQ]^{-1})_{i=1}^d.$$
On the other hand, the image of~$x$ for the map $n_{\TT,v}\circ (f_{\yyy})_{\RR}$ is $$n_{\TT,v}((f_{\yyy})_{\RR})(x)=n_{\TT,v}(y_1x\doots y_dx)=(\exp(y_ix)\cdot [F:\QQ]^{-1})_{i=1}^d,$$  hence the diagram commutes.
\end{proof}
\subsubsection{} For every split stacky torus $\TT=\Gm^d\times BG$, we will define a continuous section $n_{\TT}:X_*(\TT)\to \TT(\AAF)$ to the surjection $\log_{\TT}:\TT(\AAF)\to X_*(\TT)_{\RR}$. First, we suppose that $\TT=\Gm^d$ for certain $d\geq 0$. We define $$n_{\Gm^d}:X_*(\TT)_{\RR}\to \TT(\AAF),\hspace{1cm}\xxx\mapsto ((1)_{\vMFz},(n_{\Gm^d,v}(\xxx))_{\vMFi}).$$
Let us verify that~$n_{\Gm^d}$ is indeed a section to~$\log_{\Gm^d}$. For $\xxx\in X_*(\Gm^d)_{\RR}$, one has that:
\begin{align*}\log_{\Gm^d}(n_{\Gm^d}(\xxx))&=\sum_{\vMFi}\log(q_v)\cdot \log_{\Gm^d,v}(n_{\Gm^d,v}(\xxx))\\
&=\sum_{\vMFi}\log(q_v)\cdot\xxx\cdot [F:\QQ]^{-1}\\
&=\xxx\cdot\frac{1}{[F:\QQ]}\sum_{\vMFi}\log(q_v)\\
&=\xxx,
\end{align*}
where we recall that for $\vMFi$, we have defined $q_v=\exp(1)$ if~$v$ is real and $q_v=\exp(2)$ if~$v$ is complex. Suppose now that $\TT=\Gm^d\times BG$ for certain $d\geq 0$ and certain diagonalizable~$G$. The projection $\TT(\AAF)\to \Gm^d(\AAF)$ admits a section $\yyy\mapsto (\yyy,1)$. We define $$n_{\TT}:=n_{\Gm^d}\circ (\yyy\mapsto (\yyy,1)),$$ which is clearly continuous.
Clearly, the map~$n_{\TT}$ is a continuous section to~$\log_{\TT}$.
\begin{lem}
Let $\TT=\Gm^d\times BG$ be a split stacky torus. Let $\yyy\in X_*(\Gm^d)=\ZZ^d$. The following diagram is commutative:
$$
\begin{tikzcd}
\RR\arrow{r}{n_{\Gm}} \arrow{d}{(f_{\yyy})_{\RR}}&\Gm(\AAF)\arrow{d}{\widetilde{f_{\yyy}}(\AAF)}\\
X_*(\TT)_{\RR}\arrow{r}{n_{\TT}}&\TT(\AAF). 
\end{tikzcd}
$$
\end{lem}
\begin{proof}
Let $x\in\RR$. It follows from Lemma~\ref{fynt}, that
\begin{align*}
\widetilde{f_{\yyy}}(\AAF)(n_{\Gm}(x))&=\widetilde{f_{\yyy}}(\AAF)((1)_{\vMFz}, (n_{\Gm,v}(x))_{\vMFi})\\
&=((\widetilde f_{\yyy}(1))_{\vMFz}, (\widetilde f_{\yyy}(n_{\Gm,v}(x)))_{\vMFi})\\
&=((1)_{\vMFz}, (n_{\TT,v}((f_{\yyy})_{\RR}(x)))_{\vMFi})\\
&=n_{\TT}((f_{\yyy})_{\RR}(x)).
\end{align*}
The claim is verified.
\end{proof}
\subsubsection{}For a split stacky $F$-torus $\TT=\Gm^d\times BG$ 
and every $\yyy\in X_*(\Gm^d),$ we have a commutative diagram 
$$
\begin{tikzcd}
1\arrow{r}{}&\Gm(\AAF)_1\arrow{r}{}\arrow{d}{\widetilde {f_{\yyy}}(\AAF)|_{\Gm(\AAF)_1}}&\Gm(\AAF)\arrow{r}{\log_{\Gm}}\arrow{d}{\widetilde {f_{\yyy}}(\AAF)}& \RR\arrow{d}{(f_{\yyy})_{\RR}}\arrow{r}{}& 1\\
1\arrow{r}{}&\TT(\AAF)_1\arrow{r}{}&\TT(\AAF)\arrow{r}{\log_{\TT}}&X_*(\TT)_{\RR}\arrow{r}{}& 1.
\end{tikzcd}
$$
and its dual
$$
\begin{tikzcd}
1\arrow{r}{}&\RR^*\arrow{r}{\log_{\Gm}^*}&\Gm(\AAF)^*\arrow{r}{}&(\Gm(\AAF)_1)^*\arrow{r}{}& 1\\
1\arrow{r}{}&(X_*(\TT))^*\arrow{r}{\log_{\TT}^*}\arrow{u}{(f_{\yyy})_{\RR}^*}&\TT(\AAF)^*\arrow{r}{}\arrow{u}{\widetilde {f_{\yyy}}(\AAF) ^*}&(\TT(\AAF)_1)^*\arrow{r}{} \arrow{u}[right]{(\widetilde {f_{\yyy}}(\AAF)|_{\Gm(\AAF)_1})^*}& 1.
\end{tikzcd}
$$
The section $n_{\TT}:X_*(\TT)_{\RR}\to \TT(\AAF)$ to the surjection $\log_{\TT}:\TT(\AAF)\to X_*(\TT)_{\RR}$ induces a retraction $$r_{\TT}:\TT(\AAF)\to \TT(\AAF)_1,\hspace{1cm} \xxx\mapsto \xxx\cdot n_{\TT}(\log_{\TT}(\xxx))^{-1}.$$
The diagram 
$$
\begin{tikzcd}
\Gm(\AAF)\arrow{r}{r_{\Gm}}\arrow{d}{\widetilde {f_{\yyy}}(\AAF)}&\Gm(\AAF)_1 \arrow{d}{\widetilde {f_{\yyy}}(\AAF)|_{\Gm(\AAF)_1}}\\
\TT(\AAF)\arrow{r}{r_{\TT}}&\TT(\AAF)_1 
\end{tikzcd}
$$
 is commutative, because for $\xxx\in\Gm(\AAF)$, one has that 
 \begin{align*}
 r_{\TT} (\widetilde {f_{\yyy}}(\AAF)(\xxx))&= (\widetilde {f_{\yyy}}(\AAF)(\xxx))\cdot n_{\TT}(\log_{\TT}(\widetilde {f_{\yyy}}(\AAF)(\xxx)))^{-1}\\
&=(\widetilde {f_{\yyy}}(\AAF)(\xxx))\cdot n_{\TT}((f_{\yyy})_{\RR}(\log_{\Gm}(\xxx)))^{-1}\\
&=(\widetilde {f_{\yyy}}(\AAF)(\xxx))\cdot (\widetilde {f_{\yyy}}(\AAF)(n_{\Gm}(\log_{\Gm}(\xxx))))^{-1}\\
&=(\widetilde {f_{\yyy}}(\AAF))(\xxx\cdot n_{\Gm}(\log_{\Gm}(\xxx))^{-1})\\
&=(\widetilde {f_{\yyy}}(\AAF))(r_{\Gm}(\xxx)).
\end{align*}
The map $r_{\TT}^*:\TT(\AAF)_1^*\to\TT(\AAF)^*$ is a continuous section to the map $\TT(\AAF)^*\to \TT(\AAF)_1^*$ given by the restriction $\chi\mapsto\chi|_{\TT(\AAF)_1}$. The diagram
$$
\begin{tikzcd}
(\TT(\AAF)_1^*)\arrow{r}{r_{\TT}^*}\arrow{d}[left]{(\widetilde {f_{\yyy}}(\AAF)|_{\Gm(\AAF)_1})^*}& \TT(\AAF)^*\arrow{d}{\widetilde f_{\yyy}(\AAF)^*}\\
\Gm(\AAF)_1^*\arrow{r}{r_{\Gm}^*}&\Gm(\AAF)^*
\end{tikzcd}
$$
is commutative. The following corollary is now immediate.
\begin{cor} \label{charate}Let~$\TT=\Gm^d\times BG$ be a split stacky $F$-torus.
 Let $\yyy\in X_*(\Gm^d)$. Suppose that $\chi_{0}\in(\TT(\AAF)_1)^*$ satisfies that $$\chi_0^{(\yyy)}:= (\widetilde{f_{\yyy}}(\AAF)|_{\Gm(\AAF)_1})^*(\chi_0)=1\in(\Gm(\AAF)_1)^*.$$ Then $$(r_{\TT}^*(\chi_0))^{(\yyy)}= (\widetilde{f_{\yyy}}(\AAF))^*(r_{\TT}^*(\chi_0))=1\in\Gm(\AAF)^*.$$
\end{cor}
\begin{mydef}\label{charat}
We define a subgroup of $\mathfrak A_{\TT}\subset(K(\TT)\Delta(\TT(F)))^{\perp}[\TT(\AAF)^*]$ by $$\mathfrak A_{\TT}:=r_{\TT}^*\big((K(\TT)\Delta(\TT(F)))^{\perp}[\TT(\AAF)_1^*]\big).$$
\end{mydef}
For a stacky torus $\TT=T\times BG$, the decomposition $\TT(\AAF)=\Gm^d(\AAF)\times BG(\AAF)$ induces an identification 
\begin{equation*}\label{atagmabl}
\mathfrak A_{\TT}=\mathfrak A_{T}\times \mathfrak A_{BG}.
\end{equation*}
The group~$\mathfrak A_{\TT}$ is finitely generated. (Indeed on one hand the group~$\mathfrak A_{BG}$ is finite, because isomorphic to the finite group $(K(BG))\Delta(BG(F)))^{\perp}[BG(\AAF)^*]=(BG(\AAF)/(K(BG)\Delta(BG(F))))^*.$ On the other hand~$\mathfrak A_{T}$ is finitely generated by \cite[Lemma 4.5.2(1, 3)]{Bourqui}).  We terminate this section by proving the following claim.
\begin{lem}\label{finofbchi}
 Let $D\subset X_*(\TT)$ be a finite index subgroup. One has that $$B(D):=\{\chi\in\mathfrak A_{\TT}|\hspace{0,1cm}\forall\yyy\in D: \chi^{(\yyy)}=1\in \Gm(\AAF)^*\}$$ is finite.
\end{lem}
\begin{proof}
Let us first show that it suffices to show the claim for $D=X_*(\TT)$. Set~$k=[X_*(\TT):D]$, so that $D\supset kX_*(\TT)$. 
 For every $\chi\in B(D)$ and every~$\yyy\in X_*(\TT)$, one has that $(\chi^{k})^{(\yyy)}=\chi^{(k\yyy)}$ by Corollary~\ref{obrik}(2). We deduce that~$B(D)$ is contained in the preimage of~$B(X_*(\TT))$ of the homomorphism of finitely generated abelian groups $$\mathfrak A_{\TT}\to\mathfrak A_{\TT},\hspace{1cm}\chi\mapsto \chi^k.$$ The homomorphism is of finite kernel, hence if~$B(X_*(\TT))$ is finite, then~$B(D)$ is finite. 

Let us now prove the claim for $D=X_*(\TT)$. Let $\yyy_1\doots \yyy_{d}$ be a set of linerly independent generators of~$X_*(\TT).$ The following diagram is commutative
$$
\begin{tikzcd}
1\arrow{r}{}&\Gm^d(\AAF)_1\arrow{r}{}\arrow{d}{(\widetilde{f_{\yyy_i}}(\AAF))_{i=1}^d}&\Gm^d(\AAF)\arrow{r}{\log_{\Gm^d}}\arrow{d}{(\widetilde{f_{\yyy_i}}(\AAF))_{i=1}^d}&\RR^d\arrow{d}{(x_i)_{i=1}^d\mapsto \sum_{i=1}^d x_i\yyy_i}\arrow{r}{}&1\\
1\arrow{r}{}& \TT(\AAF)_1\arrow{r}{}\arrow{d}{}&\TT(\AAF)\arrow{r}{\log_{\TT}}\arrow{d}{p_{T}}&X_*(\TT)_{\RR}\arrow{r}{}\arrow{d}{=}&1\\
1\arrow{r}{}&T(\AAF)_1\arrow{r}{}&T(\AAF)\arrow{r}{}&X_*(\TT)_{\RR}\arrow{r}{}&1,
\end{tikzcd}
$$
where the first two lower vertical maps are the projections. Now for a character $\chi\in\mathfrak A_{\TT}$ such that $\chi^{(\yyy)}=\chi\circ (\widetilde{f_{\yyy_i}}(\AAF))_{i=1}^d =1$, one has that $\chi|_{T(\AAF)}\circ p_T\circ (\widetilde{f_{\yyy_i}}(\AAF))_{i=1}^d =1.$ The kernel of the restriction map $$\mathfrak A_{\TT}\to\mathfrak A_{T},\hspace{1cm} \chi\mapsto \chi|_{T(\AAF)}$$is the finite group~$\mathfrak A_{BG}$. It suffices hence to prove that the set of $\chi\in\mathfrak A_{T}$ such that $\chi\circ p_T\circ (\widetilde{f_{\yyy_i}}(\AAF))_{i=1}^d =1$ is finite. But, the homomorphism $$(p_T\circ (\widetilde{f_{\yyy_i}}(\AAF))_{i=1}^d):\Gm^d(\AAF)\to T(\AAF)$$is an isomorphism of locally compact abelian groups as it is the induced morphism on the adelic points of the isomorphism of the tori $\Gm^d\xrightarrow{\sim} T$ defined by the isomorphism  of the cocharacter groups$$\ZZ^d\to X_*(\TT),\hspace{1cm}(x_i)_{i=1}^d\mapsto\sum_{i=1}^dx_i\yyy_i.$$ Hence, the set of such~$\chi\in \mathfrak A_{T}$ is the singleton. The proof is now completed.
%
\end{proof}
\subsection{More on characters}
We will define measures on character groups as well as norm of characters.
\subsubsection{} Let $\TT=T\times BG$ be an $F$-split stacky torus.
\begin{lem}\label{normmeas}
\begin{enumerate}
\item One has an exact sequence 
\begin{multline*}1\to (X^*(\TT)_{\RR},(2\pi)^{-\dim(\TT)}d\mmm)\to ( \Delta(\TT(F))^{\perp}[\TT(\AAF)^*], (\mu/\Delta)^*)\to\\ \to (\Delta(\TT(F))^{\perp}[\TT(\AAF)_1^*],(\# G^D)^{-1}\cdot \mu^{\#}_{\Delta(\TT(F))^{\perp}[\TT(\AAF)_1^*]})\to 1,
\end{multline*}
where the first homomorphism is given by
$\mmm\mapsto \chi_{\mmm}$.
\item For $\chi\in \Delta(\TT(F))^{\perp}[\TT(\AAF)^*]$, one has that $\chi\in (K(\TT)\Delta(\TT(F)))^{\perp}[\TT(\AAF)^*] $ if and only if the image of~$\chi$ for the homomorphism from Part (1) lies in $(K(\TT)\Delta(\TT(F)))^{\perp}[\TT(\AAF)_1^*]$. The group $(K(\TT)\Delta(\TT(F)))^{\perp}[\TT(\AAF)^*]$ is an open subgroup of $(\Delta(\TT(F)))^{\perp}[\TT(\AAF)^*]=(\TT(\AAF)/\Delta(\TT(F)))^*.$ 
\item The image of~$X^*(\TT)_{\RR}$ for the above homomorphism lies in the subgroup $(K(\TT)\Delta(\TT(F)))^{\perp}[\TT(\AAF)^*]$. We deduce an exact sequence
\begin{multline}
1\to (X^*(\TT)_{\RR}, {(2\pi)^{-\dim(\TT)}}{d\mmm})\to (\Delta(\TT(F))K(\TT))^{\perp}[\TT(\AAF)^*], (\mu/\Delta)^*)\to\\
\to (\Delta(\TT(F))K(\TT))^{\perp}[\TT(\AAF)_1^*], (\# G^D)^{-1}\cdot\mu^{\#}_{\Delta(\TT(F))K(\TT))^{\perp}[\TT(\AAF)_1^*]}) \to 1.
\end{multline}
\end{enumerate}
\end{lem}
\begin{proof}
\begin{enumerate}
\item Applying \cite[Chapter II, \S 1, $\text{n}^{\circ}$8, Proposition 9]{TSpectrale} to the exact sequence $$1\to  (\TT(\AAF)_1/\Delta(\TT(F)), (\mu_1/\Delta)) \to (\TT(\AAF)/\Delta(\TT(F)), (\mu/\Delta)) \to (X(\TT)_{\RR},d\xxx)\to 1,$$ we obtain that 
\begin{multline*}
1\to (X(\TT)^*_{\RR},d\xxx^*) \to (\Delta(\TT(F))^{\perp}[\TT(\AAF)^*], (\mu/\Delta)^*) \to \\ \to (\Delta(\TT(F))^{\perp}[\TT(\AAF)_1^*], (\mu_1/\Delta)^*)\to 1\end{multline*} is exact. The group $(\TT(\AAF)_1/\Delta(\TT(F))$ is compact and by combining Proposition \ref{ukmar} and \cite[Chapter II, \S 1, $\text{n}^{\circ}9,$ Proposition 11]{TSpectrale}, we obtain that
$$(\mu_1/\Delta)^*=(\# G^D)^{-1}\mu^{\#}_{\Delta(\TT(F))^{\perp}[\TT(\AAF)_1^*]}.$$
Recall that in Lemma~\ref{djurdjule}, we have established an identification of pairs $$(X^*(\TT)_{\RR}, {(2\pi)^{-\dim(\TT)}}d\mmm) \xrightarrow{\sim} (X(\TT)^*, d\xxx^*).$$ The statement follows.
\item The 
first claim is immediated as $K(\TT)\subset \TT(\AAF)_1.$ 
The group 
$$(\TT(\AAF)_1/\Delta(\TT(F)))^*=\Delta(\TT(F))^{\perp}[\TT(\AAF)^*_1]$$ is discrete. Thus $$(K(\TT)\Delta(\TT(F)))^{\perp}[\TT(\AAF)^*] $$ is open being the preimage of the open $$(K(\TT)\Delta(\TT(F)))^{\perp}[\TT(\AAF)_1^*] \subset \Delta(\TT(F))^{\perp}[\TT(\AAF)_1^*]$$ for the restriction homomorphism $$(\Delta(\TT(F)))^{\perp}[\TT(\AAF)^*]\to (\Delta(\TT(F)))^{\perp} [\TT(\AAF)^*_1].$$
\item We observe that a character in the image of the first homomorphism vanishes on~$\TT(\AAF)_1$, hence on~$K(\TT)$, thus the image of the homomorphism is contained in the group $(\Delta(\TT(F))K(\TT))^{\perp}[\TT(\AAF)^*]$. Now, as a character $\chi\in \Delta(\TT(F))^{\perp}[\TT(\AAF)^*]$ is contained in $(\Delta(\TT(F))K(\TT))^{\perp}[\TT(\AAF)^*]$ if and only if its image in $(\Delta(\TT(F)))^{\perp}[\TT(\AAF)_1^*]$ is contained in the group $(\Delta(\TT(F))K(\TT))^{\perp}[\TT(\AAF)_1^*]$ and by the exactness of the sequence from Part (1), we obtain that our sequence is exact. 
\end{enumerate}
\end{proof}
\begin{mydef}
We define a homomorphism $$m: (K(\Gm)\Delta(\Gm(F)))^{\perp}[\Gm(\AAF)^*]\to\RR$$ by setting it to be the composite of the homomorphism $$n_{\Gm}^*|_{(K(\Gm)\Delta(\Gm(F)))^{\perp}[\Gm(\AAF)^*]}:(K(\Gm)\Delta(\Gm(F)))^{\perp}[\Gm(\AAF)^*]\to \RR^*$$ and the inverse of the isomorphism $$ \RR\xrightarrow{\sim} \RR^* ,\hspace{1cm} z\mapsto (y\mapsto \exp(iyz)).$$
\end{mydef}
The map~$m$ is a retraction of the injection $$\RR\to (K(\Gm)\Delta(\Gm(F)))^{\perp}[\Gm(\AAF)^*],\hspace{1cm}t\mapsto \chi_t,$$ 
because the map $t\mapsto \chi_t=(x\mapsto\exp(it\log_{\Gm}(x)))$ is the composite $$(t\mapsto\chi_t)=(\log_{\Gm})^*\circ (z\mapsto (y\mapsto \exp(iyz))) $$ and~$n_{\Gm}^*|_{(K(\Gm)\Delta(\Gm(F)))^{\perp}[\Gm(\AAF)^*]}$ is a retraction of the injection $$\log_{\Gm}^*:\RR^*\to (K(\Gm)\Delta(\Gm(F)))^{\perp}[\Gm(\AAF)^*].$$
The following corollary contains a property of~$\mathfrak A_{\TT}$ which we will be using in Proposition~\ref{growthofg}.
\begin{cor}\label{funat}
For every $(\yyy,x)\in X_*(\TT)$ and every $\chi\in\mathfrak A_{\TT}$, one has that $$\chi^{(\yyy,x)}|_{\Gm(\AAF)_1}=1\implies m(\chi^{(\yyy,x)})=0.$$
\end{cor}
\begin{proof}
Let $\chi\in\mathfrak A_{\TT}$. 
Let~$n$ be a strictly positive integer such that $nx=0$. One has that $m(\chi^{(\yyy,x)})=\frac{m(\chi^{(n\yyy,0)})}{n}$ is equal to~$0$ if and only if $m(\chi^{(n\yyy)})=m(\chi^{(n\yyy,0)})=0$. But $\chi^{(n\yyy)}=\chi^{(n\yyy,0)}$ vanishes on $\Gm(\AAF)_1$ because so does~$\chi^{(\yyy, x)},$ hence Corollary~\ref{charate} gives $\chi^{(n\yyy)}=1$. In particular, one has $m(\chi^{(n\yyy)})=0$, hence $m(\chi^{(\yyy,x)})=0$, as claimed.
\end{proof}
\subsubsection{}Let us define and prove certain properties of norms of characters.
For a character $$\chi\in(K(\TT)\Delta(\TT(F)))^{\perp}[\TT(\AAF)^*]\to\prod_{\vMFi}X^*(\TT)_{\RR},$$we define $$\chi_{\infty}=(\chi_v)_v\in(\TT(\Ov))^{\perp}[\TT(F_v)^*]=\prod_{\vMFi}X^*(T)_{\RR}.$$
We fix a norm~$||\cdot||$ on~$X^*(\TT)$. We use the same notation for the induced norm on~$\prod_{\vMFi}X^*(\TT)$.
\begin{mydef}
For a character $$\chi\in(K(\TT)\Delta(\TT(F)))^{\perp}[\TT(\AAF)^*],$$ we define $$||\chi||:=||\chi_{\infty}||.$$
\end{mydef}
Clearly, for $\ttt\in X^*(\TT)$ the character $\chi_{\ttt}=(\zzz\mapsto \exp(i\langle \log_{\TT}(\zzz),\ttt\rangle))$satisfies that
\begin{equation}
||\chi_{\ttt}||:=||\ttt||.
\end{equation}
For any two $\chi, \chi'\in(K(\TT)\Delta(\TT(F)))^{\perp}[\TT(\AAF)^*]$, it is clear that $$||\chi\chi'||\leq ||\chi||+||\chi'||.$$
\begin{lem}\label{rama}
Let $S\subset X_*(\TT)$ be finite. For $\chi\in (K(\TT)\Delta(\TT(F)))^{\perp}[\TT(\AAF)^*]$, one has that $$\max_{\yyy\in S}||\chi^{(\yyy)})||\ll ||\chi||.$$
\end{lem} 
\begin{proof}
We have a commutative diagram:
$$
\begin{tikzcd}
(K(\TT)\Delta(\TT(F)))^{\perp}[\TT(\AAF)^*]\arrow{r}{\chi\mapsto\chi_{\infty}} \arrow{d}{(\widetilde{f_{\yyy}}(\AAF)^*)_{\yyy\in S}}& \prod_{\vMFi}X_*(\TT)_{\RR}\arrow{d}{\prod_{\vMFi}\Hom((f_{\yyy})_{\yyy\in S},\RR)}\\
(K(\Gm^{S})\Delta(\Gm^{S}(F)))^{\perp}[\Gm^S(\AAF)^*]\arrow{r}{\chi\mapsto\chi_{\infty}}&\prod_{\vMFi}\RR^{S}.\\
\end{tikzcd}
$$
For any $\chi\in (K(\TT)\Delta(\TT(F)))^{\perp}[\TT(\AAF)^*]$, we deduce that $$\max_{\yyy\in S}||\chi^{(\yyy)}||=||((\widetilde {f_{\yyy}}(\AAF)^*)_{\yyy\in S}(\chi) ||\leq \big|\big|\prod_{\vMFi}\Hom((f_{\yyy})_{\yyy\in S},\RR)\big|\big|\cdot ||\chi_{\infty}||\ll ||\chi||.$$
\end{proof}
\section{Heights} \label{Heights}
For any fan~$\Sigma$ and any integer $k\geq 0$, we denote by $\Sigma(k)$ the set of faces of~$\Sigma$ of dimension~$k$. If~$\sigma$ is a face of~$\Sigma$, for an integer $k\geq 0$, we denote by~$\sigma(k)$ the subset of faces in~$\Sigma(k)$ which are contained in~$\sigma$.

Let~$\bSigma=(\Sigma, N, \beta)$ be a simplicial stacky fan.  That is the datum of a finitely generated abelian group~$N$, a simplicial fan~$\Sigma$ in~$N_{\QQ}$ and a homomorphism $\beta:\ZZ^{\Sigma(1)}\to N$ of cofinite image. 
Let~$\mathcal X$ be the associated toric stack. Based on our previous work \cite{dardayasudabm}, we will define heights on~$\XX$ and prove certain of its properties, notably the Northcott property. We note that heights in this article are slightly more general. 

Let us define $N^{\rig}:=N/N_{\tor}$, where~$N_{\tor}$ denotes the torsion of~$N$. A choice of section $N^{\rig}\to N$ induces a splitting $N=N^{\rig}\oplus N_{\tor}$ that we fix. We will often identify $N^{\rig}=N^{\rig}\times\{0\}$ and \(N^{\rig}_\QQ=N_\QQ\).  The stack~$\XX$ posses a split $F$-stacky torus~$\TT=T\times BG,$ such that $X_*(T)=N^{\rig}$. 
Denote by~$\beta^{\rig}$ the composite homomorphism $$\beta^{\rig}:\ZZ^{\Sigma(1)}\xrightarrow{\beta} N\to N^{\rig}.$$  For~$\rho\in\Sigma(1)$, we set $$b_{\rho}:={\beta^{\rig}((1)_{j=\rho}, (0)_{j\neq \rho})}\in N^{\rig}.$$
The triple \(\bSigma^{\rig}=(\Sigma, N^{\rig},\beta^{\rig})\) is again a stacky fan. We define \(\cX^{\rig}\) to be the toric stack associated to \(\bSigma^{\rig}\). This stack has a trivial generic stabilizer and admits a canonical morphism \(\cX \to \cX^{\rig}\). The stack \(\cX^{\rig}\) together with the morphism \(\cX \to \cX^{\rig}\) is nothing but the rigidification of \(\cX\) with respect to \(d\)-dimensional components of the intertia stack of \(\cX \). 

\subsection{Twisted sectors via fans} The goal of this subsection is to express the twisted sectors \cite[Definition 2.6]{dardayasudabm} of~$\XX$ by stacky fans. 
\subsubsection{} For $\sigma\in\Sigma$ and element $\yyy\in N^{\rig}\cap \sigma$, by the fact that~$\sigma$ is simplicial, for~$\rho\in\sigma(1)$,  there exists a unique $a^{\sigma}_{\rho}(\yyy)\in \QQ_{\geq 0}$ such that $$\yyy=\sum_{\rho\in\sigma(1)}a^{\sigma}_{\rho}(\yyy)b_{\rho}.$$
\begin{mydef} 
 For $\sigma\in \Sigma$, we define: 
\begin{align*}
\Boxx(\sigma)&:=\{\xxx\in N\big|\hspace{0,1cm}\forall\rho\in\sigma(1): a^{\sigma}_{\rho}(\xxx)\in [0, 1[\},\\
\Boxx^{\rig}(\sigma)&:= N^{\rig}\cap\Boxx(\sigma),\\
\Boxx (\Sigma)&:=\bigcup_{\sigma\in\Sigma(d)}\Boxx(\sigma),\\
\Boxx^{\rig}(\Sigma) &:=\bigcup_{\sigma\in\Sigma(d)}\Boxx^{\rig}(\sigma).
\end{align*}
\end{mydef}
\begin{lem}
Let $\sigma_1,\sigma_2\in\Sigma.$ For any $\rho\in\sigma_1(1)\cap\sigma_2(1)$,  the maps $a^{\sigma_1}_{\rho}:N^{\rig}\cap \sigma_1\to \QQ_{\geq 0}$ and $a^{\sigma_2}_{\rho}:N^{\rig}\cap\sigma_2\to \QQ_{\geq 0}$ coincide on the intersection $N^{\rig}\cap\sigma_1\cap\sigma_2.$ 
\end{lem}
\begin{proof}
Let $\xxx\in N^{\rig}\cap \sigma_1\cap\sigma_2$.  We have that $$\xxx=\sum_{\rho\in\sigma_1(1)}a^{\sigma_1}_{\rho}(\xxx)\cdot b_{\rho}=\sum_{\rho\in\sigma_2(1)}a^{\sigma_2}_{\rho}(\xxx)\cdot b_{\rho}.$$ As $\xxx\in\langle b_{\rho}|\hspace{0,1cm}\rho\in\sigma_1\cap\sigma_2\rangle$ because $\sigma_1\cap\sigma_2$ is simplicial, one must have for every~$\rho$ contained in~$\sigma_2$ and not in~$\sigma_1$ that $a_{\rho}^{\sigma_2}(\xxx)=0$ and for every~$\rho$ contained in~$\sigma_1$ and not in~$\sigma_2$ that $a_{\rho}^{\sigma_1}(\xxx)=0$. The vectors $\{b_{\rho}\}_{\rho\in\sigma_1(1)\cap\sigma_2(1)}$ are linearly independent as $\sigma_1\cap\sigma_2$ is simplicial, and thus $a^{\sigma_2}_{\rho}(\xxx)=a^{\sigma_1}_{\rho}(\xxx)$ for any~$\rho\in\sigma_1(1)\cap\sigma_2(1)$. The claim is proven. 
\end{proof}
The lemma enables us to give the following definition.
\begin{mydef}\label{defofqr}
For~$\rho\in\Sigma(1)$ and~$\xxx\in N^{\rig}$, we define $$a_{\rho}(\xxx):=\begin{cases}
 a_{\rho}^{\sigma}(\xxx) ,& \text{ if there exists a face $\sigma\in\Sigma$ such that $\xxx\in\sigma$ and $\rho\in \sigma(1)$,} \\
0, & \text{ otherwise.}
\end{cases}
$$
For $(\xxx, g)\in N^{\rig}\times G^D$, we set $a_{\rho}(\xxx,g)=a_{\rho}(\xxx).$
 We define a map $$q:N^{\rig}\to\Boxx^{\rig}(\Sigma),\hspace{1cm}\xxx\mapsto \sum_{\rho\in\Sigma(1)}\{a_{\rho}(\xxx)\}b_{\rho}.$$
We define a map, which by abuse of notation we also denote by~$q$:
$$q:N^{\rig}\times G^D\to \Boxx(\Sigma)=\Boxx^{\rig}(\Sigma)\times G^D,\hspace{1cm} (\xxx, g)\mapsto (q(\xxx), g).$$
We also define a map $$r:N^{\rig}\to N^{\rig},\hspace{1cm}\xxx\mapsto\xxx-q(\xxx)=\sum_{\rho\in\Sigma(1)}\lfloor a_{\rho} \rfloor b_{\rho}.$$
\end{mydef}

\begin{lem}
We have  identifications $\Boxx^{\rig}(\Sigma)=\pi_0(\mathcal J_0\XX^{\rig})$ and $\Boxx(\Sigma)=\pi_0(\mathcal J_0\XX).$
\end{lem}

\begin{proof}
This was proved in \cite[Proposition 4.7]{borisovchensmith} over complex numbers. We do not need change anything over \(\overline{F}\). Since our toric stacks are split, the  natural actions of \(\mathrm{Gal}(\overline{F}/F)\) on the sets $\Boxx^{\rig}(\Sigma)$, $\pi_0(\mathcal J_0\XX^{\rig}_{\overline{F}})$, $\Boxx(\Sigma)$, and $\pi_0(\mathcal J_0\XX_{\overline{F}})$ are trivial. This shows that the assertion holds over \(F\). 
\end{proof}

Clearly, for every $\mathcal Y\in\pijx=\Boxx(\Sigma)$ 
one has that $a_{\rho}(\YY)\in [0,1[\cap\QQ$.  
The following lemma will be used in Subsection~\ref{sectionloctran}.
\begin{lem}\label{weirdoma}
Let $\sigma\in\Sigma$ be a face. The map $$\Boxx^{\rig}(\sigma)\times \big\{\sum_{\substack{\rho\in\sigma(1)\\ a_{\rho}\in\ZZ_{\geq 0}}}a_{\rho}b_{\rho}\big\}-\bigcup_{\sigma'\subsetneq\sigma}(\sigma'\times\sigma')\to N^{\rig}$$given by $$(\yyy,\yyy')\mapsto\yyy+\yyy'$$ is injective, and its image is $\sigma^\circ \cap N^{\rig}$. 
\end{lem}
\begin{proof}
Let $(\yyy_1,\yyy_1')$ and $(\yyy,\yyy')$ be in the domain of the map and suppose that they have the same image, i.e.\ that $\yyy_1+\yyy_1'=\yyy+\yyy'$. 
From $\yyy_1=\yyy+(\yyy'-\yyy_1')$ and $\yyy_1,\yyy\in\Boxx^{\rig}(\sigma)$, we deduce $\yyy'=\yyy_1'$, and hence $\yyy_1=\yyy$. We have thus proven the  injectivity. Let now $(\yyy,\yyy')$ be in the domain. Clearly, one has $\yyy+\yyy'\in N^{\rig}\cap\sigma$ and, for every $\sigma'\subsetneq\sigma$, one has that~$\yyy$ or~$\yyy'$ not in~$\sigma'$, hence $\yyy+\yyy'\not\in\sigma'$. It follows that the image of the map is contained in $N^{\rig}\cap\sigma^\circ$. Finally, if $\xxx\in N^{\rig}\cap\sigma^{\circ}$, then $\xxx=\sum_{\rho\in \sigma(1)}a_{\rho}(\xxx)b_{\rho}$ and for every $\rho\in\Sigma(1),$ one has that $a_{\rho}(\xxx)>0$. 
One has that $$q(\xxx)=\xxx-r(\xxx)=\xxx-\sum_{\rho\in\sigma(1)}\lfloor a_{\rho}(\xxx)\rfloor b_{\rho}=\sum_{\rho\in\sigma(1)}\{a_{\rho}(\xxx)\}b_{\rho}\in\Boxx^{\rig}(\sigma).$$  Moreover, for a cone~$\sigma'\subsetneq\sigma$, one has that if $q(\xxx)\in\sigma(1)-\sigma'(1)$, that is, if for every~$\rho\in\sigma(1)-\sigma'(1)$ one has~$\{a_{\rho}(\xxx)\}=0$, then for any such~$\rho$ one has~$\lfloor a_{\rho}(\xxx) \rfloor\neq 0$ because $a_{\rho}(\xxx)>0$. Hence $r(\xxx)=\sum_{\rho\in\Sigma(1)}\lfloor a_{\rho}(\xxx)b_{\rho}\rfloor\not\in\sigma'$. In other words, the pair $(q(\xxx), r(\xxx))$ belongs to the domain of our map. The statement follows. 
\end{proof}
\subsubsection{}In \cite[Definition 2.17]{dardayasudabm}, for almost every finite place~$v$, we have defined a {\it residue map} $\psi_v:\XX(F_v)\to\pi_0(\mathcal J_0\XX).$ We give its description on the subset~$\TT(F_v)$ as follows:

\begin{lem}\label{lem:res-map}
Let~$v$ be a finite place whose  residue characteristic does not divide the order of the stabilizer subgroup of any point of \(\cX\). Then, the restriction of the residue map $\psi_{v}|_{\TT(F_v)}$ to~$\TT(F_v)$ coincides with the map $$
q\circ\log_{\TT, v}:\TT(F_v)\to X_*(\TT)=N^{\rig}\times G^D 
\to \Boxx^{\rig}(\Sigma)\times G^D.
$$
\end{lem}
The rest of this subsection is devoted to the proof of this lemma. 
We have the canonical embedding of \(G=\prod_{j=1}^{\ell}\mu_{b_{i}}\) into \(T':=\Gm^{\ell}\). From \cite[Theorem 6.25]{mann}, the morphism 
$\cX\to\cX^{\rig}$ is essentially trivial $G$-gerbe. From \cite[Prop.\ 2.1.2.6]{lieblich}
the rigidification morphism (non-canonically) factors as
\[
\cX\to\cX^{\rig}\times B T'\to\cX^{\rig}.
\]
We fix such a factorization. The left morphism restricts to the natural morphism
\[
 T\times BG \to T\times BT'.  
\]
Since $\cJ_{0}\cX$ is identified with
$\ulHom(B \widehat{\mu},\cX)$, the morphism $\cJ_{0}\cX\to\cJ_{0}\cX^{\rig}$
factors as
\[
\cJ_{0}\cX\to\cJ_{0}\cX^{\rig}\times\ulHom(B \widehat{\mu},B  T')\to\cJ_{0}\cX^{\rig}.
\]
Given a geometric point $\Spec (L)\to\cJ_{0}\cX$, we get a morphism
$B \widehat{\mu}_{L}\to B T'_{L}$, which in turn induces a morphism
$\widehat{\mu}_{L}\to T'_{L}$. From construction, the image of the
last morphism is contained in the subgroup $G_{L} \subset T'$. Thus, the morphism $B \widehat{\mu}_{L}\to B T'_{L}$
lies in the image of 
\[
\cJ_{0}(B G)=\ulHom(B\widehat{\mu}, B G)\to\ulHom(B\widehat{\mu}, B T').
\]
We get a map $|\cJ_{0}\cX|\to|\cJ_{0}\cX^{\rig}|\times|\cJ_{0}(B G)|$
making the diagram below commutative:
\[
\xymatrix{|\cJ_{0}\cX|\ar[r]\ar[dr] & |\cJ_{0}\cX^{\rig}|\times|\cJ_{0}(B G)|\ar[d]\\
 & |\cJ_{0}\cX^{\rig}|\times|\ulHom(B\widehat{\mu},B T')|
}
\]
Note that we can identify \(|\cJ_0 (BG)|\) with \(\bigoplus_{j=1}^{\ell}\frac{1}{b_{j}}\ZZ/\ZZ\). 

\begin{lem}
The map $|\cJ_{0}\cX|\to|\cJ_{0}\cX^{\rig}|\times|\cJ_{0}(B G)|$
is a homeomorphism.
\end{lem}

\begin{proof}
Let $x$ be a geometric point of $\cX$ and let $x^{\rig}$ be its
image in $\cX^{\rig}$. We have
\[
\Aut(x)\to\Aut(x^{\rig})\times T'(L)\to\Aut(x^{\rig})
\]
with $L$ the relevant algebraically closed field. The left map is
an isomorphism onto $\Aut(x^{\rig})\times G(L)$. It follows that
the map $(\cJ_{0}\cX)( L)\to(\cJ_{0}\cX^{\rig})(L)$
has fiber identified with $G(L)$. We conclude that the map of the
proposition is bijective. 

To see that the map is a homeomorphism, it is enough to observe that the map $|\cJ_{0}\cX|\to|\cJ_{0}\cX^{\rig}|$ sends each connected component of the source isomorphically onto some connected component of the target and similarly for the projection $|\cJ_{0}\cX^{\rig}| \times |\cJ_0(BG)|\to|\cJ_{0}\cX^{\rig}|$.
\end{proof}

\begin{cor}
We have an identification: 
\[
\pi_{0}(\cJ_{0}\cX)=\pi_{0}(\cJ_{0}\cX^{\rig})\times\bigoplus_{j=1}^{\ell}\frac{1}{b_{j}}\ZZ/\ZZ
\]
\end{cor}

Let $v$ be a finite place. We would like to describe the residue
map $\phi_{v}\colon\cT (F_{v}) \to\pi_{0}(\cJ_{0}\cX)$.
We have fixed an isomorphism $\cT\cong T\times B G$, which induces
\[
\cT( F_{v})=T(F_{v})\times\prod_{j=1}^{\ell}F_{v}^{*}/(F_{v}^{*})^{b_{j}}.
\]
From construction, the residue map
\(\cT(F_v)  \to \pi_0(\cJ_0 \cX)\) is identified with the product of the two maps \(T(F_v)\to \pi_0(\cJ_0 \cX^{\rig})\) and 
\(\prod_{j=1}^{\ell}F_{v}^{*}/(F_{v}^{*})^{b_{j}}\to
\bigoplus_{j=1}^{\ell}\frac{1}{b_{j}}\ZZ/\ZZ\). The latter is given by 
\[
([f_j])_j \mapsto \left(\frac{\ord_{v}(f_j)}{b_{j}}\mod{\ZZ}\right)_j,
\]
while the former is the residue map of \(\cX^{\rig}\). 

To complete the proof of Lemma \ref{lem:res-map}, it remains to describe the residue map. The desired description is essentially given in \cite[Th.~4.2]{stapledon}. We may write $T=\Spec (k[M^{\rig}])$.
For \(x\in T(F_v)\), we have the induced map
\[
M^{\rig}\hookrightarrow k[M^{\rig}]\to F_{v}\xrightarrow{\ord_{v}}\ZZ,
\]
where the last map is the normalized additive valuation. This composite
map is an element of $N^{\rig}=(M^{\rig})^{\vee}$, which is nothing but \(\log_{T,v}(x)\). Let \(\sigma \in \Sigma\) be a \(d\)-dimensional cone containing \(u\). 
We define \(N^{\rig}(\sigma):= N^{\rig}/\sum_{\rho\in \sigma(1)} \ZZ b_\rho\). 
Then, the canonical map \(\Boxx^{\rig}(\sigma) \hookrightarrow N^{\rig} \twoheadrightarrow N^{\rig}(\sigma)\) is bijective. 
Lemma \ref{lem:res-map} now follows from:

\begin{lem}
With the notation as above, the residue map $\psi_v:T(F_{v})\to\pi_{0}(\cJ_{0}\cX^{\rig})$ coincides with \(q\circ \log_{T,v}\).
\end{lem}

\begin{proof}
Let $F_{v}^{\mathrm{nr}}$ denote the maximal unramified extension of \(F_v\) and let \(\cO\) denote its integer ring. 
The proposition follows from \cite[Th.~4.2]{stapledon} and arguments
in its proof, up to replacing $\CC((t))$ and $\CC[[t]]$ with  $F_{v}^{\mathrm{nr}}$ and $\cO$ respectively and replacing \(N^{\rig}(\sigma)\) with a non-canonically isomorphic group \(G_\sigma\) at appropriate places. For the sake of completeness, we outline the argument below. 

 Let \(\cY\) denote the toric stack over \(\cO\) defined by the rigidified stacky fan \(\bSigma^{\rig}\).
This stack is covered by open substacks \(\cY(\sigma)=[\AAA_\cO^d/G_\sigma]\) for \(d\)-dimensional cones \(\sigma \in \Sigma\), where \(\AAA_\cO^d\) denotes the \(d\)-dimensional affine space over \(\cO\) (not a space of adeles). 
 Let \(x\in T(F_v)\),  let \(u:=\log_{T,v}(x)\in N^{\rig}\) and let \(\sigma \in \Sigma\) be a \(d\)-dimensional cone containing \(u\) and let \(G_{\sigma}\)  be the finite group given by 
\[
G_{\sigma}:=\Hom(\Hom(N^{\rig}(\sigma),\QQ/\ZZ),\cO^*),
\]
which is non-canonically isomorphic to \(N^{\rig}(\sigma)\). 
Then, the point \(x\) induces a twisted arc 
\[
\tilde{x}:
 [(\Spec (\cO[\pi_l]))/\mu_{l}]\to \cY(\sigma)=[\AAA_\cO^d/G_\sigma]
\]
for some uniquely determined integer \(l > 0\),
where \(\pi_{l}\) is the \(l\)-th root of a uniformizer of \(\cO\). 

Let \(\widehat{\mu}\) denote the projective limit of \(\mu_{l}\) with \(l\) running over positive integers coprime to the residue characteristic of \(v\). 
From Lemma \ref{lem:NsigmaGsigma}, an element  \(\yyy\in\Boxx^{\rig}(\sigma)\) corresponds to a map \(\widehat \mu \to G_\sigma\), which uniquely factors as \(\widehat \mu \twoheadrightarrow \mu_{l} \hookrightarrow G_\sigma\). Let \(J_{\infty}^{(\yyy)}\AAA_\cO^d\) denotes the set of \(\cO\)-morphisms \(\Spec (\cO[\pi_{l}])\to \AAA_\cO^d\) which are equivariant with respect to the above map \(\mu_{l} \hookrightarrow G_\sigma\) induced from \(\yyy\). Let \(k\) be the residue field of \(\cO\).
Let \(\rho_1,\dots,\rho_d\) be the one-dimensional faces of \(\sigma\). According to the chosen order of \(\rho_i\)'s, we may identify the coordinate ring of \(\AAA_\cO^d\) with the polynomial ring  \(\cO[x_1,\dots,x_d]\). If we write  \( a_{\rho_i}^\sigma(\yyy) =a_i/l\), then ring homorphisms corresponding to morphisms \(\Spec (\cO[\pi_{l}])\to \AAA_\cO^d\) as above are explicitly written as:
\begin{equation}\label{eq:twmap}
\cO[x_1,\dots,x_d]\to \cO[\pi_{l}],\,x_i \mapsto  \pi_{l}^{a_i} f_i, \, (f_i \in \cO).  
\end{equation}
Similarly, we define \(J_{0}^{(\yyy)}\AAA_\cO^d\) to be the set of \(\cO\)-morphisms \(\Spec (k)\to \AAA_\cO^d\) which are equivariant with respect to the \(\mu_{l} \hookrightarrow G_\sigma\), that is, the set of \(k\)-points of \(\AAA_\cO^d\) which is invariant under the action of \(\mu_{l}\) embedded into \(G_\sigma\) as above.

The spaces of twisted arcs and 0-jets of $\cY(\sigma)$, denoted by \(\cJ_{\infty}\cY(\sigma)\) and \(\cJ_{0}\cY(\sigma)\) respectively,  decompose
as
\begin{align*}
\cJ_{\infty}\cY(\sigma) & =\coprod_{\yyy\in \Boxx^{\rig}(\sigma)}(J_{\infty}^{(\yyy)}\AAA_\cO^d)/G_\sigma,\\
\cJ_{0}\cY(\sigma) & =\coprod_{\yyy\in \Boxx^{\rig}(\sigma)}(J_{0}^{(\yyy)}\AAA_\cO^d)/G_\sigma.
\end{align*}
We have also the obvious identification \(\pi_0(\cJ_0 \cY) = \pi_0(\cJ_0 \cX^{\rig})\). 
A twisted sector of \(\cY\)
which corresponds to \(\yyy\in \Boxx^{\rig}(\sigma)\) corresponds to the component \((J_{0}^{(\yyy)}\AAA_\cO^d)/G_\sigma\) via the last equality. 
Moreover, through the last equality, a twisted sector of \(\cY\), which corresponds to a twisted sector of \(\cX^{\rig}\), corresponds to the component \((J_{0}^{(\yyy)}\AAA_\cO^d)/G_\sigma\) with \(\yyy\) corresponding to \(\cY\). The natural map \(\cJ_\infty \cY(\sigma) \to \cJ_0 \cY(\sigma)\) corresponds to the natural maps \((J_{\infty}^{(\yyy)}\AAA_\cO^d)/G_\sigma \to (J_{0}^{(\yyy)}\AAA_\cO^d)/G_\sigma
\). To summarize these, we conclude that 
the residue map \(\psi_v\) sends \(x\) to \(\yyy\in \Boxx^{\rig}(\sigma) \subset \Boxx^{\rig}(\Sigma)=\pi_0(\cJ_0 \cX^{\rig})\) if the twisted arc \(\tilde{x} \in \cJ_\infty \cY\) corresponds to a point of \((J_{\infty}^{(\yyy)}\AAA_\cO^d)/G_\sigma\).

On the other hand, 
explicit computation shows that if \(x\in \cX(F_v)\) leads to a map \(\cO[x_1,\dots,x_d]\to \cO[\pi_{l}]\) of the form \eqref{eq:twmap} via these correspondences, then \(\log_{T,v}(x) \in N^{\rig}\) is \(\sum_{i=1}^d (\frac{a_i }{l}+ \ord_v f_i) b_{\rho_i}\) and 
\[
 q\circ \log_{T,v}(x) = \sum_{i=1}^d \frac{a_i}{l} b_{\rho_i} = \yyy.
\]
We have proved the lemma. 
\end{proof}

\begin{lem}\label{lem:NsigmaGsigma}
With the notation as above, 
there exists a natural isomorphism
  \[
  N^{\rig}(\sigma)\to\Hom(\widehat{\mu},G_{\sigma})=\Hom(\widehat{\mu},\Hom(\Hom(N^{\rig}(\sigma),\QQ/\ZZ),\cO^*))
  \]
  which sends $w\in N^{\rig}(\sigma)$ to 
  \[
  \left(\widehat{\mu}\ni\xi\mapsto\left(\left(\alpha\colon N^{\rig}(\sigma)\to\QQ/\ZZ\right)\mapsto(\xi,\alpha(w))\right)\right).
  \]
  Here $(\xi,\alpha(w))$ is the pairing $\widehat{\mu}\times\QQ/\ZZ\to\cO^*$
  given by 
  \[
  ((\xi_{i})_{i},[n/\ell]):=\xi_{\ell}^{n}.
  \]
  \end{lem}
  
  \begin{proof}
  The map of the lemma is natural. We only need to show that it is bijective. If we fix a sufficiently factorial positive
  integer $n$ (coprime to the residue characteristic of the finite place \(v\) in question) and if we fix identification $\mu_{n}\cong(\frac{1}{n}\ZZ)/\ZZ$,
  then we have identifications
  \begin{align*}
   & \Hom(\widehat{\mu},\Hom(\Hom(N^{\rig}(\sigma),\QQ/\ZZ),\cO^*))\\
   & =\Hom\left(\left(\frac{1}{n}\ZZ\right)/\ZZ,\Hom\left(\Hom\left(N^{\rig}(\sigma),\left(\frac{1}{n}\ZZ\right)/\ZZ\right),\left(\frac{1}{n}\ZZ\right)/\ZZ\right)\right)\\
   & =\Hom\left(\Hom\left(N^{\rig}(\sigma),\left(\frac{1}{n}\ZZ\right)/\ZZ\right),\left(\frac{1}{n}\ZZ\right)/\ZZ\right).
  \end{align*}
  Under these identification, the map of the lemma is nothing but the
  double dual map 
  \[
  N^{\rig}(\sigma)\to\Hom\left(\Hom\left(N^{\rig}(\sigma),\left(\frac{1}{n}\ZZ\right)/\ZZ\right),\left(\frac{1}{n}\ZZ\right)/\ZZ\right),
  \]
  which is bijective. 
  \end{proof}


\subsubsection{}
We define $$M^{\rig}:=X^*(T)=\Hom(X_*(T), \ZZ)=\Hom (N^{\rig},\ZZ).$$
The homomorphism $\beta^{\rig}:\ZZ^{\Sigma(1)}\to N^{\rig}$ induces an inclusion $$\Hom(N^{\rig},\ZZ):M^{\rig}\hookrightarrow \ZZ^{\Sigma(1)}$$ and $$\Hom(\beta^{\rig},\RR):M^{\rig}_{\RR}\hookrightarrow\RR^{\Sigma(1)}=\RR^{\Sigma(1)}\times\{0\}^{\pijx}\hookrightarrow\RR^{\Spijx}.$$ The last inclusion is given by $$M^{\rig}_{\RR}\to\RR^{\Spijx},\hspace{1cm}\mmm\mapsto \big(\big(\langle b_{\rho},\mmm\rangle\big)_{\rho\in\Sigma(1)}, (0)_{\mathcal Y\in\pijx}\big)$$ 
If we denote by $i_{\ker(\beta^{\rig}_{\RR})}:\ker(\beta^{\rig}_{\RR})\hookrightarrow\RR^{\Sigma(1)}$ the canonical inclusion, then it is immediate that 
\begin{equation}
\label{mrigform}
M^{\rig}_{\RR}=\ker (\Hom (i_{\ker(\beta^{\rig}_{\RR})} ,\RR))\times\{0\}^{\pijx}.
\end{equation}
\subsection{Heights} 
\subsubsection{} We  define certain pairings which enable us to define logarithmic heights. 
\begin{mydef}
Let $\sss\in\CC^{\Sigma(1)\cup\pijx}$. We define a function $\phi_{\sss}:N^{\rig}\times G^D\to\CC$ as follows:
\[\phi_{\sss}(\yyy,x) := 
\begin{cases}
 s_{q(\yyy,x)}+\sum_{\rho\in\Sigma(1)}a_\rho(\yyy) s_\rho ,& \text{ if } q(\yyy,x)\text{ is a twisted sector,} \\
\sum_{\rho\in\Sigma(1)}a_{\rho}(\yyy)s_{\rho}, & \text{ otherwise.}
\end{cases}
\]
\end{mydef}
\begin{mydef}
Let $\sss\in\CC^{\Sigma(1)\cup\pijx}$. We define a function $\phi_{\sss}^{\infty}:N^{\rig}_{\RR}\to\CC$ to be the unique piecewise linear function $N_{\RR}^{\rig}\to\CC$ which satisfies for every~$\rho\in\Sigma(1)$ that 
$
\phi_{\sss}^{\infty}(b_{\rho})=s_{\rho}.
$
\end{mydef}
%
\begin{lem} \label{obvfacp}
Let $\sss\in\CC^{\Spijx}$ and let $\mmm\in M^{\rig}_{\RR}$. Let $(\yyy,x)\in N^{\rig}\times G^D$. 
The following claims are valid:
\begin{enumerate}
\item 
Suppose that $\yyy'=\sum_{\rho\in\sigma(1)}k_{\rho}b_{\rho}$, with~$k_{\rho}\in\ZZ_{\geq 0}$. Then $$\phi_{\sss}(\yyy+\yyy',x)=\phi_{\sss}(\yyy,x)+\sum_{\rho\in\Sigma(1)}k_{\rho}s_{\rho}. $$
\item One has that $$\phi_{\sss+i\mmm}(\yyy,x)=\phi_{\sss}(\yyy,x)+i\langle\yyy,\mmm\rangle.$$ 
\item For every $\xxx\in N^{\rig}_{\RR}$, one has that $$\phi_{\sss+i\mmm}^{\infty}(\xxx)=\phi_{\sss}(\xxx)+i\langle\xxx,\mmm\rangle.$$
\end{enumerate}
\end{lem}
\begin{proof}
\begin{enumerate}
\item Obviously, one has for $\tau\in\Sigma(1)$ that $$a_{\tau}(\yyy+\yyy')=a_{\tau}\big(\yyy+\sum_{\rho\in\Sigma(1)}k_{\rho}b_{\rho}\big)=a_{\tau}(\yyy)+k_{\tau} $$ and that $$q(\yyy,x)=(q(\yyy),x)=(q(\yyy+\yyy'),x)=q(\yyy+\yyy',x).$$ Hence, 
\begin{align*}\phi_{\sss}(\yyy+\yyy',x)&=s_{q(\yyy+\yyy',x)}+\sum_{\rho\in\Sigma(1)}a_{\rho}(\yyy+\yyy')s_{\rho}\\&=s_{q(\yyy,x)}+\sum_{\rho\in\Sigma(1)}(a_{\rho}(\yyy)+k_{\rho})s_{\rho}\\&=\phi_{\sss}(\yyy,x)+\sum_{\rho\in\Sigma(1)}k_{\rho}s_{\rho}
\end{align*}
\item By definition, we have that
\begin{align*}\phi_{\sss+i\mmm}(\yyy,x)
&=s_{q(\yyy,x)}+\sum_{\rho\in\Sigma(1)}a_{\rho}(\yyy)(s_{\rho}+i\langle b_{\rho},\mmm\rangle)\\
&=s_{q(\yyy,x)}+\sum_{\rho\in\Sigma(1)}a_{\rho}(\yyy)s_{\rho}+ i \big\langle \sum_{\rho\in\Sigma(1)}a_{\rho}(\yyy)b_{\rho},\mmm\big\rangle\\
&=\phi_{\sss}(\yyy,x)+i\big\langle \sum_{\rho\in\Sigma(1)}a_{\rho}(\yyy)b_{\rho},\mmm\big\rangle\\
&=\phi_{\sss}(\yyy,x)+i\langle\yyy,\mmm\rangle.
\end{align*}
\item Let $\sigma\in\Sigma(1)$ be such that $\xxx\in\sigma.$ For $\rho\in\sigma(1)$, there exists a unique $u_{\rho}\in [0,+\infty[$ such that $\xxx=\sum_{\rho\in\sigma(1)}u_{\rho}b_{\rho}$. One has that
\begin{align*}\phi^{\infty}_{\sss+i\mmm}(\xxx)=\sum_{\rho\in\sigma(1)}u_{\rho} \cdot(s_{\rho}+i\langle b_{\rho},\mmm\rangle )&=\sum_{\rho\in\sigma(1)}u_{\rho}s_{\rho}+i\langle\sum_{\rho\in\sigma(1)}u_{\rho}b_{\rho},\mmm\rangle\\& =\phi^{\infty}_{\sss}(\xxx)+i\langle\xxx,\mmm\rangle.
\end{align*}
\end{enumerate}
\end{proof}
For a twisted sector $\mathcal Y\in \pi^*_0(\mathcal J_0\XX)$, one has that 
\begin{align}\begin{split}\label{transpofphi} \phi_{\sss-K_X}(\mathcal Y)&=s_{\mathcal Y}-\age(\mathcal Y)+1+\sum_{\rho\in\Sigma(1)}a_{\rho}(\yyy)(s_{\rho}+1)\\&=s_{\mathcal Y}-\age(\mathcal Y)+1+\sum_{\rho\in\Sigma(1)}a_{\rho}(\yyy)s_{\rho}+\age(\mathcal Y)\\&=\phi_{\sss}(\mathcal Y)+1.
\end{split}
\end{align}
\subsubsection{}In this paragraph, we define certain height pairings.
\begin{mydef}\label{finpair}
Let~$v$ be in the set of finite places~$M_F^0$. We define a pairing 
\begin{align*}H_v(-,-)&:\CC^{\Spijx}\times \TT(F_v)\to \CC^n\\
(\sss,(\yyy,x))&\mapsto\exp\bigg(\log(q_v)\cdot \phi_{\sss}\big(\log_{\TT,v}(\yyy),\psi_v^{BG}(x)\big)\bigg).
\end{align*}
\end{mydef}
It is immediate from definition that~$H_v(\sss,-)$ is $\TT(\Ov)$-invariant.
\begin{mydef}\label{infpair}
Let~$v$ be in the infinite places~$M_F^{\infty}$. We define a pairing
\begin{align*}
H_v(-,-)&:\CC^{\Sigma(1)\cup\pijx}\times \TT(F_v)\to \CC^n\\
(\sss,(\yyy,x))&\mapsto \exp\big([F_v:\RR]\cdot\phi_{\sss}^{\infty}(\log_{\TT,v}(\yyy))\big).
\end{align*}
\end{mydef}
\begin{lem}\label{otrpljati} Let~$\sss\in\CC^{\Sigma(1)\cup\pijx}$, let $\mmm\in M^{\rig}_{\RR}$.  Let~$v$ be a finite place of~$F$.  For every~$\xxx\in\TT(F_v)$, one has that $$H_v(\sss+i\mmm,\xxx)=H_v(\sss,\xxx)\exp\big(\log(q_v)\cdot i\langle\log_{\TT,v}(\xxx),\mmm\rangle )\big).$$
Let~$v$ be an infinite place of~$F$.  For every~$\xxx\in\TT(F_v)$, one has that
$$H_v(\sss+i\mmm,\xxx)=H_v(\sss,\xxx)\exp\big([F_v:\RR]\cdot i\langle\log_{\TT,v}(\xxx),\mmm\rangle )\big).$$
\end{lem}
\begin{proof}
Suppose that~$v$  is finite. Write $\xxx=(\yyy,x)\in N^{\rig}\times G^D$. From Lemma \ref{obvfacp}(2) we obtain that $$\phi_{\sss+i\mmm}(\log_{\TT,v}(\yyy),\psi^{BG}_v(x))=\phi_{\sss}(\log_{\TT,v}(\yyy),\psi^{BG}_v(x))+i\langle\log_{T,v}(\yyy),\mmm\rangle.$$
We deduce that 
\begin{align*}
H_v(\sss+i\mmm,\xxx)&=\exp(\log(q_v)\cdot\phi_{\sss+i\mmm}(\log_{\TT,v}(\yyy),\psi^{BG}_v(x)))\\
&=\exp\big(\log(q_v)\cdot (\phi_{\sss}(\log_{\TT,v}(\yyy),\psi^{BG}_v(x))+i\langle\log_{T,v}(\yyy),\mmm\rangle)\big)\\
&=H_v(\sss,\xxx)\exp\big(\log(q_v)\cdot i\langle\log_{\TT,v}(\xxx),\mmm\rangle )\big),
\end{align*}
as claimed. Suppose that~$v$ is infinite. From Lemma \ref{obvfacp}(3) we obtain that
\begin{align*}
\phi_{\sss+i\mmm}^{\infty}(\log_{\TT,v}(\xxx))&=\phi_{\sss}^{\infty}(\log_{\TT,v}(\xxx))+i\langle\log_{\TT,v}(\xxx),\mmm\rangle.\\
\end{align*}
We deduce that
\begin{align*}
H_v(\sss+i\mmm,\xxx)&=\exp([F_v:\RR]\cdot\phi_{\sss+i\mmm}^{\infty}(\log_{\TT,v}(\xxx)))\\
&=\exp\big([F_v:\RR]\cdot (\phi_{\sss}(\log_{\TT,v}(\xxx))+i\langle\log_{\TT,v}(\xxx),\mmm\rangle)\big)\\
&=H_v(\sss,\xxx)\exp\big([F_v:\RR]\cdot i\langle\log_{\TT,v}(\xxx),\mmm\rangle )\big),
\end{align*}
as claimed.
\end{proof}

\subsection{Cones} We define a cone in~$\RR^{\Spijx}$ for which we prove in Subsection~\ref{northcottsection} that its elements define heights satisfying the Northcott property.
\subsubsection{} Let us first define certain linear forms.
\begin{mydef}\label{linformxidef}
For $\rho\in\Sigma(1)$, we define a linear form on~$\RR^{\Spijx}$ by $$\Xi_{\rho}((t_{\tau})_{\tau\in\Spijx})=t_{\rho}.$$
For $\YY\in\pijx$, we define a linear form on~$\RR^{\Spijx}$ by $$\Xi_{\YY}((t_{\tau})_{\tau\in\Spijx})=t_{\YY}+\sum_{\rho\in\Sigma(1)}a_{\rho}(\YY)t_{\rho}.$$
\end{mydef}
We note that for $\YY\in\pijx$ and $\sss\in\CC^{\Spijx}$, one has that:
\begin{equation}
\label{phixi}
\Xi_{\YY}(\sss)=\phi_{\sss}(\YY).
\end{equation}
For $\mmm\in M^{\rig}_{\RR}\subset\RR^{\Spijx}$ and $\rho\in\Sigma(1)$, it is immediate that:
\begin{equation}
\label{mxij}\Xi_{\rho}(\mmm)=\langle b_{\rho},\mmm\rangle .
\end{equation}
For $\mmm\in M^{\rig}_{\RR}$ and $\YY=(\yyy,x)\in\pijx$, one has that 
\begin{align}
\begin{split}
\label{mxid}\Xi_{\YY}(\mmm)=\phi_{\mmm}(\YY)=0+\sum_{\rho\in\Sigma(1)}a_{\rho}(\YY)\langle b_{\rho},\mmm\rangle&=\sum_{\rho\in\Sigma(1)}a_{\rho}(\yyy)\langle b_{\rho},\mmm\rangle\\&=\big\langle\sum_{\rho\in\Sigma(1)}a_{\rho}(\yyy)b_{\rho},\mmm\big\rangle\\&=\langle\yyy,\mmm\rangle. 
\end{split}
\end{align}
\begin{mydef}\label{defoflambda}
We define a cone $\Lambda\subset \RR^{\Sigma(1)\cup\pi_0^*(\mathcal J_0\XX)}$ by the inequalities
$$\forall\rho\in\Spijx: \Xi_{\rho}\geq 0.$$
For a cone~$C$ we denote by~$C^\circ$ the relative interior of~$C.$
\end{mydef}

\begin{mydef}\label{def:KX}
We denote by~$K_X$ the element $$K_X:=((-1)_{\rho\in\Sigma(1)}, (\age(\mathcal Y)-1)_{\mathcal Y\in\pi^*_0(\mathcal J_0\XX)}).$$
(This element corresponds to the raised line bundle denoted by \(K_{\cX,\orb}^{-1}\) in Abstract and \cite{dardayasudabm}.) 
\end{mydef}

\begin{lem}\label{propoflambda}
\begin{enumerate}
\item The forms~$\Xi_{\rho}$ for $\rho\in\Spijx$ are linearly independent. 
\item For every $\rho\in\Spijx$, one has that $\Xi_{\rho}(-K_X)=1$. In particular, one has $-K_X\in\Lambda^{\circ}$.
\item The cone~$\Lambda$ does not contain a line.
\end{enumerate}
\end{lem}
\begin{proof}
\begin{enumerate}
\item Suppose that~$\alpha_{\rho}\in\RR$ for $\rho\in\Spijx$ are such that $$\sum_{\rho\in\Spijx}\alpha_{\rho}\Xi_{\rho}=0.$$ For every $\sss\in\RR^{\Spijx}$, one has that \begin{multline*}
\bigg(\sum_{\rho\in\Spijx}\alpha_{\rho}\Xi_{\rho}\bigg)(\sss)\\=\sum_{\rho\in\Sigma(1)}\alpha_{\rho}s_{\rho}+\sum_{\YY\in\pijx}\bigg(\alpha_{\YY}s_{\YY}+\sum_{\rho\in\Sigma(1)}\alpha_{\YY}a_{\rho}(\YY)s_{\rho}\bigg)
\end{multline*}
and thus $\alpha_{\YY}=0$ for $\YY\in\pijx$. But then it follows that $\alpha_{\rho}=0$ for $\rho\in\Sigma(1)$. The claim is verified.
\item For~$\rho\in\Sigma(1)$, one clearly has $$\Xi_{\rho}(-K_X)=1,$$ while for $\YY\in\pijx$, one has that $$\Xi_{\YY}(-K_X)=1-\age(\YY)+\sum_{\rho\in\Sigma(1)}a_{\rho}(\YY)=1.$$ 
\item This property of a cone follows from the property from Part~(1), but we recall the proof anyway. Suppose that~$\Lambda$ contains a line~$L$. Then for every ${\rho}\in\Spijx$ and every $\yyy\in L$, one has that $\Xi_{\rho}(\yyy)\geq 0$. But this is possible only if $\Xi_{\rho}|_L=0$ for every~$\rho$. As~$\Xi_{\rho}$ are linearly independent, one has that the only element where they all vanish is~$0$, a contradiction.
\end{enumerate}
\end{proof}
\subsubsection{} We dedicate this paragraph to recall basic properties of $\mathsf X$-functions of cones. We will use this theory only in Section~\ref{Finalcalculations}. Let~$V$ be a real vector space and let $d\vvv$ be a Lebesgue measure on~$V$ normalized by a lattice~$L\subset V.$ Let~$d\vvv^+$ be the Lebesgue measure on~$\Hom(V,\RR)$ normalized by the lattice $\Hom(L,\ZZ)\subset\Hom(V,\RR)$. For a cone $K\subset V$, we let~$K^{+}$ to be its dual cone and for $\yyy\in V\otimes\CC$, we set:
$$\mathsf X_K(\yyy):=\int_{K^+}\exp(-\langle\yyy,\vvv^+ \rangle)d\vvv^+.$$
For a finite set~$Z$ and a subset $S\subset\RR^Z$, we denote by 
\begin{equation}\label{defofS}
\Omega_S:=\{\zzz+i\ttt\in\CC^Z|\hspace{0,1cm}\zzz\in S, \ttt\in\RR^Z\}
\end{equation}
the tube above~$S$ in~$\CC^Z$. The integral defining $\mathsf X_K(\yyy)$ converges for any $\yyy\in\Omega_{K^{\circ}}$ and defines a holomorphic function in the domain \cite[Lemma 5.1]{Bourqui}. Set $n=\dim(V)$. If~$K$ is simplicial and $(\ell_j)_{j=1}^{n}$  are linearly independent linear forms on~$V$ such that~$\vvv\in K^{\circ}$ if and only if $\ell_j(\vvv)>0$ for every $j=1\doots n$, then one has that \cite[Section 3.1.7]{FonctionsZ}
$$\mathsf X_{K}(\yyy)=||d\ell_1\wedge\cdots\wedge d\ell_{n}||\prod_{j=1}^{n}\frac{1}{\ell_j(\yyy)},$$
where $||d\ell_1\wedge\cdots\wedge d\ell_{n}||$ is the volume of parallellopiped formed by $\ell_1\doots \ell_n.$
\begin{lem}\label{xconexi}
One has that $$\mathsf X_{\Lambda}(\sss)=\prod_{\rho\in\Spijx}\frac{1}{\Xi_{\rho}(\sss)}$$ for $\sss\in\Omega_{\Lambda^{\circ}}$.
\end{lem}
\begin{proof}
Clearly, the family defined by $\Xi'_{\rho}(\sss)=s_{\rho}$ for $\rho\in\Spijx$ satisfies that $$\big|\big|\bigwedge_{\rho\in\Spijx}d\Xi_{\rho}'\big|\big|=1.$$ The matrices $(\Xi_{\rho}')_{\rho\in\Spijx}$ and $(\Xi_{\rho})_{\rho\in\Spijx} $ coincide up to the strictly upper triangle (they do coincide on the diagonal). We deduce that $$\big|\big|\bigwedge_{\rho\in\Spijx}d\Xi_{\rho}\big|\big|=1$$ and the claim follows.
\end{proof}
\subsubsection{}\label{defrsig} In this paragraph we define certain polynomials which we will use in Subsection~\ref{sectionloctran}. For $\sigma\in\Sigma\subset N^{\rig}_{\QQ}$, we denote by~$R_{\sigma}$ 
 the rational function 
\begin{equation*}
 R_{\sigma}((X_{\rho})_{\rho\in\Spijx})=\dfrac{\prod_{\rho\in\sigma(1)}X_{\rho}+\sum_{\substack{\mathcal Y=(\yyy,x)\in\pijx\\\yyy\in\sigma}} X_{\YY}\cdot\prod_{\substack{\rho\in\sigma(1)\\a_{\rho}(\YY)= 0}}X_{\rho}}{\prod_{\rho\in\sigma(1)}(1-X_{\rho})}.
\end{equation*}
(When $\sigma=0$, then~$R_{\sigma}$ is understood to be~$1$). We set $$R_{\Sigma}:=\sum_{\sigma\in\Sigma}R_{\sigma}$$and let~$Q_{\Sigma}$ be the polynomial with integer coefficients such that $$R_{\Sigma}=\frac{Q_{\Sigma}}{\prod_{\rho\in\Sigma(1)}(1-X_{\rho})}.$$
\begin{lem}\label{valofqs}
The polynomial $Q_{\Sigma}-1-\sum_{\YY\in\pijx}X_{\YY}$ contains monomials of degree at least~$2$.
\end{lem}
\begin{proof}
In the polynomial
\begin{align*}Q_{\Sigma}&=\bigg(\sum_{\sigma\in\Sigma}R_{\sigma}\bigg)\cdot\prod_{\rho\in\Sigma(1)}(1-X_{\rho})\\
&=\sum_{\sigma\in\Sigma}\prod_{\rho\in\Sigma(1)-\sigma(1)}(1-X_{\rho})\bigg(\prod_{\rho\in\sigma(1)}X_{\rho}+\sum_{\substack{\YY=(\yyy,x)\in\pijx\\\yyy\in\sigma}} X_{\YY} \cdot\prod_{\substack{\rho\in\sigma(1)\\a_{\rho}(\YY)=0}}X_{\rho}\bigg)\\
\end{align*}
the only free coefficient comes from the term corresponding to~$\sigma=0$ and is equal to~$1$. Whenever $\sigma\in\Sigma(1)$, the monomials of degree~$1$ of the polynomial $$\prod_{\rho\in\Sigma(1)-\sigma(1)}(1-X_{\rho})\cdot\bigg(\prod_{\rho\in\sigma(1)}X_{\rho}+\sum_{\substack{\YY=(\yyy,x)\in\pijx\\\yyy\in\sigma}} X_{\YY} \cdot\prod_{\substack{\rho\in\sigma(1)\\a_{\rho}(\YY)=0}}X_{\rho}\bigg)$$ are given by~$X_{\rho}$ for $\rho=\sigma$ and $X_{\YY}$ for $\YY=(\yyy,x)\in\pijx$ with $\yyy\in\sigma$ with $a_{\rho}(\yyy)\neq 0$, that is $\yyy\in\rho^{\circ}=\sigma^{\circ}$, and the coefficient in front of any of these monomials is~$1$. The monomials in the polynomial $\prod_{\rho\in\Sigma(1)}(1-X_{\rho})$ are given by~$X_{\rho}$ for $\rho\in\Sigma(1)$ and their coefficient is~$-1$. This means that monomials~$X_{\rho}$ for $\rho\in\Sigma(1)$ do not appear in~$Q_{\Sigma}$. For every $j\geq 2$ and for every $\sigma\in\Sigma(j),$ the polynomial in the sum defining~$Q_{\Sigma}$ corresponding to~$\sigma$, has for monomials~$X_{\YY},$ for $\YY=(\yyy,x)\in\pijx,$ with $\yyy\in\sigma$ such that for every $\rho\in\sigma(1)$ one has $a_{\rho}(\YY)>0$ , that is, with $\yyy\in\sigma^{\circ}$. The coefficient in front of them is~$1$. We conclude that degree~$1$ monomials in~$Q_{\Sigma}$ are given by~$X_{\YY}$ for $$\YY\in\bigcup_{\sigma\in\Sigma-0}\big\{ \YY=(\yyy,x)\in\pijx|\hspace{0,1cm} \yyy\in\sigma^{\circ}\big\}=\pijx.$$ The statement follows.
\end{proof}
\subsection{Global height}
We now look at the global heights.
\subsubsection{}\label{cwgh} In this paragraph, we compare global heights from \cite{dardayasudabm} with previous height pairings. We let $p:\RR^{\Spijx}\to NS^1_{\orb,\RR}$ to be the canonical map. 
{By a straightforward generalization of \cite[Definition 4.3]{dardayasudabm} to raised \(\RR\)-line bundles, 
an element $\sss\in\NSr$ defines a height function $H_{\sss}:\XX(F)\to\RR_{>0}$ as in \cite[Definition 4.3]{dardayasudabm}. 
We deduce below that~$p(\Lambda)$  is contained in the orbifold pseudo-effective cone of~$\XX$.} We now fix~$\sss\in\Lambda^{\circ}$. We define $L_{\sss}:=p((s_{\rho})_{\rho\in\Spijx})\in p(\Lambda)^{\circ}$. We define a raising function $$c_{\sss}:\pijx\to\RR_{\geq 0}$$ by $$\mathcal Y\mapsto \begin{cases}
0&\mathcal Y\text{ is the untwisted sector,}\\
s_{\mathcal Y}&\mathcal Y\text{ otherwise.}
\end{cases}
$$
The raised {\(\RR\)-line bundle} $(L_{\sss},c_{\sss})$ defines a height function $H_{(L_{\sss},c_{\sss})}:\mathcal X(F)\to\RR_{>0}$ as in \cite[Definition 4.3]{dardayasudabm}.
\begin{lem}\label{gl:pair}
\begin{enumerate}
\item The cone \(p(\Lambda)\) is contained in the orbifold pseudo-effective cone, and hence $(L_{\sss},c_{\sss})$ is big in the sense of \cite[Definition 8.15]{dardayasudabm}.
\item For~$v\in M_F$, let $i_v:\TT(F)\to\TT(F_v)$ be the canonical homomorphism. Let $\xxx\in\TT(F)$. {Then, $$H_{(L_{\sss},c_{\sss})}(\xxx):=\prod_{\vMF}H_v(\sss,i_v(\xxx))$$
is a height function of the raised \(\RR\)-line bundle \((L_\sss,c_\sss)\). 
} 
\end{enumerate}
\end{lem}
\begin{proof}
\begin{enumerate}
\item 
Consider an element \(\ttt \in \Lambda\) and a covering family of stacky curves \(\widetilde{f}:\widetilde{\mathcal{C}}\to \cX_{\overline{F}}\) with \(f:\mathcal{C}\to \cX_{\overline{F}}\) a general member. We need to show that  \((f,(L_\ttt,c_\ttt))\ge 0\), where the left side is the intersection number defined in \cite[Section 8]{dardayasudabm}. 

Let \(z_1,\dots,z_n \in \mathcal{C}(\overline{F})\) be those points that do not map into the open stacky torus \(\TT\). Then, the intersection number \((f,(L_\ttt,c_\ttt))\)
is expressed as the sum  of the local intersection numbers \((f,(L_\ttt,c_\ttt))_{z_i}\).
Here, each \((f,(L_\ttt,c_\ttt))_{z_i}\) is the sum \((f,L_\ttt)_{z_i}+c_{\yyy}\) of the local (non-orbifold) intersection \((f,L_\ttt)_{z_i}\) and \(c_{\yyy}\) with \(\yyy\) the associated sector at \(z_i\). For each \(i\), localizing the morphism \(f:C\to \cX_{\overline{F}}\)  around \(z_i\) leads to an \(\overline{F}((t))\)-point  \(z_i'\in \TT(\overline{F}((t)))\). We see that the above sum \((f,L_\ttt)_{z_i}+c_{\yyy}\) is equal to
\[
\phi_{\ttt}(\log_{\cT,\overline{F}((t))}(z_i'),\psi^{BG}_{\overline{F}((t))}(z_i'))  
\]
where \(\log_{\cT,\overline{F}((t))}\) and \(\psi^{BG}_{\overline{F}((t))}\) are the counterparts of \(\log_{\cT,v}\) and \(\psi^{BG}_v\) for \(\overline{F}((t)))\)-points.
The defining inequalities \(\Xi_\rho \ge 0\) of \(\Lambda\) are precisely the condition ensuring that this number is non-negative.  
Thus, local intersection numbers  \((f,(L_\ttt,c_\ttt))_{z_i}\) are non-negative, and so is the global intersection number \((f,(L_\ttt,c_\ttt))\), as desired.

\item 
We can write
\begin{align*}
  &\prod_{\vMF}H_v(\sss,i_v(\xxx))\\
  & = \prod_{\vMF}\exp\big(\log(q_v)\cdot(\Phi_{\sss}(\log_{\TT,v}(i_v(\xxx))))\big)\times
  \prod_{\vMFz}q_v^{ s_{q(\log_{\TT,v}(i_v(\xxx)))}} 
\end{align*}
Let \(\sss_0 \in \RR^{\Sigma(1)}\) be the image of \(\sss\) and 
let \(D_\sss\) be the toric \(\RR\)-divisor on the coarse moduli space \(\overline{\cX}\) corresponding to  \(\sss_0\). Then, \(D_\sss\) and \(L_\sss\) correspond to each other by the natural identification \(NS^1_\RR (\cX)=NS^1_\RR (\overline{\cX})\).
Then 
\begin{align*}
  T(F_v) \ni \xxx \mapsto
 \exp\big(\log(q_v)\cdot(\Phi_{\sss}(\log_{T,v}(\xxx)))\big) \in \RR
\end{align*}
is a local height function of \(D_\sss\) at each place \(v\). Moreover, at finite places, this is \emph{the} one given by the natural toric model of \(\cX\) over the integer ring \(\cO_F\). 
Indeed, if \(D_\sss\) is a Cartier divisor, then we can see this from standard local equations of \(D_\sss\). 
The general case can be reduced to this case, since \(D_\sss\) is \(\RR\)-Cartier. 
Thus, 
\begin{align*}
  T(F)\ni \xxx \mapsto
   \prod_{\vMF}\exp\big(\log(q_v)\cdot(\Phi_{\sss}(\log_{T,v}(i_v(\xxx))))\big)
\end{align*}
is a height function of \(D_\sss\), and its pull-back to \(\cT(F)\) is a (stable) height function of the (non-raised) line bundle \(L_\sss = (L_\sss,0)\). 




From Lemma \ref{lem:res-map}, there exists a finite subset \(S\subset M_F^0\) such that for  every \(v\in M_F^0 - S\), \(q\circ \log_{\TT,v}\) is the residue map. {Thus, 
\begin{align*}
  & \prod_{\vMF}\exp\big(\log(q_v)\cdot(\Phi_{\sss}(\log_{\TT,v}(i_v(\xxx))))\big)\times
  \prod_{\vMFz -S}q_v^{ s_{q(\log_{\TT,v}(i_v(\xxx)))}} \\
  &= H_{L_\sss} (\xxx)\times
  \prod_{\vMFz -S}q_v^{ s_{q(\log_{\TT,v}(i_v(\xxx)))}}
\end{align*}
is a height function of \((L_\sss,c_\sss)\) restricted to points of the stacky torus. 
Then, the function of the assertion,
\begin{align*}
  \prod_{\vMF}H_v(\sss,i_v(\xxx))=H_{L_\sss} (\xxx)\times
  \prod_{\vMFz}q_v^{ s_{q(\log_{\TT,v}(i_v(\xxx)))}} ,
\end{align*}
is also a height function of \((L_\sss,c_\sss)\) restricted to points of the stacky torus, since  the function \(\xxx \mapsto \prod_{v \in S}q_v^{ s_{q(\log_{\TT,v}(i_v(\xxx)))}}\) is bounded below and above by the positive constants \(\prod_{v\in S}q_v^{s_{\min}}\) and \(\prod_{v\in S}q_v^{s_{\max}}\), where \(s_{\min}\) and \(s_{\max}\) are the minimum and maximum values of \(s_\cY\), \(\cY\in \pi_0(\cJ_0\cX)\) respectively.
}

\end{enumerate}
\end{proof}

\begin{rem}
\normalfont
As observed in the proof of the first assertion of the last lemma, 
the image of \(\Lambda\) of \(\NSr\) is characterized by the non-negativity of \emph{local} orbifold intersection numbers (for points of the stacky torus \(\TT\)), while the orbifold pseudo-effective cone defined in \cite{dardayasudabm} is characterized by the non-negativity of \emph{global} orbifold intersection numbers. These two cones coincide for varieties. We do not know whether they coincide also for DM stacks. 
\end{rem}
{
\begin{rem}
\normalfont
We speculate that our choice of the height function \(H_{(L_\sss,c_\sss)}\) is very natural in the sense that the maps \(q\circ \log_{\TT,v}\) would be regarded as residue maps even for finitely many ``bad/exceptional'' finite places, once we have made a verbatim generalization of  \(\cJ_0 \cX\)  and the residue map to all finite places by using group schemes \(\mu_l\) even when \(l\) devides the residue characteristic.
\end{rem}
}
\subsubsection{}We prove the following estimate that we will be used later. For $\yyy=(\yyy_v)_{\vMF}\in\TTT(\AAF)$ and $\sss\in\CC^{\Spijx}$, we set:
$$H(\sss,\yyy):=\prod_{\vMF}H_v(\sss,\yyy_v).$$
\begin{lem}\label{estheighty}
Let~$\yyy\in\TT(\AAF)$. For every $\sss\in{]0,+\infty[^{\Spijx}}$, there exists $C(\yyy)>0$ such that for every $\xxx\in\TT(\AAF)$, one has that:
$$(\exp(||\sss||)\cdot C(\yyy))^{-1}H(\sss,\xxx)\leq H(\sss,\yyy\xxx)\leq (\exp(||\sss||)\cdot C(\yyy))\cdot H(\sss,\xxx).$$
\end{lem}
\begin{proof}We prove only the second inequality, because then the first inequality follows by setting $\xxx=\yyy\xxx$ and $\yyy=\yyy^{-1}$, and changing $C(\yyy)$ by $\max(C(\yyy),C(\yyy^{-1}))$ if needed. There exists a finite set of places $S\subset M_F$, such that for every $v\in M_F-S$, one has that $\log_{\TT,v}(\yyy_v)=0$. For every $v\in M_F-S$, one hence has that 
\begin{align*}H_v(\sss,\yyy_v\xxx_v)=\exp(\log(q_v)\cdot\phi_{\sss}(\log_{\TT,v}(\yyy_v\xxx_v)))&=\exp(\log(q_v)\cdot\phi_{\sss}(\log_{\TT,v}(\xxx_v)))\\&=H_v(\sss,\xxx_v).
\end{align*}
For every $v\in S\cap M_F^0$, one has that
 \begin{align*}
\phi_{\sss}(\log_{\TT,v}(\yyy_v\xxx_v))-\phi_{\sss}(\log_{\TT,v}(\xxx_v))\hskip-4,5cm&\\
&\leq\max_{\YY,\YY'\in\pijx}|s_{\YY}-s_{\YY'}|+\sum_{\rho\in\Sigma(1)}(a_{\rho}(\log_{\TT,v}(\yyy_v\xxx_v))-a_{\rho}(\log_{\TT,v}(\xxx_v)))s_{\rho}.
\end{align*}
Write $\xxx=q(\xxx)+\sum_{\rho\in\Sigma(1)}n_{\rho}b_{\rho}$, with $n_{\rho}\in\ZZ_{\geq 0}$ and $\yyy=q(\yyy)+\sum_{\rho\in\Sigma(1)}k_{\rho}b_{\rho}$, with $b_{\rho}\in\ZZ_{\geq 0}$. We have that
\begin{align*}a_{\rho}(\log_{\TT,v}(\yyy_v\xxx_v))&=a_{\rho}(\log_{\TT,v}(\yyy_v)+\log_{\TT,v}(\xxx_v))\\&\leq \big(\max_{\YY\in\pijx}a_{\rho}(\YY)\big)+k_{\rho}+n_{\rho}. 
\end{align*}
We deduce that there exists $C_v'(\yyy_v)>0$ such that \begin{align*}\sum_{\rho\in\Sigma(1)}(a_{\rho}(\log_{\TT,v}(\yyy_v\xxx_v))-a_{\rho}(\log_{\TT,v}(\xxx_v)))s_{\rho}\hskip-2cm&\\&=\sum_{\rho\in\Sigma(1)}\big(\big(\max_{\YY\in\pijx}a_{\rho}(\YY)\big)+k_{\rho}+n_{\rho}-n_{\rho}\big)s_{\rho}\\&\leq \#\Sigma(1)\cdot\big(\max_{\substack{\YY\in\pijx\\\rho\in\Sigma(1)}}a_{\rho}(\YY)\big)||\sss||+\sum_{\rho\in\Sigma(1)}k_{\rho}s_{\rho}\\
&\leq ||\sss||\cdot C_v'(\yyy_v).
\end{align*}

We deduce that for some $C_v(\yyy_v)>0$, one has that:
$$\phi_{\sss}(\log_{\TT,v}(\yyy_v\xxx_v))-\phi_{\sss}(\log_{\TT,v}(\xxx_v))\leq C_v(\yyy_v)\cdot||\sss||.$$
For $v\in S\cap M_F^{\infty}$, by using that $\phi^{\infty}_{\sss}$ is $\Sigma$-piecewise linear, we deduce that 
\begin{align*}\phi^{\infty}_{\sss}(\log_{\TT,v}(\yyy_v\xxx_v))-\phi^{\infty}_{\sss}(\log_{\TT,v}(\xxx_v))\hskip-2cm&\\&=\phi^{\infty}_{\sss}(\log_{\TT,v}(\yyy_v)+\log_{\TT,v}(\xxx_v))-\phi^{\infty}_{\sss}(\log_{\TT,v}(\xxx_v))\\
&\leq C_v'(\yyy_v)\cdot||\sss||\cdot||\log_{\TT,v}(\yyy_v)||\\
&\leq C_v(\yyy_v)\cdot ||\sss||
\end{align*} for some $C_v'(\yyy_v), C_v(\yyy_v) >0.$ 
Hence, for $$C(\yyy):=\prod_{v\in S} \log(q_v)\cdot \log(C(\yyy_v)),$$ one has that 
\begin{align*}\frac{H(\sss,\xxx\yyy)}{H(\sss,\xxx)}&\leq \prod_{v\in S}\exp(\log(q_v)\cdot (\phi_{\sss}(\log_{\TT,v}(\yyy_v\xxx_v))-\phi_{\sss}(\log_{\TT,v}(\xxx_v))))\\&\leq C(\yyy)\exp(||\sss||).
\end{align*}
The statement follows.
\end{proof}
%
\subsection{Northcott property}\label{northcottsection} The goal of this subsection is to prove Northcott's property for our heights defined by elements of~$\Lambda^\circ$.

\subsubsection{}
Let
\(
  \ttt=(t_\rho)_{\rho \in \Sigma(1)\cup \pijx} \in \Lambda^\circ
\)
and let \(H=H_\ttt:\TT(F)\to \RR\) be the associated height function. 
Let $H\colon\cT(F)\to\RR$ be the height function associated to \(\ttt\).  It is written as a product $H=\prod_{v\in M_F}H_{v}$. 
We define the \emph{finite part} $H_{f}:=\prod_{v\in M_F^0}H_{v}$
and the \emph{infinite part} $H_{\infty}:=\prod_{v\in M_F^\infty}H_{v}$
so that $H=H_{f}\cdot H_{\infty}$. In this paragraph, we establish that if on some subset of~$\TT(F)$ the global height is bounded then its finite and infinite parts are bounded. In what follows, \(B\) denotes a positive real number. 

\begin{lem}
Let $V\subset\cT(F)$ be a subset. There exists a positive constant \(C\) such that  if $H|_{V}\le B$, then $H_{\infty}|_{V}\le B$
and $H_{f}|_{V}\le CB$. 
\end{lem}

\begin{proof}
Suppose $H|_{V}\le B$.
The assumption on $\bt$ shows that $H_{f,v}\ge1$. Thus, $H_{f}\ge1$.
This shows that $H_{\infty}|_{V}\le B$. 
To show $H_{f}|_{V}\le CB$, it suffices to show that \(H_\infty\) is bounded below by a positive constant. 
The function \(H_\infty\) factors through the corresponding function \(H_\infty'\) on \(T(F)\), which is  the infinite part of the height function of an effective toric \(\QQ\)-divisor on the coarse moduli space \(X\) of \(\cX\). Moreover, this divisor has positive coefficients at all toric prime divisors. 
The function \(H_\infty'\) factors as 
\begin{equation}
T(F)\to\prod_{v\in M_F^\infty}T(F_{v})\to\prod_{v \in M_F^\infty}X(F_{v})\to\RR_{>0}\cup\{\infty\}\subset\PP^{1}(\RR).\label{eq:fact-inf}
\end{equation}
Since the  map \(\prod_{v \in M_F^\infty}X(F_{v})\to\PP^{1}(\RR)\) is continuous and \(\prod_{v \in M_F^\infty}X(F_{v})\) is compact, the image of  \(\prod_{v \in M_F^\infty}X(F_{v})\to\PP^{1}(\RR)\) is a compact subset of \(\RR_{>0} \cup \{\infty\}\). This completes the proof.
\end{proof}

\subsubsection{}In this paragraph we will construct a map $\alpha:\TT(F)\to\bigoplus_{\vMFz}(N^{\rig}\oplus G^D)$ such that the image of $\{x\in\TT(F)|\hspace{0,1cm}H(\xxx)\leq B\}$ is finite set of cardinality bounded by a polynomial in~$B$.

Recall that in Definition~\ref{defofqr}, we have defined maps $q, r:N^{\rig}\to N^{\rig}$ by $$q(\xxx)=\sum_{\rho\in\Sigma(1)}\{ a_{\rho}(\xxx)\} b_{\rho},\hspace{1cm}r(\xxx)=\sum_{\rho\in\Sigma(1)}\lfloor a_{\rho}(\xxx)\rfloor b_{\rho}.$$
For a point $\xxx\in\cT(F)$
and for a finite place $v$, the associated residue $\psi_{v}(\xxx)\in\pi_{0}(\cJ_{0}\cX)$
corresponds to the element $q(\log_{T,v}(\xxx))$ of $\Boxx^{\rig}(\Sigma)$
together with some element of $G^{D}$. Thus, 
\begin{align*}
H_{v}(\xxx)&=q_{v}^{\phi_{\bt}(\log_{T,v}(\xxx))+t_{\psi_{v}(\xxx)}}\\
&=q_{v}^{\phi_{\bt}(r(\log_{T,v}(\xxx)))+\phi_{\bt}(q(\log_{T,v}(\xxx)))+t_{\psi_{v}(\xxx)}}\\
&=q_{v}^{\phi_{\bt}(r(\log_{T,v}(\xxx)))+\Xi_{\psi_v(\xxx)}(\ttt)}.
\end{align*}
In particular, $H_{v}(\xxx)\ge1$.

For \(\cY \in \pi_0(\cJ_0 \cX)\), we denote by \(\overline{\cY}\) its image in \(\Boxx^{\rig}(\Sigma)\subset N^{\rig}\) by the natural map. 
Let 
\[
(N^{\rig}\times \pi_0(\cJ_0\cX))^\circ
:=\{(u,\cY)\in N^{\rig}\times \pi_0(\cJ_0\cX) \mid q(u)=\overline{\cY}\}. 
\]
The local height function
$H_{v}$ at a finite place factors as
\[
H_{v}\colon\cT(F)\xrightarrow{(\log_{T,v},\psi_{v})}(N^{\rig}\times\pi_{0}(\cJ_{0}\cX))^\circ\xrightarrow{r \times\id}N^{\rig}\times\pi_{0}(\cJ_{0}\cX)\to\RR.
\]
Note that the middle map above is injective. Indeed, a pair \((u,\cY)\) in \((N^{\rig}\times\pi_{0}(\cJ_{0}\cX))^\circ\) is recovered from its image \((r(u), \cY)\), since \(u= r(u) + q(u) = r(u)+\overline{\cY}\).

We define the ``integral part'' $H_{i}$ of $H_{f}$ by
\[
H_{i}\colon\bigoplus_{v\in M_F^0}N^{\rig}\to\RR,\,(u_{v})_{v}\mapsto\prod_{v}q_{v}^{\phi_{\bt}(u_{v})}
\]
and the ``extra part'' $H_{e}$ by
\[
H_{e}:\sideset{}{^{\prime}}\prod_{v\in M_F^0}\pi_{0}(\cJ_{0}\cX)\to\RR,\,(\cY_{v})_{v}\mapsto\prod_{v}q_{v}^{\Xi_{\cY_v}(\ttt)}.
\]
Here the restricted product is taken for the singleton $\{\cX\}\subset\pi_{0}(\cJ_{0}\cX)$
consisting of the non-twisted sector. Thus, the finite part $H_{f}$
factors as 
\[
\cT(F)\xrightarrow{((r\times\id)\circ(\log_{T,v},\psi_{v}))_{v}}\bigoplus_{v\in M_F^0}N^{\rig}\times\sideset{}{^{\prime}}\prod_{v\in M_F^0}\pi_{0}(\cJ_{0}\cX)\xrightarrow{H_{i}\cdot H_{e}}\RR.
\]

By definition, $H_{i},H_{e}\ge1$. Therefore, if $H_{f}|_{V}\le B$
for some subset $V\subset\cT(F)$, then $H_{i}|_{\overline{V}}\le B$
for the image $\overline{V}$ of $V$ in $\bigoplus_{v\in M_F^0}N^{\rig}$.
Similarly for $H_{e}$. From the case of toric varieties, \(H_i: \bigoplus_{v\in M_F^0}N^{\rig}\to \RR\) is polynomially bounded. To show that \(H_e: \sideset{}{^{\prime}}\prod_{v\in M_F^0}\pi_{0}(\cJ_{0}\cX) \to \RR \) is polynomially bounded, 
we can consider the following embedding:
\begin{align*}
\sideset{}{^{\prime}}\prod_{v\in M_F^0}\pi_{0}(\cJ_{0}\cX) & \hookrightarrow\{\text{radical ideals of }\cO_{F}\}^{\pijx}\\
(\cY_{v})_{v} & \mapsto\left(\prod_{\cY_{v}=\cY}\fp_{v}\right)_{\cY\in\pijx}
\end{align*}
Here \(\fp_v\) denotes the prime ideal of \(\cO_F\) corresponding to \(v\).
The extra part $H_{e}$ is the composition of this embedding  with:
\begin{align*}
\eta:\{\text{radical ideals of }\cO_{F}\}^{\pijx} & \to\RR\\
(\fa_{\cY})_{\cY} & \mapsto\prod_{\cY}N(\fa_{\cY})^{\Xi_{\cY}(\ttt)}
\end{align*}

In what follows, we say that a function \( f : A \to \RR_{>0}\)
is \emph{polynomially bounded} if there exists a polynoial \(P(X)\) such that for every \(B>0\), \[\# \{a\in A \mid f(a)\le B\} \le P(B)\]. 

\begin{lem}
The above map \(\eta\) is polynomially bounded. 
\end{lem}

\begin{proof}
The map 
\begin{align*}
  \{\text{radical ideals of }\cO_{F}\} & \to\RR,\,\fa \mapsto N(\fa)
\end{align*}
is polynomially bounded \cite[pages 145--150]{zbMATH00843465}. The lemma easily follows from this fact together with the positivity of \(\Xi_{\cY}(\ttt)\).
\end{proof}

The last lemma shows that \(H_e\) is polynomially bounded, and hence so is \(H_i\cdot H_e : \bigoplus_{v\in M_F^0}N^{\rig}\times\sideset{}{^{\prime}}\prod_{v\in M_F^0}\pi_{0}(\cJ_{0}\cX) \to \RR \). 

\begin{lem}\label{lem:Northcott-aux1}
The image of $\{\xxx\in\cT(F)\mid H(\xxx)\le B\}$ by the map
\[
\cT(F)\xrightarrow{((\log_{T,v},\psi_{v}))_{v}}\bigoplus_{v\in M_F^0}N^{\rig}\times\sideset{}{^{\prime}}\prod_{v\in M_F^0}\pi_{0}(\cJ_{0}\cX)
\]
is a finite set and has cardinality bounded by a polynomial in $B$. 
\end{lem}

\begin{proof}
We have proved above that the image of the same set by the map $((r\times\id)\circ(\log_{T,v},\psi_{v}))_{v}$
is a finite set with cardinality bounded by a polynomial in \(B\).  Since the map
\begin{multline*}
  (r\times\id)_{v}
  :  {\prod_{v\in M_F^0}}'(N^{\rig}\times\pi_{0}(\cJ_{0}\cX))^\circ
  \to
  {\prod_{v\in M_F^0}}'(N^{\rig}\times\pi_{0}(\cJ_{0}\cX))\\=
  \bigoplus_{v\in M_F^0}N^{\rig}\times\sideset{}{^{\prime}}\prod_{v\in M_F^0}\pi_{0}(\cJ_{0}\cX)
\end{multline*}
is injective, we also get the corresponding
result for the image by the map $((\log_{T,v},\psi_{v}))_{v}$ as desired.
\end{proof}
The map $(\log_{T,v},\psi_{v})\colon\cT(F)\to N^{\rig}\times\pi_{0}(\cJ_{0}\cX)$
is written as the composite
\[
\cT(F)\xrightarrow{\log_{\cT,v}} N^{\rig}\oplus G^{D}\to N^{\rig}\times\pi_{0}(\cJ_{0}\cX).
\]
The left map is a homomorphism of abelian groups and
 the right map is injective. Moreover, the composition 
\[
N^{\rig}\hookrightarrow N^{\rig}\oplus G^{D}\to N^{\rig}\times\pi_{0}(\cJ_{0}\cX)\xrightarrow{\pr_{1}}N^{\rig}
\]
is the identity map.
From Lemma \ref{lem:Northcott-aux1} and the injectivity of \(N^{\rig}\oplus G^{D}\to N^{\rig}\times\pi_{0}(\cJ_{0}\cX)\), it is immediate to see:
\begin{lem}\label{lem:bound-image}
The image of $\{\xxx\in\cT(F)\mid H(\xxx)\le B\}$ by the homomorphism
\[
\alpha\colon\cT(F)\to\bigoplus_{v\in M_F^0}(N^{\rig}\oplus G^{D}).
\]
is a finite set and has cardinality bounded by a polynomial in $B$. 
\end{lem}

\subsubsection{}

We would like to show:
\begin{lem}
\label{lem:bound-fiber}Each fiber of the map
\[
\alpha|_{\{\xxx\in\cT(F)\mid H(\xxx)\le B\}}\colon\{\xxx\in\cT(F)\mid H(\xxx)\le B\}\to\bigoplus_{v\in M_F^0}(N^{\rig}\oplus G^{D})
\]
is a finite set of cardinality bounded by a polynomial in $B$ which
is independent of the fiber. 
\end{lem}

Let us write $\cT=(\Gm)^{d}\times B\left(\prod_{i=1}^{\ell}\mu_{n_{i}}\right)$
 and identify $\cT(F)$ with $(F^{\times})^{d}\times\prod_{i}F^{\times}/(F^{\times})_{n_{i}}$.
The map $\alpha$ is then the product of 
\[
\mathbf{log}=(\log_{T,v})_{v}\colon(F^{\times})^{d}\to\bigoplus_{v\in M_F^0}\ZZ^{d}
\]
and 
\[
\gamma\colon\prod_{i}F^{\times}/(F^{\times})_{n_{i}}\to\bigoplus_{v\in M_F^0}\bigoplus_{i}\ZZ/n_{i}\ZZ,\,(x_{i})_{i}\mapsto((\ord_{v}(x_{i})\mod{n_{i}})_{i})_{v}.
\]
Here \(\ord_v\) means the normalized valuation corresponding to \(v\). 
If we denote by $U$ the group of units of $F$, then we have 
\[
\Ker(\mathbf{log})=U^{d}.
\]

\begin{lem}
The kernel of $\gamma$ is finite. 
\end{lem}

\begin{proof}
We have the following snake commutative diagram:
$$
\begin{tikzcd}
U^{\ell} \arrow{d}{} \arrow{r}{(n_i)}& U^{\ell}\arrow{r}{}\arrow{d}{}&\Ker(\gamma)\arrow{d}{}\\
(F^{\times})^{\ell}\arrow{r}{(n_i)}\arrow{d}{\ord_v}&(F^{\times})^{\ell}\arrow{r}{}\arrow{d}{\ord_v} &\prod_i F^{\times}/(F^{\times})^{n_i}\arrow{d}{\gamma}\\
\bigoplus_{\vMFz}\ZZ^{\ell}\arrow{r}{}\arrow{d}{}&\bigoplus_{\vMFz}\ZZ^{\ell}\arrow{r}{}\arrow{d}{}&\bigoplus_i\bigoplus_{\vMFz}(\ZZ/n_i\ZZ)\arrow{d}{}\\
Cl(F)^{\ell}\arrow{r}{}&Cl(F)^{\ell}\arrow{r}{}&\Coker(\beta).
\end{tikzcd}
$$
 Here \(Cl(F)\) denote the class group of \(F\). 
The snake lemma gives the exact sequence
\[
0\to\prod_{i}U/U_{n_{i}}\to\Ker(\gamma)\to\Ker(Cl(F)^{\ell}\to Cl(F)^{\ell})\to0,
\]
where $U_{n_i}=U\cap (F^{\times})_{n_i}$. Since $Cl(F)$ and $\prod_{i}U/U_{n_{i}}$ are finite groups, so is
$\Ker(\gamma)$.\end{proof}
We consider the map
\[
(\log_{T,v})_{v\in M_F^\infty}\colon U^{d}\to\bigoplus_{v\in M_F^\infty}N_{\RR}^{\rig}.
\]
Its kernel is the $d$-th power of the group of roots of unity, which
is finite, while its image is a finitely generated subgroup of the
target real vector space. 
Each fiber $\Phi$ of 
\[
\alpha|_{\{\xxx\in\cT(F)\mid H(\xxx)\le B\}}\colon\{\xxx\in\cT(F)\mid H(\xxx)\le B\}\to\bigoplus_{v\in M_F^0}(N^{\rig}\oplus G^{D})
\]
is identified with the product of the following three sets:
\begin{itemize}
\item $\Ker(\gamma)$,
\item $\Ker((\log_{T,v})_{v\in M_F^\infty})$,
\item the intersection of the polytope $\phi_{\bt}\le\log B$ and $((\log_{T,v})_{v\in M_F^\infty})(U)+w_{\Phi}$,
where $w_{\Phi}$ is an element in $\Phi$ if any.
\end{itemize}
The first two sets are finite sets independent of the fiber. 
\begin{lem}
The intersection of the polytope $\phi_{\bt}\le\log B$ and $((\log_{T,v})_{v\in M_F^\infty})(U)+w_{\Phi}$
has cardinality bounded by a polynomial in $\log B$ which is independent
of the fiber.
\end{lem}

\begin{proof}
Choosing a suitable coordinate system, we identify $\bigoplus_{v\in M_F^\infty}N_{\RR}^{\rig}$
with the Euclidean space $\RR^{n}$ in such a way that $((\log_{T,v})_{v\in M_F^\infty})(U)$ 
is contained in $\ZZ^{n}$. We choose a positive integer $m$ such
that the polytope $\phi_{\bt}\le1$ is contained in the cube $[-m,m]^{n}$.
Then, the intersection set in the lemma is contained in 
\[
[-m\log B,m\log B]^{n}\cap(\ZZ^{n}+w_{\Phi}).
\]
Its cardinality is at most $(2m\log B+1)^{n}$. 
\end{proof}
This lemma complete the proof of Lemma \ref{lem:bound-fiber}.
Combining Lemmas \ref{lem:bound-image} and \ref{lem:bound-fiber} gives:

\begin{thm}\label{northcottproperty}
For every $\ttt\in\Lambda^\circ$,  there exists a polynomial~$P$ such that for any \(B \ge 0 \),   $$\#\{\xxx\in\TT(F)|\hspace{0,1cm}H_{\ttt}(\xxx)\leq B\}\leq P(B).$$
\end{thm}

\section{Analysis of height zeta functions} \label{aohzf}
We are in the situation of the last section. In order to count rational points on the open substack $\TT\subset\XX$ of bounded height, we define a height zeta function and use harmonic analysis to understand its behavior. For every finite set~$S$ and every domain $D\subset\RR^S$, we denote by~$\Omega_D$ the set of~$\zzz\in\CC^S$, given by $\Re(\zzz)\in D$.
\subsection{Local transform}\label{sectionloctran}
We will calculate Fourier transforms of local heights.
We will fix a notation. For $\rho\in\Sigma(1)$, for the purpose of notation, we may identify~$\rho$ with $(b_{\rho},0)$. E.g.\ we denote by $f_{\rho}:\ZZ\to N^{\rig}\times G^D$ the homomorphism $$f_{\rho}=f_{(b_{\rho},0)}:\ZZ\to N^{\rig}\times G^D.$$
Similarly, by~$\chi^{(\rho)}$ we will denote $$\chi^{(\rho)}=\chi^{(b_{\rho},0)}.$$
\subsubsection{}
Let~$v$ be a finite place of~$F$.
\begin{lem}\label{loctraexp} Let $\chi\in\TT(\Ov)^{\perp}[\TT(F_v)^*]$. One has that
\begin{align*}
\wH_v(\sss,\chi^{-1})=\int_{\TT(F_v)}H_v(\sss,-)^{-1}\chi\mu_v=R_{\Sigma}(\xxx),
\end{align*}
where
$$\xxx=\big(\exp(-\log(q_v)\cdot\Xi_{\rho}(\sss))\overline{\chi}^{(\rho)}(\pi_v)\big)_{\rho\in\Spijx}. $$
\end{lem}
\begin{proof}
Recall the exact sequence $$1\to (\TT(\Ov), \mu^{\proba})\to (\TT(F_v), \mu_v)\xrightarrow{\log_{\TT, v}} (N^{\rig}\times G^D, \mu^{\#})\to 1$$
from Definition~\ref{muvdef}. Recall from Definition~\ref{finpair} that for~$\xxx\in\TT(F_v)$, one has that $$H(\sss,\xxx)=\exp(\log(q_v)\cdot\phi_{\sss}(\log_{\TT, v}(\xxx))).$$
For $\xxx\in \TT(F_v)$, we have by Lemma~\ref{coolrelat} that $$\overline{\chi}(\xxx)=\big((\log_{\TT, v}^*)^{-1}(\overline\chi)\big)(\log_{\TT,v}(\xxx))=\overline\chi^{(\log_{\TT,v}(\xxx))}(\pi_v).$$
We decompose $N^{\rig}\times G^D$ in a disjoint union $N^{\rig}\times G^D=\coprod_{\sigma\in\Sigma} (N^{\rig}\cap\sigma^{\circ})\times G^D.$ Now we can integrate $\TT(\Ov)$-invariant function $H_v(\sss,-)^{-1}\overline{\chi}$ as follows:
\begin{align}\label{sigmsumexp}
\begin{split}\int_{\TT(F_v)}H_v(\sss,-)^{-1}\overline{\chi}\mu_v\hskip-3cm&\\&
=\sum_{\sigma\in\Sigma}\sum_{(\yyy,x)\in(\sigma^{\circ}\cap N^{\rig})\times G^D}\exp(-\log(q_v)\phi_{\sss}(\yyy,x))\cdot (((\log_{\TT, v}^*)^{-1}(\overline\chi))(\yyy,x))\\
&=\sum_{\sigma\in\Sigma}\sum_{(\yyy,x)\in(\sigma^{\circ}\cap N^{\rig})\times G^D}\exp(-\log(q_v)\phi_{\sss}(\yyy,x))\cdot \overline\chi^{(\yyy,x)}(\pi_v).
\end{split}
\end{align}
For each face $\sigma\in\Sigma$, Lemma~\ref{weirdoma} establishes a bijection $$\Boxx^{\rig}(\sigma)\times \big\{\sum_{\substack{\rho\in\sigma(1)\\ k_{\rho}\in\ZZ_{\geq 0}}}k_{\rho}b_{\rho}\big\}-\bigcup_{\sigma'\subsetneq\sigma}(\sigma'\times\sigma') \to \sigma^{\circ}\cap N^{\rig},\hspace{0,2 cm}(\yyy,\yyy')\mapsto \yyy+\yyy'.$$
The term corresponding to~$\sigma$ in our sum becomes $$\sum_{\substack{\yyy\in\Boxx^{\rig}(\sigma),\\x\in G^D}}\sum_{\substack{\yyy'\in\{\sum_{\rho\in\sigma(1)}k_{\rho}b_{\rho}|\hspace{0,1cm}\forall\rho\in\sigma(1): k_{\rho}\in\ZZ_{\geq 0}\},\\\forall\sigma'\subsetneq\sigma, \yyy\in\sigma'\implies \yyy'\not\in\sigma'}}\hskip-1cm\exp(-\log(q_v)\phi_{\sss}(\yyy+\yyy',x))\cdot\overline{\chi}^{(\yyy+\yyy',x)}(\pi_v).$$
Setting $\yyy'=\sum_{\rho\in\sigma(1)}k_{\rho}b_{\rho}$, and using Lemma~\ref{obvfacp}(1) and Corollary~\ref{muiop}(3) gives that \begin{multline*}\exp(-\log(q_v)\phi_{\sss}(\yyy+\yyy',x))\cdot\overline{\chi}^{(\yyy+\yyy',x)}(\pi_v)\\=\exp(-\log(q_v)\phi_{\sss}(\yyy,x))\overline{\chi}^{(\yyy,x)}(\pi_v)\cdot  \exp\big(-\log(q_v)\sum_{\rho\in\sigma(1)}k_{\rho}s_{\rho}\big)\cdot \overline{\chi}^{(\sum_{\rho\in\sigma(1)}k_{\rho}b_{\rho},0)}(\pi_v).
\end{multline*}
We recall from Equation~(\ref{phixi}) that $\Xi_{(\yyy,x)}(\sss)=\phi_{\sss}(\yyy,x) $ whenever $(\yyy,x)\in\Boxx^{\rig}(\Sigma(1))-\{0\}$ and the term corresponding to~$\sigma$ becomes:
\begin{multline*}
\sum_{\substack{\yyy\in\Boxx^{\rig}(\sigma)\\x\in G^D}}\exp(-\log(q_v)\Xi_{(\yyy,x)}(\sss))\overline{\chi}^{(\yyy,x)}(\pi_v)\times\\
\times\sum_{\substack{(k_{\rho})_{\rho\in\sigma(1)}\in\ZZ^{\sigma(1)}_{\geq 0}\\\forall\sigma'\subsetneq\sigma:\yyy\in\sigma'\implies\sum_{\rho\in\sigma(1)}k_{\rho}b_{\rho}\not\in\sigma'}}\exp\bigg(-\log(q_v)\sum_{\rho\in\sigma(1)}k_{\rho}s_{\rho}\bigg)\cdot \overline{\chi}^{(\sum_{\rho\in\sigma(1)}k_{\rho}b_{\rho},0)}(\pi_v)
\end{multline*}
Note that condition $\forall\sigma'\subsetneq\sigma:\yyy\in\sigma'\implies \sum_{\rho\in\sigma(1)}k_{\rho}b_{\rho}\not\in\sigma'$
is in fact equivalent to the condition $a_{\rho}(\yyy)=0\implies k_{\rho}\geq 1$.
 Thus the inner sum is 
\begin{multline*}
\sum_{\substack{(k_{\rho})_{\rho\in\{\rho:a_{\rho}(\yyy)=0\}}\in\ZZ^{\{\rho:a_{\rho}(\yyy) =0\}}_{\geq 1}\\(k_{\rho})_{\rho\in\{\rho:a_{\rho}(\yyy)\neq  0\}}\in\ZZ^{\{\rho: a_{\rho}(\yyy)\neq 0\}}_{\geq 0}}}\exp\bigg(-\log(q_v)\sum_{\rho\in\sigma(1)}k_{\rho}s_{\rho}\bigg)\cdot \overline{\chi}^{(\sum_{\rho\in\sigma(1)}k_{\rho}b_{\rho},0)}(\pi_v).
\end{multline*}
Using Corollary~\ref{muiop}(3) and summing the geometric series, we obtain that the last  sum
\begin{multline*}
\prod_{\substack{\rho\in\sigma(1)\\ a_{\rho}(\yyy)\neq 0}}\frac{1}{(1-\exp(-\log(q_v)s_{\rho})\overline{\chi}^{(\rho)}(\pi_v))}\cdot \prod_{\substack{\rho\in\sigma(1)\\a_{\rho}(\yyy)=0}}\frac{\exp(-\log(q_v)s_{\rho})\overline{\chi}^{(\rho)}(\pi_v)}{(1-\exp(-\log(q_v)s_{\rho})\overline{\chi}^{(\rho)}(\pi_v))}.
\end{multline*}
Recall from Definition~\ref{linformxidef} that for $\rho\in\Sigma(1)$, one has that $\Xi_{\rho}(\sss)=s_{\rho}$.  Hence, the term corresponding to~$\sigma$ in the sum~(\ref{sigmsumexp}) is 
\begin{multline*}\sum_{\substack{\yyy\in\Boxx^{\rig}(\sigma)\\x\in G^D}}\exp(-\log(q_v)\Xi_{(\yyy,x)}(\sss))\overline\chi^{(\yyy,x)}(\pi_v)\times\\\times \prod_{\rho\in\sigma(1)}\frac{1}{(1-\exp(-\log(q_v)\Xi_{\rho}(\sss))\overline{\chi}^{\rho}(\pi_v))}\cdot\prod_{\substack{\rho\in\sigma(1)\\a_{\rho}(\yyy)=0}}\exp(-\log(q_v)\Xi_{\rho}(\sss))\overline{\chi}^{(\rho)}(\pi_v).
\end{multline*}
We separate the term $(0,0)$ and our sum becomes:
\begin{multline*}\frac{\prod_{\rho\in\sigma(1)}\exp(-\log(q_v)\Xi_{\rho}(\sss))\overline{\chi}^{(\rho)}(\pi_v)}{\prod_{\rho\in\sigma(1)}(1-\exp(-\log(q_v)\Xi_{\rho}(\sss))\overline{\chi}^{(\rho)}(\pi_v))}+\\\dfrac{\sum_{\YY\in\Boxx(\sigma)-\{0,0\}} \exp(-\log(q_v)\Xi_{\YY}(\sss))\overline\chi^{(\YY)}(\pi_v)\prod_{\substack{\rho\in\sigma(1)\\a_{\rho}(\YY)=0}}\exp(-\log(q_v)\Xi_{\rho}(\sss))\overline\chi^{(\rho)}(\pi_v))}{\prod_{\rho\in\sigma(1)}(1-\exp(-\log(q_v)\Xi_{\rho}(\sss))\overline{\chi}^{\rho}(\pi_v))}.
\end{multline*}
This is precisely the rational function from Subsection~\ref{defrsig} $$R_{\sigma}=\frac{1}{\prod_{\rho\in\sigma(1)}(1-X_{\rho})}\cdot\bigg(\prod_{\rho\in\sigma(1)}X_{\rho}+\sum_{\YY\in\Boxx(\sigma)-\{0,0\}}X_{\YY}\prod_{\substack{\rho\in\sigma(1)\\a_{\rho}(\YY)=0}}X_{\rho}\bigg).$$
evaluated at $\xxx=(\exp(-\log(q_v)\Xi_{\rho}(\sss))\overline\chi^{(\rho)}(\pi_v))_{\rho\in\Spijx}.$
 Hence, $$\wH(\sss,\chi)=\sum_{\sigma\in\Sigma}R_{\sigma}(\xxx)=R_{\Sigma}(\xxx), $$
as claimed. 
\end{proof}
%
\subsubsection{}  Let~$v$ be an infinite place of~$F$. The following proposition is a variant of \cite[Proposition 4.2.4]{FonctionsZ}. The proof presented there applies immediately to our situation. For every finite set~$Z$ and every $\sss\in\RR^{Z}$, we set $$||\sss||:=\max_{z\in S}(\sss|_{\{z\}}).$$ When there will be no confusion, we will refer to $\sss|_{\{z\}}$ by~$s_z$.
\begin{prop}[Chambert-Loir, Tschinkel]\label{explemphi} Let~$Z$ be a finite set and let ~$\Delta\subset\RR^Z$ be a complete simplicial fan. For a ray $\rho\in\Delta$, we let~$e_{\rho}$ to be its generator. For $\sss\in \CC^{\Delta(1)},$ we define a piecewise linear function $$\eta(\sss,-):\RR^Z\to\CC,\hspace{1cm}e_{\rho}\mapsto s_{\rho}.$$ Let $W\subset\RR^Z$ be a vector subspace. For every compact~$\mathcal K\subset\RR^{\Delta(1)}_{>0}$, there exists $C(\mathcal K)>0$ such that for every $\sss\in\Omega_{\mathcal K}$ and every $\mmm\in W$ one has that 
\begin{multline*}\bigg|\int_{\RR^Z}\exp(-\eta({\sss},\xxx)-i\langle \xxx,\mmm\rangle) d\xxx\bigg|\\\leq\frac{C(\mathcal K)}{1+||\mmm||}\sum_{\sigma\in\Delta(\# Z)}\frac{1+||\Im(\sss)|_{\sigma(1)}||}{\prod_{\rho\in\sigma(1)}(1+|\Im(s_{\rho})+\langle e_{\rho},\mmm\rangle|)}.
\end{multline*}
\end{prop}
We fix an isomorphism~$M^{\rig}_{\RR}\cong \RR^{Z}$, where~$Z$ is a finite set. Using the isomorphism, we see~$||\cdot||$ as a norm on~$M^{\rig}_{\RR}$. We deduce a norm on $\prod_{\vMFi}M^{\rig}_{\RR}$. We have a canonical identification 
$$(\prod_{\vMFi}(\log_{\TT,v})^*)^{-1}:\prod_{\vMFi}\TT(\Ov)^{\perp}[\TT(F_v)^*]\xrightarrow{\sim}\prod_{\vMFi}M^{\rig}_{\RR}.$$
For a character $$\chi\in \prod_{\vMFi} \TT(\Ov)^{\perp}[\TT(F_v)^*],$$ we deduce a meaning of~$||\chi||$.
\begin{cor}\label{pinft}
Let us denote by~$\widetilde\Sigma$ the fan $\prod_{\vMFi}\Sigma$ in $\widetilde N^{\rig}_{\RR}:=\prod_{\vMFi}N_{\RR}^{\rig}.$ For an element $\sss\in\CC^{\Sigma(1)}$, we set $\widetilde\sss:=(\sss)_{\vMFi}$ and for $\sss\in\CC^{\Sigma(1)\cup\pijx}$ we set $\widetilde\sss:=\widetilde{\sss|_{\Sigma(1)}}$. 
 For every $\chi=(\chi_v)_v \in\prod_{\vMFi}\TT(\Ov)^{\perp}[\TT(F_v)^*],$ one has that the function $$\sss\mapsto \prod_{\vMFi}\wH_v(\sss,\chi_v)=:\wH_{\infty}(\sss,\chi)$$is holomorphic in the domain~$\Omega_{\Lambda^{\circ}}$. Moreover, for every compact $\mathcal K\subset\Lambda^{\circ}$  there exists $C(\mathcal K)>0$ such that for every $\sss\in \CC^{\Sigma(1)\cup \pi_0^*(\mathcal J_0\XX)}$ with $\Re(\sss)\in\mathcal K$ 
 one has that
\begin{align*}\wH_{\infty}(\sss,\chi)\leq \frac{C(\mathcal K)}{1+||\chi||}\sum_{\widetilde\sigma\in\widetilde\Sigma(d\cdot \# M_F^{\infty})}\frac{1+||\Im(\widetilde\sss)|_{\widetilde\sigma(1)}||}{\prod_{\rho\in\widetilde\sigma(1)}(1+|\Im(s_{\rho})+\langle e_{\rho},\chi\rangle )|)}.
\end{align*}
\end{cor}
\begin{proof}
The holomorphicity is immediate. Let $\mathcal K\subset\Lambda^{\circ}$ be a compact. Let $(\mmm_v)_{\vMFi}\in\prod_{\vMFi} M^{\rig}_{\RR}$ corresponds to~$\chi$ for the above identification. 
We have that 
\begin{align*}
\wH_{\infty}(\sss,\chi)&=\prod_{\vMFi}\wH_v(\sss,\chi_v)\\&=\prod_{\vMFi}\int_{\TT(F_v)}H(\sss,\cdot)^{-1}\overline{\chi_v}\mu_v\\&=\prod_{\vMFi}\int_{N^{\rig}_{\RR}}\exp(-\phi^{\infty}_{\sss}(\xxx)-i\langle\xxx,\mmm_v\rangle) d\xxx\\
&=\int_{\widetilde N^{\rig}_{\RR}}\exp\big(-\eta(\widetilde{\sss},\widetilde\xxx)-i\langle\widetilde\xxx,(\mmm_v)_{\vMFi}\rangle\big) d\widetilde\xxx,
\end{align*}
where $d\widetilde\xxx$ is the product measure $\prod_{\vMFi}d\xxx$. Note that as $\sss\in\mathcal K,$ then $\Re(\widetilde\sss)$ belongs to the compact $\widetilde {\mathcal K}:=(\mathcal K|_{\Sigma(1)})^{\vMFi},$ where $\mathcal K|_{\Sigma(1)}:=\{f|_{\Sigma(1)}|\hspace{0,1cm}f\in \mathcal K\}.$ Lemma \ref{explemphi} gives that
\begin{align*}
\bigg|\int_{\widetilde N^{\rig}_{\RR}}\exp\big(-\eta(\widetilde{\sss},\widetilde\xxx)-i\langle\widetilde\xxx,(\mmm_v)_{\vMFi}\rangle\big)d\widetilde\xxx\bigg|\hskip-5cm&\\&\leq\frac{C(\widetilde {\mathcal K})}{||(\mmm_v)_{\vMFi}||}\sum_{\widetilde\sigma\in \widetilde\Sigma(d\cdot \#M_F^{\infty})}\frac{1+||\Im(\widetilde\sss)||_{\widetilde\sigma(1)}}{\prod_{\rho\in\widetilde\sigma(1)}(1+|s_{\rho}+\langle e_{\rho}, (\mmm_v)_{\vMFi}\rangle|)}\\
&=\frac{C(\widetilde {\mathcal K})}{||\chi||}\sum_{\widetilde\sigma\in \widetilde\Sigma(d\cdot \#M_F^{\infty})}\frac{1+||\Im(\widetilde\sss)||_{\widetilde\sigma(1)}}{\prod_{\rho\in\widetilde\sigma(1)}(1+|s_{\rho}+\langle e_{\rho}, \chi\rangle|)}.
\end{align*}
The statement is proven.
\end{proof}
We add an easily deduced fact that we will use latter.
\begin{lem}\label{hinfnnz}
For every $\sss\in\RR_{\Lambda^{\circ}}^{\Spijx}$, one has that $\wH_{\infty}(\sss,1)>0$.
\end{lem}
\begin{proof}
The character~$1$ corresponds to $0\in\prod_{\vMFi}M^{\rig}_{\RR}$ for the above identification. As in the proof of the previous corollary, we obtain that 
$$\wH_{\infty}(\sss,1)=\int_{\widetilde N_{\RR}^{\rig}}\exp(-\eta(\widetilde\sss,\widetilde\xxx))d\widetilde\xxx,$$
with the same notation. But for $\sss\in\RR_{\Lambda^{\circ}}^{\Spijx}$, the function $\widetilde\xxx\mapsto \exp(-\eta(\widetilde\sss,\widetilde\xxx))$ is strictly positive. Hence, its integral is strictly positive.
\end{proof}
\subsection{Product of transforms over finite places} The goal of this subsection is to give an estimate for the product of Fourier transforms over finite places of~$F$.
\subsubsection{} We recall basic facts on $L$-functions. For every $\vMFz,$ every $\chi_v\in\Gm(\Fv)^*$ and every $s\in\CC$, we set:
$$L_v(s,\chi_v):=\frac{1}{1-\exp(-s\log(q_v))\chi_v(\pi_v)}.$$
We write $\zeta_v(s):=L_v(s,1)$. For every $t\in\RR$, note that 
\begin{align*}L_v(s,\chi_v\chi_t)&=\frac{1}{1-\exp(-s\log(q_v))\chi_v(\pi_v)\exp(i\log(q_v)\langle\log_{\Gm,v}(\pi_v),t\rangle)}\\&=\frac{1}{1-\exp(-(s-it)\log(q_v))\chi_v(\pi_v)}\\&=L_v(s-it,\chi_v).
\end{align*}
For a character $\chi\in\Gm(\AAF)^*$ we set:
$$L(s,\chi):=\prod_{\vMFz}L_v(s,\chi_v),  $$whenever the product converges.  By \cite[VII, \S 1, Corollary 1]{basicnt}, the product converges absolutely in the domain~$\Omega_{>1}$ to a holomorphic function. We write $\zeta(s):=L(s,1).$ It is immediate from the local situation that for every $t\in\RR$, one has that $$L(s,\chi\chi_t)=L(s-it,\chi).$$  Recall that $$m:(K(\Gm)\Delta(\Gm(F)))^{\perp}[\Gm(\AAF)^*]\to \RR $$ is a retraction to the injection $$\RR\to (K(\Gm)\Delta(\Gm(F)))^{\perp}[\Gm(\AAF)^*],\hspace{1cm}t\mapsto \chi_t.$$
We recall several well known statements.
\begin{prop}[{\cite[VII, \S 7, Theorem 6]{basicnt}}] \label{lfunc}
For every $$\chi\in (K(\Gm)\Delta(\Gm(F)))^{\perp}[\Gm(\AAF)^*],$$ the function $s\mapsto L(s,\chi)$ is meromorphic in~$\CC$. Moreover, if $\chi|_{\Gm(\AAF)_1}\neq 1$, the function is entire and if $\chi|_{\Gm(\AAF)_1}=1$,  the function admits precisely one pole, which is at $s=1+im(\chi)$  and which is simple.
\end{prop}
The following claim is well known and can be derived from Rademacher's estimates \cite{Rademacher}.
\begin{prop}[Rademacher]
For every $\delta>0$, there exists $1>\epsilon>0$ and $C>0$ such that in the domain $s\in\Omega_{>1-\epsilon}$ for every $\chi\in (K(\Gm)\Delta(\Gm(F)))^{\perp}[\Gm(\AAF)]$ one has $$|L(s,\chi)|\leq C (1+|\Im(s)|)^{\delta}(1+||\chi||)^{\delta}$$ if $\chi|_{\Gm(\AAF)_1}\neq 1$ and $$\bigg|\frac{s-im(\chi)-1}{s-im(\chi)}
\cdot L(s,\chi)\bigg|\leq C(1+|\Im(s)|)^{\delta}(1+||\chi||)^{\delta},$$ if $\chi|_{\Gm(\AAF)_1}=1$. 
\end{prop}
\subsubsection{} In this paragraph we estimate the product of local Fourier transforms over finite places. For~$\sss\in \CC^{\Sigma(1)\cup\pijx}$ and $\chi\in\TT(\AAF)^*$, we define formally $$\wH_f(\sss,\chi):=\prod_{\vMFz}\wH_v(\sss,\chi_v).$$
Recall that in Definition~\ref{defoflambda}, we have defined a cone $\Lambda\subset\RR^{\Spijx}$ by asking that it is the domain where the linear forms~$\Xi_{\rho}$ from~Definition~\ref{linformxidef} are all non-negative. 
\begin{prop} \label{ljuka}
For every $\chi\in\TT(\AAF)^*$, the product defining $\sss\mapsto\wH(\sss,\chi)$ converges in the domain $\Omega_{\Lambda^{\circ}-K_X}$ to a holomorphic function. There exists~$\varepsilon_1>0$ such that for $\epsilon_1=(\varepsilon_1)_{\rho\in\Spijx}$ all the functions $\sss\mapsto\wH_f(\sss,\chi)$ extend to a meromorphic function in the domain~$\Omega_{\Lambda-K_X-\mathbf\epsilon_1}.$  More precisely, if for $\chi\in\TT(\AAF)^*$, we define a function $\gamma(-,\chi):\Omega_{\Lambda-K_X-\epsilon_1}\to\CC $ by $$\wH_f(\sss,\chi)=\gamma(\sss,\chi)\cdot\prod_{\rho\in\Spijx}L(\Xi_{\rho}(\sss),\overline{\chi}^{(\rho)}),$$ then all functions $\sss\mapsto \gamma(\sss,\chi)$ are holomorphic in the domain~$\Omega_{\Lambda-K_X-\epsilon_1}.$ Furthermore, for every compact~$\mathcal K\subset \Lambda-K_X-\epsilon_1$, one has that there exists $C(\mathcal K)>0$ such that for every $\chi\in\TT(\AAF)^*$ and all $\sss\in\Omega_{\Lambda-K_X-\epsilon}$ with $\Re(\sss)\in\mathcal K$, one has that $|\gamma(\sss,\chi)|<C(\mathcal K)$.
\end{prop}
\begin{proof}
We set 
$$W_v(\sss,\chi_v)=\bigg(\prod_{\rho\in\Spijx} L_v(\Xi_{\rho}(\sss),\overline{\chi}^{(\rho)}_v)\bigg)^{-1}.$$
We define a polynomial $P\in\ZZ[\{X_{\rho}\}_{\rho\in\Spijx}]$ by $$P=\prod_{\rho\in\Spijx}(1-X_{\rho}).$$ 
For $\xxx\in\CC^{\Sigma(1)\cup\pijx}$, given by $x_{\rho}=\exp(-\log(q_v)\cdot\Xi_{\rho}(\sss))\overline{\chi}^{(\rho)}_v(\pi_v)$ for $\rho\in\Spijx,$ 
one has that $$W_v(\sss,\chi_v)=P(\xxx).$$   Consider the polynomial $D\in\ZZ[\{X_{\rho}\}_{\rho\in\Spijx}]$ given by $$D=Q_{\Sigma}\cdot \prod_{\mathcal Y\in\pi_0^*(\mathcal J_0(\XX))}(1-X_{\mathcal Y})=R_{\Sigma}\cdot P,$$where~$R_{\Sigma}$ and~$Q_{\Sigma}$ are a rational function and a polynomial defined in Paragraph \ref{defrsig} related by $R_{\Sigma}=Q_{\Sigma}\cdot\prod_{\rho\in\Sigma(1)}(1-X_{\rho})^{-1}$. It follows from Lemma \ref{valofqs}, that the free coefficient of~$D$ is~$1$ and that~$D$ does not contain monomials of degree~$1$. Lemma \ref{loctraexp} gives that $\wH_v(\sss,\chi_v)=R_{\Sigma}(\xxx)$ and we deduce $$\gamma_v(\sss,\chi):=\wH(\sss,\chi_v)\cdot W_v(\sss,\chi_v)=Q_{\Sigma}(\xxx)P_{\Sigma}(\xxx)=D(\xxx).$$ 
We now verify that terms in form $c\exp(-s\log(q_v))$, with $c,s\in\CC$, appearing in $D(\xxx)$ have $\Re(s)>1$ if~$\varepsilon$ is sufficiently small. Indeed, as the degree of a non-free monomial in~$D$ is at least~$2$, any such~$s$ is a sum  $s=\sum_{\rho\in\Spijx}k_{\rho}\Xi_{\rho}(\sss)$. 
with $k_{\rho}\in\ZZ_{\geq 0}$ and 
 $\sum_{\rho\in\Spijx}k_\rho\geq 2$. As the coefficients of every linear form~$\Xi_{\rho}$ are non-negative, it suffices to verify the claim when $\sum_{\rho\in\Spijx}k_{\rho}=2$, what we do next. By the property that $\Xi_{\rho}(-K_X)=1$ from Lemma~\ref{propoflambda}(2) and the definition of~$\Lambda$, we deduce that for every $\rho\in\Spijx$ one has that $\Xi_{\rho}(\Lambda^{\circ}-K_X)=]1,+\infty[$. Moreover, for certain $\varepsilon_1>0$, one will have that $\Re(\Xi_{\rho}(\sss)+\Xi_{\rho'}(\sss))>1$ whenever $\sss\in\Lambda^{\circ}-K_X-(\varepsilon_1)_{\rho\in\Spijx}$. 
For~$\epsilon_1=(\varepsilon_1)_{\rho\in\Spijx}$, it follows that for every~$\chi$ the product defining $\gamma(\sss,\chi)$ converges absolutely to a holomorphic function in the domain~$\Omega_{\Lambda^{\circ}-K_X-\epsilon_1}$ and that for every compact~$\mathcal K\subset {\Lambda^{\circ}-K_X-\epsilon_1}$, there exists a constant~$C=C(\mathcal K)$ such that $|\gamma(\sss,\chi)|\leq C$ if provided $\Re(\sss)\in\mathcal K$. Now, by Proposition~\ref{lfunc} and the fact that $\Re(\Xi_{\rho}(\sss))>1$ for every $\rho\in\Spijx$ and every $\sss\in\Omega_{\Lambda^{\circ}-K_X}$, we deduce that 
\begin{align*}\wH_f(\sss,\chi)=\prod_{\vMFz}\wH_v(\sss,\chi_v)&=\prod_{\vMFz}\gamma_v(\sss,\chi)W_v(\sss,\chi_v)^{-1}\\&=\gamma(\sss,\chi)\prod_{\rho\in\Spijx}L(\Xi_{\rho}(\sss),\overline\chi^{(\rho)})
\end{align*} converges in the domain $\Omega_{\Lambda^{\circ}-K_X}$ to a holomorphic function and extends meromorphically to the domain $\Omega_{\Lambda^{\circ}-K_X-\epsilon_1}$. The proof is completed.
\end{proof}
\begin{prop}
\label{smrka} Let~$\epsilon_1=(\varepsilon_1)_{\rho\in\Spijx}$ be given by Proposition~\ref{ljuka}. For every character $\chi\in(K(\TT)\Delta(\TT(F)))^{\perp} [\TT(\AAF)^*]$, one has
$$\wH_f^*(\sss,\chi):=\bigg(\prod_{\substack{\rho\in\Spijx\\\chi^{(\rho)}|_{\Gm(\AAF)_1}=1}}\frac{\Xi_{\rho}(\sss)-im(\chi^{(\rho)})-1}{\Xi_{\rho}(\sss)-im(\chi^{(\rho)})}\bigg)\cdot\wH_f(\sss,\chi).$$
 is a holomorphic function of~$\sss$ in the domain $\Omega_{\Lambda-K_X-\epsilon_1}.$ 
Finally, for every $\delta>0$, there exists $\varepsilon_1>\varepsilon(\delta)>0$ such that setting $\epsilon=\epsilon(\delta)=(\varepsilon(\delta))_{\rho\in\Spijx}$ gives that for every compact $\mathcal K\subset \Lambda-K_X-\epsilon$ there exists $C=C(\mathcal K)>0$, such that $$|\wH^*_f(\sss,\chi)|\leq C\cdot (1+||\Im(\sss)||)^{\delta}(1+||\chi||)^{\delta}.$$ \label{finparfour}
\end{prop} 
\begin{proof}
By using Proposition~\ref{lfunc} and Proposition~\ref{ljuka}, we obtain that
\begin{multline*}
\sss\mapsto \wH^*_f(\sss,\chi)=\gamma(\sss,\chi)\times\\
\times\bigg(\prod_{\substack{\rho\in\Spijx\\\chi^{(\rho)}|_{\Gm(\AAF)_1}=1}}\hskip-0,3cm(\Xi_{\rho}(\sss)-im(\chi^{(\rho)})-1)\cdot L(\Xi_{\rho}(\sss),\overline{\chi}^{(\rho)})\bigg)\cdot\prod_{\substack{\rho\in\Spijx\\\chi^{(\rho)}|_{\Gm(\AAF)_1}\neq 1}}L(\Xi_{\rho}(\sss),\overline{\chi}^{(\rho)})
\end{multline*}
 is holomorphic in the domain $\Omega_{\Lambda^{\circ}-K_X-\epsilon_1}$. The first claim is proven, and now we prove the second claim.  
For the $L$-function of a Hecke character~$\tau$ one has, by Rademacher's estimates, that for every $\delta>0$ there exists $\varepsilon'>0$ such that in the domain $\Re(s)>1-\varepsilon'$ one has $$|L(s,\tau)|\ll (1+|\Im(s)|)^{\frac{\delta}{\#\Sigma(1)+\#\pijx}}(1+||\tau||)^{\frac{\delta}{\#\Sigma(1)+\#\pijx}}$$if $\tau|_{\Gm(\AAF)_1}\neq 1$ and $$\bigg|\bigg(\frac{s+im(\tau)-1}{s+im(\tau)}\bigg)L(s,\tau)\bigg|\ll (1+|\Im(s)|)^{\frac{\delta}{\#\Sigma(1)+\#\pijx}}(1+||\tau||)^{\frac{\delta}{\#\Sigma(1)+\#\pijx}}$$  if $\tau|_{\Gm(\AAF)_1}=1$. Set $\varepsilon=\min(\varepsilon_1,\varepsilon')$ and $\epsilon=(\varepsilon)_{\rho\in\Spijx}$. Now, for a compact $\mathcal K\subset\Lambda-K_X-\epsilon$, by using this and the fact from Proposition \ref{ljuka} that $\gamma(\sss,\chi)$ is absolutely bounded in vertical strips when provided that $\Re(\sss)\in\mathcal K$, we obtain that there exists $C>0$ such that
\begin{equation*}|\wH^*_f(\sss,\chi)|\leq C \bigg(\prod_{\rho\in\Spijx}(1+|\Im(\Xi_{\rho}(s_{\rho}))|)^{\delta}(1+||\overline{\chi}^{(\rho)}||)^{\delta}
\bigg)^{\frac{1}{\#\Sigma(1)+\#\pijx}}.\end{equation*}
Now by using an evident fact that
$$\prod_{\rho\in\Spijx}(1+|\Im(\Xi_{\rho}(s_{\rho}))|)\ll (1+||\Im(\sss)||)^{\#\Sigma(1)+\#\pijx}$$
and Lemma \ref{rama}, which implies that
$$\prod_{\rho\in\Spijx}(1+||\overline{\chi}^{\rho}||)\ll (1+||\overline{\chi}||)^{\#\Sigma(1)+\#\pijx},$$
we obtain the wanted claim. 
\end{proof}
\subsection{Global transform}We now study the global transform. For~$\sss\in\CC^{\Spijx}$ and~$\chi\in\TT(\AAF)$, we define formally $$\wH(\sss,\chi):=\int_{\TT(\AAF)}H(\sss,-)\overline{\chi}\mu.$$
We now estimate the transform over all places. 
As for all~$v$ one has that $H_v(\sss,-)$ is equal to~$1$ on the subgroup~$\TT(\Ov),$ we have by \cite[Proposition 5.6(2)]{Ramakrishnan} that
$$\wH(\sss,\chi)=\prod_{\vMF}\wH_v(\sss,\chi_v)=\wH_f(\sss,\chi)\cdot \wH_{\infty}(\sss,\chi).$$
if the second quantity converges. It is immediate from Lemma \ref{ljuka} and Corollary \ref{pinft} that the functions $\sss\mapsto\wH(\sss,\chi)$ are meromorphic in the domain $\Omega_{>\Lambda^{\circ}-K_X-\epsilon}.$
\subsubsection{} In this paragraph we show that the Fourier transform at a ramified character is~$0$.
\begin{lem}\label{ramch}
Let $\chi\in \TT(\AAF)^*-(K(\TT))^{\perp}[\TT(\AAF)^*]$ and let $\sss\in\CC^{\Spijx}$. One has that 
$$\wH(\sss,\chi):=\prod_{\vMF}H_v(\sss,\chi_v)=0.$$
\end{lem}
\begin{proof}
Let~$v$ be a place such that $\chi_v\not\in (\TT(\Ov))^{\perp}[\TT(\Fv)^*]$ and let $\xxx\in\TT(\Ov)$ be such that $\chi_v(\xxx)\neq 1$. Depending wether~$v$ is finite or not, it follows from Definition~\ref{finpair} or Definition~\ref{infpair} that $$\TT(\Fv)\to\CC,\hspace{1cm}\xxx\mapsto H_v(\sss,\xxx)^{-1}$$ is $\TT(\Ov)$-invariant. By using the change of the variables formula, we deduce that
\begin{align*}
\wH_v(\sss,\chi_v)&=\int_{\TT(\Fv)}H(\sss,\yyy)^{-1}\overline{\chi_v(\yyy)}d\mu_v(\yyy)
\\
&=\int_{\TT(F_v)}H(\sss,\xxx\yyy)^{-1}\overline{\chi_v(\xxx\yyy)}d\mu_v(\xxx\yyy)\\
&=\overline{\chi_v(\xxx)}\int_{\TT(F_v)}H(\sss,\yyy)^{-1}\overline{\chi_v(\yyy)}d\mu_v(\yyy)\\
&=\overline{\chi_v(\xxx)}\wH_v(\sss,\chi_v),
\end{align*}
which implies that $\wH(\sss,\chi_v)=0$.
\end{proof}
\subsubsection{} \label{bascalh}We now combine the estimate for the finite part and the estimate at infinite places to estimate the global transform. But before we show a useful relation between heights and characters.
\begin{lem} Let $\mmm\in M^{\rig}_{\RR}$. \label{relhc}
\begin{enumerate}
\item 
Let $\xxx\in\TT(\AAF)$. One has that
$$H(\sss,\xxx)\cdot\chi_{\mmm}(\xxx)=H(\sss+i\mmm,\xxx).$$
\item Let $\chi\in\TT(\AAF)^*$ be a character. One has that $$\wH(\sss,\chi \cdot\chi_{\mmm})=\wH(\sss+i\mmm,\chi),$$whenever the quantities on both hand sides converge.
\end{enumerate}
\end{lem}
\begin{proof}
\begin{enumerate}
\item By using Lemma \ref{otrpljati}, we have that:
\begin{align*}
H(\sss,\xxx)\cdot\chi_{\mmm}(\xxx)&=
H(\sss,\xxx)\cdot\exp(i\langle\log_{\TT}(\xxx),\mmm\rangle)\\&=\prod_{\vMF}H_v(\sss,\xxx)\cdot\exp\bigg(i\cdot\langle\sum_{\vMFz}\log(q_v)\log_{\TT,v}(\xxx),\mmm\rangle +\\&\quad\quad\quad\quad\quad\quad\quad
\quad\quad+i\cdot\langle\sum_{\vMFi}[F_v:\RR]\log_{\TT,v}(\xxx),\mmm\rangle\bigg)\\
&=\prod_{\vMFz}H_v(\sss,\xxx)\exp\bigg(i\cdot\langle\log(q_v)\log_{\TT,v}(\xxx),\mmm\rangle\bigg)\times\\
&\quad\quad\times\prod_{\vMFi}H_v(\sss,\xxx) \exp\bigg(i\cdot\langle[F_v:\RR]\log_{\TT,v}(\xxx),\mmm\rangle\bigg)\\
&=\prod_{\vMFz}H_v(\sss+i\mmm,\xxx)\cdot\prod_{\vMFi}H_v(\sss+i\mmm,\xxx)\\&=H(\sss+i\mmm,\xxx).
\end{align*}
 \item By using the first part, we obtain that:
\begin{align*}
\wH(\sss,\chi \cdot\chi_{\mmm})&=\int_{\TT(\AAF)}H(\sss,\xxx)^{-1}\overline{\chi(\xxx)\chi_{\mmm}(\xxx)}d\muu(\xxx)\\
&=\int_{\TT(\AAF)}H(\sss,\xxx)^{-1}\overline{\chi(\xxx)}\chi_{\mmm}(\xxx)^{-1}d\muu(\xxx)\\
&=\int_{\TT(\AAF)}H(\sss+i\mmm,\xxx)^{-1}\overline{\chi(\xxx)}d\muu(\xxx)\\
&=\wH(\sss+i\mmm,\chi).
\end{align*}
\end{enumerate}
\end{proof}
Here is the estimate of the global transform.
\begin{prop}\label{gltrfin}
Let $\epsilon_1=(\varepsilon_1)_{\rho\in\Spijx}$ be given by Proposition~\ref{ljuka}. For every $\chi\in (K(\TT)\Delta(\TT))^{\perp}[\TT(\AAF)^*]$, one has that 
$$\wH^*(\sss,\chi)=\bigg(\prod_{\substack{\rho\in\Spijx\\\chi^{(\rho)}|_{\Gm(\AAF)_1}=1}}\frac{\Xi_{\rho}(\sss)-1-im(\chi^{(\rho)})}{\Xi_{\rho}(\sss)-im(\chi^{(\rho))}}\bigg)\cdot\wH(\sss,\chi).$$
is a holomorphic function of~$\sss$ in the domain $\Omega_{\Lambda-K_X-\epsilon}.$ Moreover, 
if for every $\delta>0$ we set $\epsilon=(\varepsilon(\delta))_{\rho\in\Spijx},$ where $\varepsilon_1>\varepsilon(\delta)>0$ is given by Proposition~\ref{smrka}, then for every compact $\mathcal K\subset\Lambda-K_X-\epsilon$, there exists $C(\mathcal K)>0$ such that for every $\chi\in (K(\TT)\Delta(\TT(F)))^{\perp}[\TT(\AAF)^*]$ one has that (with the notation as in Corollary \ref{pinft})
\begin{align*}|\wH^*(\sss,\chi)| \leq \frac{C(\mathcal K)\cdot (1+||\Im(\sss)||)^{1+\delta}}{(1+||\chi||)^{1-\delta}}\sum_{\widetilde\sigma\in\widetilde\Sigma(d\cdot\#M^{\infty}_F)}\prod_{\rho\in \widetilde\sigma (1)}\frac{1}{1+|\langle e_{\rho},\chi_{\infty}\rangle+ \Im(s_{\rho})|}
\end{align*}
whenever $\sss\in\Omega_{\mathcal K}$.
\end{prop}
\begin{proof}
The first claim is follows from Corollary \ref{pinft} and Proposition \ref{finparfour}. Combining the estimate for the finite part of the Fourier transform given in Proposition \ref{finparfour} and the estimate for the infinite part of the Fourier transform given in Corollary \ref{pinft}, we obtain that for every $\delta>0$ and every $\mathcal K\subset\Lambda-K_X-\epsilon,$ there exist constants $C(\mathcal K), C'(\mathcal K), C''(\mathcal K)>0$ such that:
\begin{align*}
|\wH^*(\sss,\chi)|&=\bigg|\prod_{\substack{\rho\in\Spijx\\\chi^{(\rho)}|_{\Gm(\AAF)_1}=1}}\frac{\Xi_{\rho}(\sss)-1-im(\chi^{(\rho)})}{\Xi_{\rho}(\sss)-im(\chi^{(\rho)})}\cdot \wH_f(\sss,\chi)\cdot\wH_{\infty}(\sss,\chi)\bigg|\\
&=C''(\mathcal K)(1+||\Im(\sss)||)^{\delta}(1+||\chi||)^{\delta}\cdot\wH_{\infty}(\sss,\chi)\\
&=\frac{C'(\mathcal K)\cdot (1+||\Im(\sss)||)^{\delta}}{(1+||\chi||)^{1-\delta}}\sum_{\widetilde\sigma\in\widetilde\Sigma(d\cdot\# M_F^{\infty})}\frac{1+||\Im(\widetilde\sss)|_{\widetilde\sigma(1)}||}{\prod_{\rho\in\widetilde\sigma(1)}(1+|\langle e_{\rho},\chi_{\infty}\rangle+\Im(s_{\rho})|)}\\
&\leq \frac{C(\mathcal K)(1+||\Im(\sss)||)^{1}}{(1+||\chi||)^{1-\delta}}\sum_{\widetilde\sigma\in\widetilde\Sigma(d\cdot \# M_F^{\infty})}\frac{1+||\Im(\sss)||}{\prod_{\rho\in\widetilde\sigma(1)}(1+|\langle e_{\rho},\chi_{\infty}\rangle+\Im(s_{\rho})|)}\\
&= \frac{C(\mathcal K)(1+||\Im(\sss)||)^{1+\delta}}{(1+||\chi||)^{1-\delta}}\sum_{\widetilde\sigma\in\widetilde\Sigma(d\cdot \# M_F^{\infty})}\frac{1}{\prod_{\rho\in\widetilde\sigma(1)}(1+|\langle e_{\rho},\chi_{\infty}\rangle+\Im(s_{\rho})|)}
\end{align*}
for every $\sss\in\Omega_{\mathcal K}$ and every $\chi\in(K(\TT)\Delta(\TT(F)))^{\perp}[\TT(\AAF)^*].$ 
The statement is proven. 
\end{proof}
Before the next proposition, let us recall that $\Xi_{\rho}(-K_X)=1$ for every $\rho\in\Spijx$, and so that $\Xi_{\rho}(-K_X+\sss)=\Xi_{\rho}(\sss)+1$. 
\begin{prop}\label{growthofg}
Let~$\epsilon_1=(\varepsilon_1)_{\rho\in\Spijx}$ be given by Proposition~\ref{ljuka}. One has that 
$$g(\sss):= \sum_{\chi_0\in\mathfrak A_{\TT}}\wH(\sss,\chi_0)$$converges absolutely and uniformly to a holomorphic function in the domain $\sss\in\Omega_{\Lambda^{\circ}-K_X-\epsilon_1}$. Moreover, 
if for every $\delta>0$ we set $\epsilon=\epsilon(\delta)=(\varepsilon(\delta))_{\rho\in\Spijx},$ where $\varepsilon_1>\varepsilon(\delta)>0$ is given by Proposition~\ref{smrka}, then
 \begin{align*}g^*(\sss):=\bigg(\prod_{\rho\in\Spijx(1)}\frac{\Xi_{\rho}(\sss)-1}{\Xi_{\rho}(\sss)}\bigg)\cdot g(\sss)
\end{align*}
extends to a holomorphic function in the domain $\Omega_{\Lambda^{\circ}-K_X-\epsilon}$ and there exist linear forms $(\ell_{i,j})_{i\in I, j\in J}\in (\RR^{\Spijx})^*$ such that for every $j\in J$ the forms $(\ell_{i,j})_{i\in I}$ form a basis when restricted to~$M^{\rig}_{\RR}$ and such that for every compact~$\mathcal K$ in the domain $\RR_{\Lambda^{\circ}-\epsilon(\delta)}$ there exists $C(\mathcal K)>0$ such that whenever $\sss\in\Omega_{\mathcal K}$ and $\mmm\in 
 M^{\rig}_{\RR}$ one has 
\begin{multline*}\bigg|\prod_{\rho\in\Spijx}\frac{\Xi_{\rho}(\sss-i\mmm)}{1+\Xi_{\rho}(\sss-i\mmm)}\cdot g(\sss-K_X+i\mmm) \bigg|\\
\leq\frac{C(\mathcal K)(1+||\Im(\sss)||)^{1+\delta}}{(1+||\mmm||)^{1-\delta}}\sum_{i\in I}\prod_{j\in J}\frac{1}{(1+|\ell_{i,j}(\Im(\sss)+\mmm)|)}.
\end{multline*}
\end{prop}
\begin{proof}
 We fix a compact $\mathcal K\subset \RR_{\Lambda^{\circ}-K_X-\epsilon_1}$.
By Corollary~\ref{funat}, whenever $\chi_0\in\mathfrak A_{\TT}$,  for every $\rho\in\Spijx$, one has
$$\chi_0^{(\rho)}|_{\Gm(\AAF)_1}=1\implies m(\chi_0^{(\rho)})=0.$$ Whenever $\sss\in\Omega_{\mathcal K}$, we have that all the functions $\sss\mapsto\frac{\Xi_{\rho}(\sss)-1}{\Xi_{\rho}(\sss)}$ are bounded. Thus by Proposition~\ref{gltrfin}, for $\sss\in\Omega_{\mathcal K}$, we have that \begin{align*}\bigg(\prod_{\rho\in\Spijx}\frac{\Xi_{\rho}(\sss)-1}{\Xi_{\rho}(\sss)}\bigg)\cdot\wH(\sss,\chi_0)\hskip-4cm&\\&=\bigg(\prod_{\substack{\rho\in\Spijx\\\chi^{(\rho)}|_{\Gm(\AAF)_1}\neq 1}}\frac{\Xi_{\rho}(\sss)-1}{\Xi_{\rho}(\sss)}\bigg)\cdot\wH^*(\sss,\chi_0)\\
&\leq  \frac{C_1\cdot (1+||\Im(\sss)||)^{1+\delta}}{(1+||\chi_0||)^{1-\delta}}\sum_{\widetilde\sigma\in\widetilde\Sigma(d\cdot\#M^{\infty}_F)}\prod_{\rho\in \widetilde\sigma (1)}\frac{1}{1+|\langle e_{\rho},(\chi_0)_{\infty}\rangle+ \Im(s_{\rho})|}
\end{align*}
for certain $C_1>0$, independent of $\chi_0\in\mathfrak A_{\TT}$.  Hence
\begin{align*}
|g(\sss)|&\leq \sum_{\chi_0\in\mathfrak A_{\TT}}\frac{C_1\cdot (1+||\Im(\sss)||)^{1+\delta}}{(1+||\chi_0||)^{1-\delta}}\sum_{\widetilde\sigma\in\widetilde\Sigma(d\cdot \# M_F^{\infty})}\prod_{\rho\in\widetilde\sigma(1)}\frac{1}{1+\left|\langle e_{\rho},(\chi_0)_{\infty}\rangle+\Im(s_{\rho})\right|}
\end{align*}
for every $\sss\in\Omega_{\mathcal K}$.
Recall from Equation~(\ref{atagmabl}) that $\mathfrak A_{\TT}=\mathfrak A_T\times \mathfrak A_{BG}$ and that $\mathfrak A_{BG}$ is finite.  For a character $\chi'\in \mathfrak A_{BG}=\mathfrak A_{BG}\times\{1\}\subset\mathfrak A_{\TT}$, we have that $$(\chi_0\chi')_{\infty}=(\chi_0)_{\infty}$$ and, hence, that $$||\chi_0||=||\chi_0\chi'||.$$ We deduce that for $C_2=\#\mathfrak A_{BG}\cdot C_1$, one has that $$|g(\sss)|\leq\sum_{\chi_0\in\mathfrak A_T}\frac{C_2\cdot (1+||\Im(\sss)||)^{1+\delta}}{(1+||\chi_0||)^{1-\delta}}\sum_{\widetilde\sigma\in\widetilde\Sigma(d\cdot \# M_F^{\infty})}\prod_{\rho\in\widetilde\sigma(1)}\frac{1}{1+\left|\langle e_{\rho},(\chi_0)_{\infty}\rangle+\Im(s_{\rho})\right|}.$$
Let us establish a uniform convergence of the sum $$\sum_{\chi_0\in\mathfrak A_T}\frac{1}{(1+||\chi_0||)^{1-\delta}}\sum_{\widetilde\sigma\in\widetilde\Sigma(d\cdot (r_1+r_2))}\prod_{\rho\in\widetilde\sigma(1)}\frac{1}{1+\left|\langle e_{\rho},(\chi_0)_{\infty}\rangle+\Im(s_{\rho})\right|}.$$
Let~$\mathfrak A_{T,\infty}$ be the image of $\mathfrak A_T$ for the homomorhism $\theta\mapsto\theta_{\infty}\in  M^{\rig}_{\RR,\infty}.$ By \cite[Lemma 4.52(1)]{Bourqui}, the later homomorphism if of finite kernel, hence it suffices for each $\widetilde\sigma\in\widetilde\Sigma(d\cdot \# M_F^{\infty})$ to establish a uniform convergence of the sum $$\sum_{\yyy\in\mathfrak A_{T,\infty}}\frac{1}{(1+||\yyy||)^{1-\delta}}\frac{1}{\prod_{\rho\in\widetilde\sigma(1)}\big(1+\left|\langle e_{\rho},\yyy\rangle+\Im(s_{\rho})\right|\big)}.$$
By \cite[Lemma 4.52(3)]{Bourqui}, one has that $\mathfrak A_{T,\infty}$ is a lattice of a subspace of~$M^{\rig}_{\RR,\infty}$. We can thus bound the last sum by the integral $$\int_{\mathfrak A_{T,\infty}\otimes\RR}\frac{1}{(1+||\yyy||)^{1-\delta}}\frac{d\yyy}{\prod_{\rho\in\widetilde\sigma(1)}(1+|\langle e_{\rho},\yyy\rangle+\Im(s_{\rho})|)} $$ where $d\yyy$ is the Lebesgue measure on $\mathfrak A_{T,\infty}\otimes\RR$ normalized by~$\mathfrak A_{T,\infty}$. We now apply \cite[Proposition B3]{FonctionsZ} for the vector space $M^{\rig}_{\RR,\infty},$  with the respect to the complementary subspaces $\mathfrak A_{T,\infty}\otimes\RR$ and $M^{\rig}_{\RR}$ (they are complementary by \cite[Lemma 4.52(2)]{Bourqui}), with $v_1=0$ and $v_2=(\Im(s_{\rho}))_{\rho\in\widetilde\sigma(1)}$ and  the dual basis of $M^{\rig}_{\RR,\infty}$ given by $(\langle e_{\rho},\cdot\rangle)_{\rho\in\widetilde\sigma(1)}.$ We deduce that the series defining~$g(\sss)$ converges uniformly on the compacts in the domain $\Omega_{\Lambda^{\circ}-K_X}$ that $\sss\mapsto g(\sss)$ is holomorphic in the domain. The first claim is proven.

Let us prove the second claim. For $\rho\in\Spijx$, note that $$(\chi_0\chi_{\mmm})^{(\rho)}|_{\Gm(\AAF)_1}\iff \chi_0^{(\rho)}|_{\Gm(\AAF)_1}= 1. $$ Furthermore, for $\rho\in\Spijx$ if $\chi_0^{(\rho)}|_{\Gm(\AAF)_1}=1,$ then $m(\chi_0^{(\rho)})=0$ by  and Lemma~\ref{obrikmmm}: $$m((\chi_0\chi_{\mmm})^{(\rho)})=m(\chi_0^{(\rho)})+m(\chi^{(\rho)}_{\mmm})=0+\Xi_{\rho}(
i\mmm)=\Xi_{\rho}(i\mmm).$$ Let $\delta>0$ and let $\epsilon=\epsilon(\delta)$.  We have that 
\begin{align*}
|g^*(\sss-K_X+i\mmm)|\hskip-1cm&\\
&=\bigg|\prod_{\rho\in\Spijx}\frac{\Xi_{\rho}(\sss-i\mmm)}{1+\Xi_{\rho}(\sss-i\mmm)} \cdot g(\sss-K_X+i\mmm)\bigg|\\
&=\bigg|\prod_{\rho\in\Spijx}\frac{\Xi_{\rho}(\sss-i\mmm)}{1+\Xi_{\rho}(\sss-i\mmm)} \cdot\sum_{\chi_0\in\mathfrak A_{\TT}}\wH(\sss-K_X+i\mmm,\chi_0)\bigg|
\\&\leq\sum_{\chi_0\in\mathfrak A_{\TT}}\bigg|\prod_{\rho\in\Spijx}\frac{\Xi_{\rho}(\sss-i\mmm)}{1+\Xi_{\rho}(\sss-i\mmm)} \cdot\wH(\sss-K_X+i\mmm,\chi_0)\bigg|\\
&\leq\sum_{\chi_0\in\mathfrak A_{\TT}}\bigg|\prod_{\rho\in\Spijx}\frac{\Xi_{\rho}(\sss-i\mmm)}{1+\Xi_{\rho}(\sss-i\mmm)} \cdot\wH(\sss-K_X,\chi_0\chi_{\mmm})\bigg|\\
&\leq \sum_{\chi_0\in\mathfrak A_{\TT}}\bigg|\prod_{\substack{\rho\in\Spijx\\\chi_0^{(\rho)}|_{\Gm(\AAF)_1}= 1}}\frac{\Xi_{\rho}(\sss-i\mmm)}{1+\Xi_{\rho}(\sss-i\mmm)}\cdot\wH(\sss-K_X,\chi_0\chi_{\mmm})\times\\
&\quad\quad\quad\quad\quad\quad\quad\quad\quad\quad\quad\quad\times  \prod_{\substack{\rho\in\Spijx\\\chi_0^{(\rho)}|_{\Gm(\AAF)_1}\neq 1}}\frac{\Xi_{\rho}(\sss-i\mmm)}{1+\Xi_{\rho}(\sss-i\mmm)} \bigg|.
\end{align*}
  Now, Proposition~\ref{gltrfin} gives that
\begin{align*}
|\wH^*(\sss-K_X,\chi_0\chi_{\mmm})|\hskip-3cm&\\
&\leq \frac{C_3(1+||\Im(\sss)+K_X||)^{1+\delta}}{(1+||\chi_0\chi_{\mmm}||)^{1-\delta}}\sum_{\widetilde\sigma\in\widetilde\Sigma (d\cdot\#M_F^{\infty})}\prod_{\rho\in\widetilde\sigma(1)}\frac{1}{1+|\langle e_{\rho},(\chi_0\chi_{\mmm})_{\infty}\rangle+\Im(s_{\rho}-1)|}\\
&= \frac{C_3(1+||\Im(\sss)||)^{1+\delta}}{(1+||\chi_0\chi_{\mmm}||)^{1-\delta}}\sum_{\widetilde\sigma\in\widetilde\Sigma (d\cdot\#M_F^{\infty})}\prod_{\rho\in\widetilde\sigma(1)}\frac{1}{1+|\langle e_{\rho},(\chi_0\chi_{\mmm})_{\infty}\rangle+\Im(s_{\rho})|}.
\end{align*}
By using that the functions $\frac{\Xi_{\rho}(\sss)}{1+\Xi_{\rho}(\sss)}$ and $\frac{\phi_{\sss}(\YY)}{1+\phi_{\sss}(\YY)}$ are bounded on $\Omega_{\mathcal K}$, we deduce that for some $C_4>0$, one has that
\begin{multline*}|g^*(\sss-K_X+i\mmm)|\\\leq \sum_{\chi_0\in\mathfrak A_{\TT}} \frac{C_4(1+||\Im(\sss)||)^{1+\delta}}{(1+||\chi_0\chi_{\mmm}||)^{1-\delta}}\sum_{\widetilde\sigma\in\widetilde\Sigma (d\cdot\#M_F^{\infty})}\prod_{\rho\in\widetilde\sigma(1)}\frac{1}{1+|\langle e_{\rho},(\chi_0\chi_{\mmm})_{\infty}\rangle+\Im(s_{\rho})|}. 
\end{multline*}
By the same arguing as before, we deduce that for some $C_5>0$, one has 
\begin{multline*}|g^*(\sss+K_X+i\mmm)|\\\leq \sum_{\yyy\in\mathfrak A_{\TT,\infty}} \frac{C_5(1+||\Im(\sss)||)^{1+\delta}}{(1+||\yyy+\widetilde\mmm||)^{1-\delta}}\sum_{\widetilde\sigma\in\widetilde\Sigma (d\cdot\#M_F^{\infty})}\prod_{\rho\in\widetilde\sigma(1)}\frac{1}{1+|\langle e_{\rho},\yyy+\widetilde{\mmm}\rangle+\Im(s_{\rho})|}. 
\end{multline*}
When~$\mmm$ is fixed, we have that $$(1+||\yyy+\widetilde\mmm||) \asymp (1+||\yyy||)\cdot (1+||\widetilde\mmm||)=(1+||\yyy||)\cdot (1+||\mmm||).$$ 
We also observe that:
\begin{align*}
1+|\langle e_{\rho},\yyy+\widetilde\mmm\rangle+\Im(s_{\rho})|=1+|\langle e_{\rho},\yyy+\widetilde\mmm+\widetilde\sss\rangle|
\end{align*}
It follows that 
\begin{align*}
|g^*(\sss+K_X+i\mmm)|\hskip-2cm&\\&\leq \frac{C_6(1+||\Im(\sss)||)^{1+\delta}}{(1+||\mmm||)^{1-\delta}}\times\\&\quad\quad\times\sum_{\yyy\in\mathfrak A_{\TT,\infty}} \frac{1}{(1+||\yyy||)^{1-\delta}}\sum_{\widetilde\sigma\in\widetilde\Sigma (d\cdot\#M_F^{\infty})}\prod_{\rho\in\widetilde\sigma(1)}\frac{1}{1+|\langle e_{\rho},\yyy+\widetilde{\mmm}+\Im(\widetilde\sss)\rangle|}\\
&\leq \frac{C_6(1+||\Im(\sss)||)^{1+\delta}}{(1+||\mmm||)^{1-\delta}}\times\\&\quad\quad\times\int_{\mathfrak A_{T,\infty}}\frac{d\yyy}{(1+||\yyy||)^{1-\delta}}\sum_{\widetilde\sigma\in\widetilde\Sigma(d\cdot\# M_F^{\infty})}\prod_{\rho\in\widetilde\sigma(1)}\frac{1}{1+|\langle e_{\rho},\yyy+\widetilde{\mmm}+\Im(\widetilde\sss)\rangle|}.
\end{align*}
We now apply \cite[Proposition B3]{FonctionsZ} for the vector space $M^{\rig}_{\RR,\infty},$  with the respect to the complementary subspaces $\mathfrak A_{T,\infty}\otimes\RR$ and $M^{\rig}_{\RR}$ (they are complementary by \cite[Lemma 4.52(2)]{Bourqui}), with $v_1=0$ and $v_2=(m_{\rho}+\Im(s_{\rho}))_{\rho\in\widetilde\sigma(1)}$ and  the dual basis of $M^{\rig}_{\RR,\infty}$ given by $(\langle e_{\rho},\cdot\rangle)_{\rho\in\widetilde\sigma(1)}.$ The claim follows.
\end{proof}
We now prove that~$g^*(-K_X)$ is strictly positive. 
Set $$B:=\{\chi_0\in\mathfrak A_{\TT}|\hspace{0,1cm}\forall\rho\in\Spijx:\chi_0^{(\rho)}=1\}\subset\mathfrak A_{\TT}.$$ We note that for every $\chi_0\in B$ and every $\yyy\in D=\{\sum_{\rho\in\Sigma(1)}k_{\rho}b_{\rho}|\hspace{0,1cm}k_{\rho}\in\ZZ\},$ one has that $\chi_0^{(\yyy)}=1$. From Lemma~\ref{finofbchi} we deduce that~$B$ is finite. The subgroup~$B^{\perp}[\TT(\AAF)]\subset\TT(\AAF)$ is thus of finite index. 
\begin{prop}\label{gkxz}
One has that \begin{align*}g^*(-K_X)\hskip-1.6cm&\\&=\lim_{\sss\to (0)_{\rho\in\Spijx}}\prod_{\rho\in\Spijx}\frac{\Xi_{\rho}(\sss)}{\Xi_{\rho}(\sss)+1}\cdot(\#B)\cdot\int_{B^{\perp}[\TT(\AAF)]}H(\sss-K_X,-)^{-1}\mu\\&> 0.
\end{align*}
\end{prop}
\begin{proof}
For every~$\chi_0\in\mathfrak A_{\TT}$ for which there exists $\rho\in\Spijx$ such that $\chi_0^{(\rho)}\neq 1$, one has by Proposition~\ref{growthofg} $$\lim_{\sss\to(0)_{\rho\in\Spijx}}\bigg(\prod_{\rho\in\Spijx}\frac{\Xi_{\rho}(\sss)}{1+\Xi_{\rho}(\sss)}\cdot\wH(\sss-K_X,\chi_0)\bigg)=0.$$Thus, \begin{multline*}g^*(-K_X)\\=\lim_{\sss\to (0)_{\rho\in\Spijx}}\bigg(\prod_{\rho\in\Spijx}\frac{\Xi_{\rho}(\sss)}{1+\Xi_{\rho}(\sss)}\cdot\sum_{\substack{\chi_0\in B}}\wH(-K_X,\chi_0) \bigg).
\end{multline*}
The exact sequence:
$$1\to (B, \mu_B^{\#})\to (\TT(\AAF)^*,\mu^*)\to (\TT(\AAF)/B, \mu^*/\mu_B^{\#})\to 1$$ induces the exact sequence of the dual groups endowed with the dual Haar measures:
$$1\to (B^{\perp}[\TT(\AAF)], (\mu^*/\mu^{\#}_B)^*) \to (\TT(\AAF), \mu) \to (B^*, (\#B)^{-1}\mu^{\proba}_{B^*})\to 1$$
where $\mu^{\proba}_{B^*}$ denotes the probability Haar measure on~$B^*$. We deduce in particular that $(\mu^*/\mu^{\#}_B)^*=(\# B)\cdot\mu|_{B^{\perp}[\TT(\AAF)]}.$
We apply the Poisson formula aplied to the function $\wH(\sss,-)$ with the respect to the inclusion $B\subset\TT(\AAF)^*$ of locally compact abelian groups endowed with the Haar measures~$\mu^{\#}_B$ and~$\mu^{*}$, respectively.  Note that $\wH(\sss,-)$ is obviously $\mu^{\#}_B$-integrable whenever $\wH(\sss,-)$ converges, hence at least for $\sss\in\Omega_{\Lambda^{\circ}-K_X}$. Moreover, in the same domain for~$\sss$, by the Fourier inversion theorem, the function $\wH(\sss,-)$ is also $\mu^*$-integrable and for~$\xxx\in\TT(\AAF)=(\TT(\AAF)^*)^*$, one has $$H(\sss,\xxx)^{-1}=\int_{\TT(\AAF)^*}\wH(\sss,\chi)\overline{\chi(\xxx)}d\mu^*(\chi).$$ The Poisson formula gives that in the domain $\sss\in\Omega_{\Lambda^{\circ}-K_X}$:
$$\sum_{\chi_0\in B}\wH(\sss,\chi_0)= \#B\cdot\int_{B^{\perp}[\TT(\AAF)]}H(\sss,-)^{-1}\mu.$$
It remains to prove that
 \begin{multline*}\lim_{\sss\to (0)_{\rho\in\Spijx}}\prod_{\rho\in\Spijx}\frac{\Xi_{\rho}(\sss)}{\Xi_{\rho}(\sss)+1}\cdot g(\sss-K_X)\\=\lim_{\sss\to (0)_{\rho\in\Spijx}}\prod_{\rho\in\Spijx}\frac{\Xi_{\rho}(\sss)}{\Xi_{\rho}(\sss)+1}\cdot(\#B)\cdot\int_{B^{\perp}[\TT(\AAF)]}H(\sss-K_X,-)^{-1}\mu. 
\end{multline*}
is strictly positive.  Let $\yyy_1\doots \yyy_{\# B}\in\TT(\AAF)$ be a system of representatives modulo $B^{\perp}[\TT(\AAF)]$. It follows from Lemma~\ref{estheighty}, that there exists $C>0$ such that for every~$\yyy_i$ and every $\xxx\in\TT(\AAF)$, one has that  $$H(\sss-K_X,\yyy\xxx)^{-1}\geq \exp(||\sss||)^{-1}\cdot C\cdot H(\sss-K_X,\xxx)$$ whenever $\sss\in \RR_{>0}^{\Spijx}$ and thus for some $C'>0$, one has $$H(\sss-K_X,\yyy\xxx)^{-1}\geq C' H(\sss-K_X,\xxx)^{-1}$$ whenever $\sss\in(\RR_{>0}^{\Spijx}\cap \RR^{\Spijx}_{<\frac12})$. In particular for every $i=1\doots \# B$, one has:
\begin{align*}\int_{\yyy_iB^{\perp}[\TT(\AAF)]}H(\sss-K_X,-)^{-1}\mu&=\int_{B^{\perp}[\TT(\AAF)]}H(\sss-K_X,\yyy_i\xxx)^{-1}d\mu(\xxx) \\&\geq C'\int_{B^{\perp}[\TT(\AAF)]}H(\sss-K_X,-)^{-1}\mu 
\end{align*}
 We have that
\begin{align*}
\lim_{\sss\to (0)_{\rho\in\Spijx}}\prod_{\rho\in\Spijx}\frac{\Xi_{\rho}(\sss)}{\Xi_{\rho}(\sss)+1}\cdot(\#B)\cdot\int_{B^{\perp}[\TT(\AAF)]}H(\sss-K_X,-)^{-1}\mu\hskip-12,5cm&\\
&\geq\lim_{\sss\to (0)_{\rho\in\Spijx}}C'\prod_{\rho\in\Spijx}\frac{ \Xi_{\rho}(\sss)}{\Xi_{\rho}(\sss)+1}\cdot\sum_{i=1}^{\# B}\int_{\yyy_iB^{\perp}[\TT(\AAF)]}H(\sss-K_X,-)^{-1}\mu\\
&\geq  \lim_{\sss\to (0)_{\rho\in\Spijx}}C'\prod_{\rho\in\Spijx}\frac{ \Xi_{\rho}(\sss)}{\Xi_{\rho}(\sss)+1}\cdot\int_{\TT(\AAF)}H(\sss-K_X,-)^{-1}\mu.\\
&=\lim_{\sss\to (0)_{\rho\in\Spijx}}C'\prod_{\rho\in\Spijx}\frac{ \Xi_{\rho}(\sss)}{\Xi_{\rho}(\sss)+1}\cdot \wH(\sss-K_X,1).
\end{align*}
By Proposition~\ref{ljuka}(2), one has \begin{align*}
\lim_{\sss\to (0)_{\rho\in\Spijx}}\prod_{\rho\in\Spijx}\frac{ \Xi_{\rho}(\sss)}{\Xi_{\rho}(\sss)+1}\cdot\wH(\sss-K_X,1)\hskip-9cm&\\&=\lim_{\sss\to (0)_{\rho\in\Spijx}}\bigg(\wH_f(\sss-K_X,1)\cdot\prod_{\rho\in\Spijx}\frac{ \Xi_{\rho}(\sss)}{\Xi_{\rho}(\sss)+1}\bigg)\cdot\wH_{\infty}(\sss-K_X,1)\\&=\gamma(-K_X,1)\wH_{\infty}(-K_X,1)\lim_{\sss\to (0)_{\rho\in\Spijx}}\prod_{\rho\in\Spijx}\frac{\Xi_{\rho}(\sss)}{\Xi_{\rho}(\sss)+1}\cdot\zeta_F(\Xi_{\rho}(\sss-K_X)),
\end{align*}where 
the function $\gamma(-,1)$ is defined in {\it id}.\ by a converging Euler product in the domain $\Omega_{\Lambda^{\circ}-K_X-\epsilon}$, with $\epsilon$ as in {\it id}.\ In particular $\gamma(-K_X,1)>0$.  Lemma~\ref{hinfnnz} gives that $$\wH_{\infty}(-K_X,1)>0.$$ 
 One has that 
\begin{align*}
\lim_{\sss\to (0)_{\rho\in\Spijx}}\prod_{\rho\in\Spijx}\frac{\Xi_{\rho}(\sss)}{\Xi_{\rho}(\sss)+1}\zeta_F(\Xi_{\rho}(\sss-K_X))\hskip-5cm&\\&=\lim_{\sss\to (0)_{\rho\in\Spijx}}\prod_{\rho\in\Spijx}\frac{\Xi_{\rho}(\sss)}{\Xi_{\rho}(\sss)+1}\zeta_F(\Xi_{\rho}(\sss-K_X))\\&=\prod_{\rho\in\Spijx}\Res(\zeta_F,1)\\&>0,\\
\end{align*} where $\Res(\zeta_F,1)$ denotes the residue of the Dedekind zeta function~$\zeta_F$ at~$1$. We deduce:
\begin{align*}
g^*(-K_X)\hskip-2cm&\\&=\lim_{\sss\to (0)_{\rho\in\Spijx}}\frac{\Xi_{\rho}(\sss)}{\Xi_{\rho}(\sss)+1}\cdot g(-K_X+\sss)\\
&=\lim_{\sss\to (0)_{\rho\in\Spijx}}\prod_{\rho\in\Spijx}\frac{\Xi_{\rho}(\sss)}{\Xi_{\rho}(\sss)+1}\cdot(\#B)\cdot\int_{B^{\perp}[\TT(\AAF)]}H(\sss-K_X,-)^{-1}\mu\\
&>0.
\end{align*} 
as claimed.
\end{proof}
\subsection{Final calculations}\label{Finalcalculations}
We are ready to establish the wanted meromorphic behavior of the height zeta function. 

\subsubsection{} The real orbifold Néron-Severi group identifies with the quotient $$NS^1_{\orb,\RR}=\RR^{\Spijx}/M^{\rig}_{\RR}.$$
We set $b=\rk(NS^1_{\orb})$. Recall that we have denoted by $p:\RR^{\Spijx}\to NS^{1}_{\orb,\RR}$ the quotient map. We also denote by~$p$ the induced map: $$p=p\otimes\CC:\CC^{\Spijx}\to (\RR^{\Spijx}/M^{\rig}_{\RR})\otimes\CC.$$
We set $\Lambda_{\orb}=p(\Lambda).$
For every $\sss\in \NSc=\NSr\otimes\CC$ let us choose $\widetilde\sss\in\CC^{\Spijx}$ such that $p(\widetilde\sss)=\sss$. 
It follows from Lemma~\ref{gl:pair} that for every $\sss\in\NSc$, that one has that $$H_{\sss}(\xxx)=H(\widetilde\sss,\xxx).$$
In particular, for any $\widetilde\sss'$ for which $p(\widetilde\sss')=p(\sss)$, one has that $H(\widetilde\sss,-)=H(\widetilde\sss',-)$, hence that $\wH(\widetilde\sss,-)=\wH(\widetilde\sss',-)$ and that $g(\widetilde\sss)=g(\widetilde\sss')$. 
We define a height zeta function.
\begin{mydef}For $\sss\in\NSc$, we define formally $$Z(\sss):=\sum_{\xxx\in\TT(F)}H_{\sss}(\xxx)^{-1}=\sum_{\xxx\in\TT(F)}H(\widetilde\sss,\xxx)^{-1}.$$
\end{mydef}
\begin{thm}\label{princip}
For $\eta\gg 1$, one has that whenever $\sss\in\Omega_{p(]\eta,+\infty[^{\Spijx})}$ that \begin{equation}\label{eqofzf}
Z(\sss)=\frac{\# \Sh^1(G)}{(\# G^D)\cdot (2\pi)^{\rk (M^{\rig})}}\int_{M^{\rig}_{\RR}}g(\widetilde{\sss}+i\mmm)d\mmm,
\end{equation}where~$\widetilde\sss$ is a lift of~$\sss$. There exists $\varepsilon'>0$ such that for $\epsilon=(\varepsilon')_{\rho\in\Spijx}$ the function $\sss\mapsto Z(\sss)$ extends to a meromorphic function in the domain $\Omega_{p(\Lambda^{\circ}-\epsilon)-p(K_X)}.$ The only possible poles of $\sss\mapsto Z(\sss)$ in the domain $\Omega_{p(\Lambda^{\circ}-\epsilon)-p(K_X)}$ are along the faces of~$p(\Lambda)-p(K_X)$ and they are at worst simple. Furthermore, there exists~$\alpha>0$ such that for every compact $\mathcal K\subset\Lambda-p(K_X)-\epsilon$ there exists $C>0$ such that $$\bigg|\frac{Z(\sss)}{\mathsf X_{p(\Lambda)}(\sss+p(K_X))}|\leq C (1+||\Im(\sss)||)^{\alpha},$$
whenever $\sss\in\Omega_{\mathcal K}$. Finally, for every $\sss_0\in p(\Lambda^{\circ}),$ one has that 
\begin{equation*}
\lim_{z\to 0^+}\frac{Z(z\sss_0)}{\mathsf X_{p(\Lambda)}(z\sss_0+p(K_X))}=\frac{\#\Sh^1(G)\cdot g^*(-K_X)}{\# G^D}>0.
\end{equation*}
\end{thm}
\begin{proof}
We will apply the Poisson formula for the function $H(\widetilde\sss,-)^{-1}$ and for the inclusion of the closed subgroup $$\Delta(\TT(F))\subset\TT(\AAF)$$ with respect to Haar measures $(\# \Sh^1(G))(\# G(F))^{-1}\mu_{\Delta(\TT(F))}^{\#}$ and~$\mu$. Recall the notation $$\mu/\Delta= \mu/({\# \Sh^1(G)}\cdot{\# G(F)}^{-1}\cdot\mu_{\Delta(\TT(F))}^{\#}).$$ For $\sss\in\Omega_{p(]\eta,+\infty[^{\Spijx})}$, the function $H(\widetilde\sss,-)^{-1}$ is absolutely integrable on~$\Delta(\TT(F))$ and~$\TT(\AAF)$ by Theorem~\ref{northcottproperty} and Proposition~\ref{ljuka}(1), respectively.  We obtain that \begin{align*}
Z(\sss)&=(\#\Sh^1(G))\cdot\sum_{\xxx\in\Delta(\TT(F))}H(\widetilde\sss,\xxx)^{-1}\\&=(\#\Sh^1(G))\cdot\int_{(\Delta(\TT(F)))^{\perp}[\TT(\AAF)^*]}\wH(\widetilde\sss,-) (\mu/\Delta)^*,
\end{align*}
where $\widetilde\sss\in\Omega_{]\eta,+\infty[^{\Spijx}}$ is a lift of~$\sss$.
Whenever $\chi\in (\Delta(\TT(F)))^{\perp}[\TT(\AAF)^*]$ satisfies that $\chi|_{K(\TT)}\neq 1$, Lemma~\ref{ramch} gives that$$\wH(\widetilde\sss,\chi)=0.$$ Thus it suffices to integrate over the subgroup $(K(\TT)\Delta(\TT(F)))^{\perp}[\TT(\AAF)^*]$ which is open in $(\Delta(\TT(F)))^{\perp}[\TT(\AAF)^*]$ by Lemma~\ref{normmeas}(2):
\begin{equation*}Z(\sss)=\#\Sh^1(G)\cdot\int_{(K(\TT)\Delta(\TT(F)))^{\perp}[\TT(\AAF)^*]}\wH(\widetilde\sss,-)(\mu/\Delta)^*.
\end{equation*}
Lemma~\ref{normmeas}(3) gives an exact sequence of locally compact abelian groups endowed with Haar measures:
\begin{multline*}1\to (M^{\rig}_{\RR}, \frac{d\mmm}{(2\pi)^{\rk(M^{\rig})}})\xrightarrow{\mmm\mapsto\chi_{\mmm}} \big((K(\TT)\Delta(\TT(F)))^{\perp}[\TT(\AAF)^*], (\mu/\Delta)^*\big)\to\\
\to \big((K(\TT)\Delta(\TT(F)))^{\perp}[\TT(\AAF)_1^*], (\# G^D)^{-1}\cdot\mu^{\#}_{\Delta(\TT(F))K(\TT))^{\perp}[\TT(\AAF)_1^*]}\big) \to 1,
\end{multline*}
where~$d\mmm$ is the Lebesgue measure on~$M^{\rig}_{\RR}$ normalized by the lattice~$M^{\rig}$. The subgroup~$\mathfrak A_{\TT}\subset (K(\TT)\Delta(\TT(F)))^{\perp}[\TT(\AAF)^*]$ satisfies that the restriction of the map to $(K(\TT)\Delta(\TT(F)))^{\perp}[\TT(\AAF)_1^*]$ is a bijection: $$\mathfrak A_{\TT}\xrightarrow{\sim}\big(K(\TT)\Delta(\TT(F))\big)^{\perp}[\TT(\AAF)^*_1].$$ The set~$\mathfrak A_{\TT}$ is obviously discrete and we endow it with the measure for which the bijection is measure preserving, that is with the measure $(\#G^D)^{-1}\mu_{\mathfrak A_{\TT}}^{\#}$. We deduce that 
$$Z(\sss)=\frac{\#\Sh^1(G)}{(\#G^D)(2\pi)^{\rk (M^{\rig})}}\int_{M^{\rig}_{\RR}}\sum_{\chi_0\in\mathfrak A_{\TT}}\wH(\widetilde\sss,\chi_0\chi_{\mmm})d\mmm$$ and using Lemma~\ref{relhc} which gives $$\wH(\widetilde\sss,\chi_0\chi_{\mmm})=\wH(\widetilde\sss+i\mmm,\chi_0),$$ that:
\begin{align*}Z(\sss)&=\frac{\#\Sh^1(G)}{(\#G^D)(2\pi)^{\rk (M^{\rig})}}\int_{M^{\rig}_{\RR}}\sum_{\chi_0\in\mathfrak A_{\TT}}\wH(\widetilde\sss+i\mmm,\chi_0)d\mmm\\&=\frac{\#\Sh^1(G)}{(\#G^D)(2\pi)^{\rk (M^{\rig})}}\int_{M^{\rig}_{\RR}}g(\widetilde\sss+i\mmm)d\mmm.
\end{align*}
We have proven Equality~(\ref{eqofzf}). Using it and applying \cite[Lemma 3.1.6]{FonctionsZ} for the {\it translated} function $$\Omega_{\Lambda^{\circ}}\to\CC,\hspace{1cm}\zzz\mapsto g(\zzz-K_X)$$ gives that $\sss\mapsto Z(\sss-p(K_X))$ converges absolutely on compacts to a holomorphic function in the domain~$\Omega_{p(\Lambda^{\circ})}$. This proves the first claim after Equality~(\ref{eqofzf}). 
To get the other claims, we will apply \cite[Theorem 3.1.14]{FonctionsZ}. With the notation as in \textit{id}.\ we set: $M=M':=M^{\rig}_{\RR}$ and $V:=\RR^{\Spijx}$, for the cone we take~$\Lambda,$ while for the function we take~$\zzz\mapsto g(\zzz-K_X)$ for which we have verified in Proposition~\ref{growthofg} that it indeed satisfies the conditions of {\it id.} 
Applying \textit{id}.\ gives that $\sss\mapsto Z(\sss-p(K_X))),$  
is holomorphic in the domain~$\Omega_{p(\Lambda^{\circ})},$  
that it extends meromorphically to the domain~$\Omega_{p(\Lambda^{\circ}-\epsilon)}$ and that its poles are precisely in the domain are faces of~$p(\Lambda)$, which are at worst simple.  
 Moreover, there exists $\alpha>0$ such that whenever $\mathcal K\subset\Omega_{p(\Lambda^{\circ}-\epsilon)} $ is a compact, one gets that for every $\sss\in\Omega_{\mathcal K}$ that $$\bigg|\frac{Z(\sss-p(K_X))}{\mathsf X_{p(\Lambda)}(\sss-p(K_X))}\bigg|\leq C\cdot (1+||\Im(\sss))||)^{\alpha}.$$
Let us now use the second part of \cite[Theorem 3.1.14]{FonctionsZ}. We have verified in  Lemma~\ref{propoflambda} that~$\Lambda$ does not contain a line. Recall that by Lemma~\ref{xconexi} one has that $$\mathsf X_{\Lambda}(\sss)=\prod_{\rho\in\Spijx}\frac{1}{\Xi_{\rho}(\sss)}.$$For every $\zzz_0\in\Lambda$, one has that \begin{align*}\lim_{z\to 0^+}\frac{g^*(-K_X)^{-1}\cdot g(z\cdot\zzz_0-K_X)}{\mathsf X_{\Lambda}(z\cdot\zzz_0)}\hskip-3cm&\\&=g^*(-K_X)^{-1}\lim_{z\to 0^+}g(z\cdot \zzz_0-K_X)\cdot\prod_{\rho\in\Spijx}\Xi_{\rho}(z\cdot\zzz_0) \\
&=g^*(-K_X)^{-1}\lim_{z\to 0^+}g(z\cdot\zzz_0-K_X)\cdot\prod_{\rho\in\Spijx}\frac{\Xi_{\rho}(z\cdot\zzz_0)}{\Xi_{\rho}(z\cdot\zzz_0)+1}\\
&=g^*(-K_X)^{-1}\cdot g^*(-K_X)\\
&=1.
\end{align*}
We set $Z':=\frac{Z\cdot \# G^D}{\#\Sh^1(G)\cdot g^*(-K_X)},$ so that $$Z'(\sss)=\frac{1}{(2\pi)^{\rk(M^{\rig})}}\int_{M^{\rig}_{\RR}}(g^*(-K_X))^{-1}\cdot g(\widetilde\sss+i\mmm-K_X)d\mmm.$$
Applying \textit{id}.\ gives that for every~$\sss_0\in p(\Lambda)^{\circ}=p(\Lambda^{\circ})$, one has that
$$\lim_{z\to 0^+}\frac{Z'(z\sss_0-p(K_X))}{\mathsf X_{p(\Lambda)}(z\sss_0)}=1,$$
i.\ e.\ that
$$\lim_{z\to 0^+}\frac{Z(z\sss_0-p(K_X))}{\mathsf X_{p(\Lambda)}(z\sss_0)}=\frac{\#\Sh^1(G)\cdot g^*(-K_X)}{\# G^D}.$$
\end{proof}
\begin{thm}
 We consider the function $$Z^{-p(K_X)}:z\mapsto Z(-zp(K_X))=\sum_{\xxx\in\TT(F)}H(-zK_X,\xxx)^{-1}.$$ The function~$Z^{-p(K_X)}$ is holomorphic in the domain~$\Omega_{]1,+\infty[},$ for some $\varepsilon>0$ it extends meromorphically to the domain $\Omega_{]1-\varepsilon, +\infty[} $ where its only pole is at~$z=1$ which is of order precisely~$b$. Moreover, there exists $\alpha>0$ such that for every compact $\mathcal K\subset ]1-\varepsilon, +\infty[$, there exists $C(\mathcal K)>0$ such that $$|Z^{-p(K_X)}(z)|\leq C(\mathcal K)(1+|\Im(z)|)^{\alpha}$$ if provided $z\in\Omega_{\mathcal K}$. One has that $$\lim_{z\to 1}(z-1)^bZ^{-p(K_X)}(z)=\mathsf X_{p(\Lambda)}(-p(K_X))\cdot \omega_H(\XX(\AAF)),$$
where we have used the notation 
\begin{multline*}
\omega_H(\XX(\AAF)):=\frac{\#\Sh^1(G)\cdot (\# B)}{\# G^D}\times \\\times \lim_{\sss\to (0)_{\rho\in\Spijx}}\prod_{\rho\in\Spijx}\frac{\Xi_{\rho}(\sss)}{\Xi_{\rho}(\sss)+1}\cdot\int_{B^{\perp}[\TT(\AAF)]}H(\sss-K_X,-)^{-1}\mu.
\end{multline*} 
\end{thm}
\begin{proof}
In Theorem~\ref{princip} we set $\sss=-zp(K_X)$, where~$z$ is a single variable. We obtain that in the domain $-zp(K_X)\in\Omega_{p(\Lambda^{\circ})-p(K_X)}$, that is, in the domain $z\in\Omega_{>1}$, the function~$Z^{-p(K_X)}=z\mapsto Z(-zp(K_X))$ is holomorphic, the function \begin{align*}z&\mapsto\frac{ Z^{-p(K_X)}(z)}{\mathsf X_{p(\Lambda)}((1-z)p(K_X))}=(z-1)^b\cdot\frac{ Z^{-p(K_X)}(z)}{\mathsf X_{p(\Lambda)}(-p(K_X))}\end{align*} extends holomorphically to the domain $\Omega_{>1-\varepsilon}$ for some $\varepsilon>0$ and that there exists $\alpha>0$ such that for every compact $\mathcal K\subset ]1-\varepsilon, +\infty[$ one has that $$\bigg|(z-1)^b\cdot\frac{ Z^{-p(K_X)}(z)}{\mathsf X_{p(\Lambda)}(-p(K_X))}\bigg|\ll (1+|\Im(z)|)^{\alpha}$$ whenever $z\in\Omega_{\mathcal K}$.  By the last claim of Theorem~\ref{princip}, one has that 
\begin{align*}\lim_{z\to 1}(z-1)^bZ^{-p(K_X)}(z)\hskip-4cm&\\&={\mathsf X_{p(\Lambda)}(-p(K_X))}\lim_{z\to 1^+}\frac{Z(-zp(K_X))}{\mathsf X_{p(\Lambda)}((1-z)p(K_X))}\\
&=\frac{\mathsf X_{p(\Lambda)}(-p(K_X))\cdot\#\Sh^1(G)\cdot g^*(-K_X)}{\# G^D},\\
&=\frac{\mathsf X_{p(\Lambda)}(-p(K_X))\cdot \#\Sh^1(G)\cdot \#B}{\# G^D} \times\\&\quad\times\lim_{\sss\to (0)_{\rho\in\Spijx}}\prod_{\rho\in\Spijx}\frac{\Xi_{\rho}(\sss)}{\Xi_{\rho}(\sss)+1}\cdot\int_{B^{\perp}[\TT(\AAF)]}H(\sss-K_X,-)^{-1}\mu,
\end{align*}
where the last equality is established in~Proposition~\ref{gkxz}.
\end{proof}
\begin{cor}\label{cor:main}We shorten $H=H_{-p(K_X)}$. For~$X>1$, we set $$N_{H}(X):=\#\{\xxx\in\TT(F)|\hspace{0,1cm} H(\xxx)\leq X\}.$$ There exists a monic polynomial~$P$ of degree~$b-1$ and $\varepsilon>0$ such that $$N_H(X)=CX P(\log(X))+O(X^{1-\varepsilon})$$when $X\to \infty$, 
where $$C=\mathsf X_{p(\Lambda)}(-p(K_X))\cdot \omega_H(\XX(\AAF)).$$
\end{cor}

We can easily deduce the main result of this paper, Theorem \ref{firstthm}, from this corollary, recalling that \(K_X\) corresponds to the anti-canonical raised line bundle \(K_{\cX,\orb}\) (Definition \ref{def:KX}) and that $b=\rk(NS^1_{\orb}) = \rho(\XX)+\#\pijx$. 
\bibliography{bibliography}
\bibliographystyle{alpha}
\end{document}